\documentclass[10pt,a4paper]{amsart}

\RequirePackage[OT1]{fontenc}
\RequirePackage{amsthm,amsmath}
\RequirePackage[numbers]{natbib}
\RequirePackage[colorlinks,citecolor=blue,urlcolor=blue]{hyperref}

\usepackage{graphicx}
\usepackage{tikz}
\usepackage{enumitem}

\newtheorem{theorem}{Theorem}
\newtheorem{lemma}[theorem]{Lemma}
\newtheorem*{lemma*}{Lemma}
\newtheorem{proposition}[theorem]{Proposition}

\newtheorem{definition}[theorem]{Definition}
\newtheorem{remark}[theorem]{Remark}

\usepackage[utf8]{inputenc}
\usepackage[english]{babel}
\usepackage{latexsym}
\usepackage{amssymb}
\usepackage{amsmath}

\newcommand{\Z}{\mathbb{Z}}
\newcommand{\R}{\mathbb{R}}
\newcommand{\C}{\mathbb{C}}
\newcommand{\E}{\mathbb{E}}

\newcommand{\Tree}{\mathcal{T}}

\newcommand{\Jcal}{\mathcal{J}}

\renewcommand{\Re}{\mathrm{Re} \,}

\def\cT{\mathcal{T}}
\def\ph{\varphi}

\DeclareMathOperator{\var}{Var}

\usepackage{fullpage}

\numberwithin{equation}{section}
\numberwithin{theorem}{section}

\title[Random Hermitian matrices and GMC]{Random Hermitian Matrices and Gaussian Multiplicative Chaos}

\makeatletter
\def\@tocline#1#2#3#4#5#6#7{\relax
  \ifnum #1>\c@tocdepth 
  \else
    \par \addpenalty\@secpenalty\addvspace{#2}%
    \begingroup \hyphenpenalty\@M
    \@ifempty{#4}{%
      \@tempdima\csname r@tocindent\number#1\endcsname\relax
    }{%
      \@tempdima#4\relax
    }%
    \parindent\z@ \leftskip#3\relax \advance\leftskip\@tempdima\relax
    \rightskip\@pnumwidth plus4em \parfillskip-\@pnumwidth
    #5\leavevmode\hskip-\@tempdima
      \ifcase #1
       \or\or \hskip 1em \or \hskip 2em \else \hskip 3em \fi%
      #6\nobreak\relax
    \dotfill\hbox to\@pnumwidth{\@tocpagenum{#7}}\par
    \nobreak
    \endgroup
  \fi}
\makeatother

\author{Nathana\"el Berestycki}
\address{{Statistical Laboratory, DPMMS, University of Cambridge. Wilberforce Rd. Cambridge CB3 0WB, United Kingdom.}}
\email{{N.Berestycki@statslab.cam.ac.uk}}

\author{Christian Webb}
\address{Department of mathematics and systems analysis, Aalto University, P.O.
Box 11000, 00076 Aalto, Finland}
\email{christian.webb@aalto.fi}

\author{Mo Dick Wong}
\address{{Statistical Laboratory, DPMMS, University of Cambridge. Wilberforce Rd. Cambridge CB3 0WB, United Kingdom.}}
\email{{mdw46@cam.ac.uk}}

\begin{document}

\begin{abstract}
We prove that when suitably normalized, small enough powers of the absolute value of the characteristic polynomial of random Hermitian matrices, drawn from one-cut regular unitary invariant ensembles, converge in law to Gaussian multiplicative chaos measures. We prove this in the so-called $L^2$-phase of multiplicative chaos. Our main tools are asymptotics of Hankel determinants with Fisher-Hartwig singularities. Using Riemann-Hilbert methods, we prove a rather general Fisher-Hartwig formula for one-cut regular unitary invariant ensembles.
\end{abstract}

\maketitle

{\hypersetup{linkcolor=black}
\tableofcontents}

\section{Introduction}

\subsection{Main result}

Log-correlated Gaussian fields, namely Gaussian random generalized functions whose covariance kernels have a logarithmic singularity on the diagonal, are known to show up in various models of modern probability and mathematical physics -- e.g. in combinatorial models describing random partitions of integers \cite{io}, random matrix theory \cite{fks,hko,rivi}, lattice models of statistical mechanics \cite{kenyon}, the construction of conformally invariant random planar curves such as stochastic Loewner evolution \cite{ajks,sheffield}, and growth models \cite{bf} just to name a few examples. A recent and fundamental development in the theory of these log-correlated fields has been that while these fields are rough objects -- distributions instead of functions -- their geometric properties can be understood to some degree. For example, one can describe the behavior of the extremal values and level sets of the fields in a suitable sense -- see e.g. \cite[Section 4 and Section 6.4]{rv1}.

A fundamental tool in describing these geometric properties of the fields is a class of random measures, which can be formally written as an exponential of the field. As these fields are distributions instead of functions, exponentiation is not an operation one can naively perform, but through a suitable limiting and normalization procedure, these random measures can be rigorously constructed and they are known as Gaussian multiplicative chaos measures. These objects were introduced by Kahane in the 1980s \cite{kahane}. For a recent review, we refer the reader to \cite{rv1} and for a concise proof of existence and uniqueness of these measures we refer to \cite{berestycki}.

A typical example of how log-correlated fields show up can be found in random matrix theory. For a large class of models of random matrix theory, the following is true: when the size of the matrix tends to infinity, the logarithm of the characteristic polynomial behaves like a log-correlated field. This is essentially equivalent to a suitable central limit theorem for the global linear statistics of the random matrix -- see \cite{fks,hko,rivi} for results concerning the GUE, Haar distributed random unitary matrices, and the complex Ginibre ensemble.

One would thus expect that the characteristic polynomial and powers of it should behave asymptotically like a multiplicative chaos measure. A related question was explored thoroughly though non-rigorously in \cite{fk,fs}. The issue here is that the construction of the multiplicative chaos measure goes through a very specific approximation of the Gaussian field and typically uses things like independence and Gaussianity very strongly. In the random matrix theory situation these are present only asymptotically. Thus the precise extent of the connection between the theory of log-correlated processes and random matrix theory is far from fully understood. For rigorous results concerning multiplicative chaos and the study of extrema of approximately Gaussian log-correlated fields in random matrix theory we refer to \cite{abb,cmn,LamOstSimm,LamPaq,pz,webb}.

In this article we establish a universality result showing that for a class of random Hermitian matrices, small enough powers of the absolute value of the characteristic polynomial can be described in terms of a Gaussian multiplicative chaos measure. More precisely, we prove the following result (for definitions of the relevant quantities, see Section \ref{sec:model}).

\begin{theorem}\label{th:main}
Let $H_N$ be a random $N\times N$ Hermitian matrix drawn from a one-cut regular, unitary invariant ensemble whose equilibrium measure is normalized to have support $[-1,1]$. Then for $\beta\in[0,\sqrt{2})$, the random measure

\begin{equation*}
\frac{|\det (H_N-x)|^\beta}{\E |\det(H_N-x)|^\beta}dx
\end{equation*}

\noindent on $(-1,1)$, converges in distribution with respect to the topology of weak convergence of measures on $(-1,1)$ to a Gaussian multiplicative chaos measure which can be formally written as $e^{\beta X(x)-\frac{\beta^2}{2}\E X(x)^2}dx$, where $X$ is a centered Gaussian field with covariance kernel

\begin{equation*}
\E X(x)X(y)=-\frac{1}{2}\log|2(x-y)|.
\end{equation*}
\end{theorem}

\noindent We note that in particular, this result holds for the Gaussian Unitary Ensemble (GUE) of random matrices, with a suitable normalization. The proof here is a generalization of that in \cite{webb} by the second author and relies on understanding the large $N$ asymptotics of quantities which can be written in the form $\E [e^{\mathrm{Tr}\  \mathcal{T}(H_N)}\prod_{j=1}^k |\det(H_N-x_j)|^{\beta_j}]$ for a suitable function $\mathcal{T}:\R\to \R$, $x_j\in (-1,1)$ and $\beta_j\geq 0$.

It is easy to see, and we will recall the relevant derivations below, that such expectations can be written in terms of Hankel determinants with Fisher--Hartwig symbols, and while such quantities (and corresponding Toeplitz determinants) have been studied in great detail \cite{cf,dik1,dik2,krasovsky}, it seems that in the generality we require for Theorem \ref{th:main}, many of the results are lacking. Thus we give a proof of such results using Riemann--Hilbert techniques;  see Proposition \ref{prop:fh} for the precise result. This settles some conjectures due to Forrester and Frankel -- see Remark \ref{rem:ff} and \cite[Conjecture 5 and Conjecture 8]{ff} for further information about their conjectures.

\subsection{Motivations and related results}

One of the main motivations for this work is establishing multiplicative chaos measures as something appearing universally when studying the global spectral behavior of random matrices. This is a new type of universality result in random matrix theory and also suggests that it should be possible to establish some of the geometric properties of log-correlated fields in the setting of random matrix theory as well. Perhaps on a more fundamental level, a further motivation for the work here is a general picture of when does the exponential of an approximation to a log-correlated field converge to a multiplicative chaos measure. Naturally we don't answer this question here, but the fact that our approach works so generally, suggests that part of this argument is something that transfers beyond random matrix theory to general models where one expects multiplicative chaos measures to play a role.

On a more speculative level, we also mention as motivation the connection to two-dimensional quantum gravity. It is well known that random matrix theory is related to a discretization of two-dimensional quantum gravity, namely the analysis of random planar maps -- see e.g. \cite{eml} for a mathematically rigorous discussion of this connection. On the other hand, multiplicative chaos measures play a significant role in the study of Liouville quantum gravity \cite{dkrv,ds} which is in some instances known to be the scaling limit of a suitable model of random planar maps \cite{legall,miermont,ms2,ms3,ms4}. The appearance of multiplicative chaos measures from random matrix theory seems like a curious coincidence from this point of view, and one that deserves further study.

One interpretation of Theorem \ref{th:main} is that it gives a way of probing the (random fractal) set of points $x$ where the recentered log characteristic polynomial $\log |\det (H_N-x)|-\E \log |\det (H_N-x)|$ is exceptionally large. In analogy with standard multiplicative chaos results (see e.g. \cite[Theorem 4.1]{rv1} or the approach of \cite{berestycki}), one would expect that Theorem \ref{th:main} implies that asymptotically,  $\frac{|\det(H_N-x)|^\beta}{\E|\det(H_N-x)|^\beta} dx$ lives on the set of points $x$ where 

\begin{align}\label{eq:thick}
\lim_{N\to\infty}\frac{\log |\det (H_N-x)|-\E \log |\det (H_N-x)|}{\var (\log | \det (H_N - x) |) }=\beta.
\end{align}

\noindent We emphasize that this really means that the (approximately Gaussian) random variable $\log |\det (H_N-x)|-\E \log |\det (H_N-x)|$ would be of the order of its variance instead of its standard deviation -- as the variance is exploding, this is what motivates the claim of the log-characteristic polynomial taking exceptionally large values. Moreover, as it is known that the measure $\mu_\beta$ vanishes for $\beta\geq 2$, this connection suggests that for $\beta>2$, there are no points where \eqref{eq:thick} is satisfied and that $\beta=2$ corresponds to the scale of where the maximum of the field lives (note that it is rigorously known through other methods that the maximum is indeed on the scale of two times the variance of the field -- see \cite{LamPaq} and see also \cite{abb,pz,cmn} for analogous results in the case of ensembles of random unitary matrices). This suggests that suitable variants of Theorem \ref{th:main} should provide a tool for studying extremal values of the characteristic polynomial, or even that more generally, existence of multiplicative chaos measures can be used to study the extremal behavior of log-correlated field. This is significant because maxima of logarithmically correlated fields (such as the log characteristic polynomial) are believed to display universality, and have as such been extensively studied in recent years (see e.g. \cite{fhk} and references below). In fact, the construction of Gaussian multiplicative chaos measures supported on points where the value of the field is a given fraction of the maximal value, may be viewed as part of the programme of establishing universality for such processes. While our results do not extend to the full range of values of $\beta$ where one expects the result to be valid (roughly, we examine only the $L^2$ regime in  Gaussian multiplicative chaos terminology), we believe that an appropriate modification of the methods of this paper eventually will yield the result in its full generality (for instance by combining it with a suitable modification of the approach in \cite{berestycki}).

Regarding this programme, we mention the papers of Arguin, Belius and Bourgade  \cite{abb} which verify the leading order of the maximum of the CUE log characteristic polynomial, as well as Paquette and Zeitouni \cite{pz} which refined this to obtain the second order, doubly logarithmic (``Bramson") correction. This is consistent with a prediction of Fyodorov, Hiary and Keating \cite{fhk}. In turn this was subsequently refined and generalized to the so-called circular $\beta$-ensemble by \cite{cnn} where tightness of the centered maximum was proved. For a large class of random Hermitian matrices, the leading order behavior was established recently by Lambert and Paquette \cite{LamPaq}, while in the case of the Riemann zeta function, the first order term was obtained (assuming the Riemann hypothesis) by Najnudel  \cite{naj} as well as (unconditionally) by Arguin et al. \cite{abbrs}. In the case of the discrete Gaussian free field in two dimensions, the convergence in law of the recentered maximum was obtained recently in an important paper of Bramson, Ding and Zeitouni \cite{bdz}. As for Gaussian multiplicative chaos measures (in the $L^2$-phase), the construction in the case of CUE random matrices was achieved by Webb \cite{webb}. Very recently, a related construction of a Gaussian multiplicative chaos measure was obtained by Lambert, Ostrovsky and Simm \cite{LamOstSimm} in the full $L^1$ regime of CUE random matrices, but for a slightly regularized version of the logarithm of the characteristic polynomial which is closer to a Gaussian field.

\subsection{Organisation of the paper} The outline of the article is the following: in Section \ref{sec:model}, we describe our model and objects of interest, our main results, and an outline of the proof. After this, in Section \ref{sec:hankel}, we recall how the relevant moments can be expressed as Hankel determinants as well as how these determinants are related to orthogonal polynomials on the real line and Riemann-Hilbert problems. In this section we also recall from \cite{dik2} a differential identity for the relevant determinants. Then in Section \ref{sec:rh} we go over the analysis of the relevant Riemann-Hilbert problem. This is very similar to the corresponding analysis in \cite{krasovsky,dik2}, but for completeness and due to slight differences in the proofs, we choose to present details of this in appendices. After this, in Section \ref{sec:di} we use the solution of the Riemann-Hilbert problem to integrate the differential identity to find the asymptotics of the relevant moments. Finally in Section \ref{sec:mainproof}, we put things together and prove our main results. 

We have chosen to defer a number of technical proofs to the end of the paper in the form of multiple appendices. These contain proofs of results which might be considered in some sense routine calculations by experts in random matrix and integrable models, but which would require significant effort to readers not familiar with these techniques. Since we hope that the paper will be of interest to different communities, we have chosen to keep them in the paper at the cost of increasing its length.

\vspace{0.3cm}

{\bf Acknowledgements:}  First of all, we wish to thank the three anonymous reviewers of this article for their careful reading, helpful comments, and pointing out errors in a previous version of this article.  We also wish to thank Igor Krasovsky for pointing out to us how to extend our main result from the GUE to general one-cut regular ensembles,  Benjamin Fahs for helpful discussions about \cite{cf}, and Christophe Charlier for pointing out some errors in a previous version of this article. Further, we wish to thank the Heilbronn Institute for Mathematical Research for support during the workshop Extrema of Logarithmically Correlated Processes, Characteristic Polynomials, and the Riemann Zeta Function, during which part of this work was carried out. N. Berestycki's work is supported by EPSRC grants EP/L018896/1 and EP/I03372X/1. M. D. Wong is a PhD student at the Cambridge Centre for Analysis, supported by EPSRC grant EP/L016516/1. Some of this work was carried out while the first and third authors visited the University of Helsinki, funded in part by EPSRC grant EP/L018896/1. They also wish to thank the University of Helsinki for its hospitality during this visit. C. Webb wishes to thank the Isaac Newton Institute for Mathematical Sciences for its hospitality during the Random Geometry program, during which this project was initiated. C. Webb was supported by the Eemil Aaltonen Foundation grant Stochastic dynamics on large random graphs and Academy of Finland  grants 288318 and 308123. 

\section{Preliminaries and outline of the proof}\label{sec:model}

In this section, we describe the main objects we shall discuss in this article, state our main results, and give an outline of the proof of them.

\subsection{One-cut regular ensembles of random Hermitian matrices} The basic objects we are interested in are $N\times N$ random Hermitian matrices $H_N$ whose distribution can be written as

\begin{equation}\label{eq:matlaw}
\mathbf{P}(dH_N)= \frac{1}{\widetilde{Z}_N(V)}e^{-N\mathrm{Tr}V(H_N)}dH_N,
\end{equation}

\noindent where $dH_N=\prod_{j=1}^N dH_{jj}\prod_{1\leq i<j\leq N}d(\mathrm{Re}H_{ij})d(\mathrm{Im}H_{ij})$ denotes the Lebesgue measure on the space of $N\times N$ Hermitian matrices, $\mathrm{Tr}V(H_N)$ denotes $\sum_{j=1}^N V(\lambda_j)$, where $(\lambda_j)$ are the eigenvalues of $H_N$ (we drop the dependence on $N$ from our notation), the potential $V:\R\to \R$ is a smooth function with nice enough growth at infinity so that this makes sense, and $\widetilde{Z}_N(V)$ is a normalizing constant. Perhaps the simplest model of such form is the Gaussian Unitary Ensemble for which $V(x)=2x^2$. This corresponds to the diagonal entries of $H_N$ being i.i.d. centered normal random variables with variance $1/(4N)$, and the entries above the diagonal being i.i.d. random variables whose real and imaginary parts are centered normal random variables with variance $1/(8N)$ and are independent of each other and of the diagonal entries. The entries below the diagonal are determined by the condition that the matrix is Hermitian.

The distribution \eqref{eq:matlaw} induces a probability distribution for the eigenvalues of $H_N$. In analogy with the GUE (see e.g. \cite{agz}) one finds that the distribution of the eigenvalues (on $\R^N$) is given by
\begin{equation}\label{eq:evlaw}
\mathbb{P}(d\lambda_1,...,d\lambda_N)=\frac{1}{Z_N(V)}\prod_{i<j}|\lambda_i-\lambda_j|^2 \prod_{j=1}^N e^{-NV(\lambda_j)}d\lambda_j,
\end{equation}

\noindent where $Z_N(V)$ is a normalizing constant called the partition function. Our main goal will be to describe the large $N$ behavior of the characteristic polynomial of $H_N$, and more generally a power of this characteristic polynomial. To do this, we will have to impose further constraints on the function $V$. A general family of functions $V$ for which our argument works is the class of one-cut regular potentials. We will review the relevant concepts here, but for more details, see \cite{kml}.

First of all, we assume that $V$ is real analytic on $\R$ and $\lim_{x\to \pm \infty}V(x)/\log|x|=\infty$. Further conditions on $V$ are rather indirect as they are statements about the associated equilibrium measure $\mu_V$ which is defined as the unique minimizer of the functional
\begin{equation*}
\mathcal{I}_V(\mu)=\int\int \log \frac{1}{|x-y|}\mu(dx)\mu(dy)+\int V(x)\mu(dx)
\end{equation*}
on the space of Borel probability measures on $\R$. For further information about $\mu_V$, see e.g. \cite{dkml,st}. The measure $\mu_V$ can also be characterized in terms of Euler--Lagrange equations:
\begin{equation}\label{eq:el1}
2\int\log |x-y|\mu_V(dy)=V(x)+\ell_V, \quad x\in \mathrm{supp}(\mu_V)\\
\end{equation}
\begin{equation}\label{eq:el2}
2\int\log |x-y|\mu_V(dy)\leq V(x)+\ell_V, \quad x\notin \mathrm{supp}(\mu_V)
\end{equation}

\noindent for some constant $\ell_V$ depending on $V$.

Our first constraint on $V$ is that the support of $\mu_V$ is a single interval, and we normalize it to be $[-1,1]$. In this case, on $[-1,1]$, $\mu_V$ can be written as

\begin{equation}\label{eq:equilib}
\mu_V(dx)=d(x)\sqrt{1-x^2}dx,
\end{equation}

\noindent where $d$ is real analytic in some neighborhood of $[-1,1]$ -- see \cite{dkml}. For one-cut regularity, we further assume that $d$ is positive on $[-1,1]$ and that the inequality \eqref{eq:el2} is strict. We collect this all into a single definition.

\begin{definition}[One-cut regular potentials]\label{def:1cp}
We say that the potential $V:\R\to \R$ is one-cut regular {\rm{(}}with normalized support of the equilibrium measure{\rm{)}} if it satisfies the following conditions:

\begin{itemize}[leftmargin=0.5cm]
\item[1.] $V$ is real analytic.
\item[2.] $\lim_{x\to \pm \infty}V(x)/\log |x|=\infty$.
\item[3.] The support of the equilibrium measure $\mu_V$ is $[-1,1]$.
\item[4.] The inequality \eqref{eq:el2} is strict.
\item[5.] The real analytic function $d$ from \eqref{eq:equilib} is positive on $[-1,1]$.
\end{itemize}
\end{definition}

The condition that the support is $[-1,1]$ instead of say $[a,b]$ is not a real constraint since the general case can be mapped to this with a simple transformation. Moreover, note that the support of the equilibrium measure is where the eigenvalues accumulate asymptotically, as the size of the matrix tends to infinity. So in this limit, we expect that nearly all of the eigenvalues of $H_N$ are in $[-1,1]$.

We also point out that this is a non-empty class of functions $V,$ since for the GUE ($V(x)=2x^2$), it is known that all of the conditions of Definition \ref{def:1cp} are satisfied -- in particular $d(x)=2/\pi$ in this case.

\subsection{The characteristic polynomial and powers of its absolute value} As mentioned, our main goal is to describe the large $N$ behavior of the characteristic polynomial of $H_N$. There are several possibilities for what one might want to say. One could consider the characteristic polynomial at a single point, say inside the support of the equilibrium measure, in which case one might expect in analogy with random unitary matrices \cite{ks} that the logarithm of the characteristic polynomial should, as a linear statistic of eigenvalues, be asymptotically a Gaussian random variable with exploding variance. One could consider the behavior of the characteristic polynomial in a microscopic neighborhood of a fixed point, where one might expect it to be asymptotically a random analytic function as it is for the CUE -- see \cite{cnn}, or one could consider the logarithm of the absolute value of the characteristic polynomial on a macroscopic scale inside or outside the support of the equilibrium measure. For the GUE, on the macroscopic scale and in the support of the equilibrium measure, it is known \cite{fks} that the recentered logarithm of the absolute value of the characteristic polynomial behaves like a random generalized function which is formally a Gaussian process with a logarithmic singularity in its covariance.

Our goal is to ``exponentiate" this last statement. (Note that since the limiting process describing the logarithm of a the characteristic polynomial is only a generalized function, and not an actual function defined pointwise, taking its exponential is \emph{a priori} highly nontrivial). More precisely, we make the following definitions.

\begin{definition}\label{def:cpfield}
For $N\in \Z_+$, let $H_N$ be distributed according to \eqref{eq:matlaw}. For $x\in \C$, define
\begin{equation}\label{eq:cpoly}
P_N(x)=\det(H_N-x\mathbf{1}_{N\times N})=\prod_{j=1}^N (\lambda_j-x).
\end{equation}
Moreover, let
\begin{equation}\label{eq:rmtfield}
X_N(x)=\log |P_N(x)|=\sum_{j=1}^N \log \left|\lambda_j- x\right|,
\end{equation}
and for $\beta>0$, define the following measure on $(-1,1)$:
\begin{equation}\label{eq:rmtmeas}
\mu_{N,\beta}(dx)=\frac{e^{\beta X_N(x)}}{\E e^{\beta X_N(x)}}dx=\frac{|P_N(x)|^\beta}{\E |P_N(x)|^\beta}dx.
\end{equation}
\end{definition}

While exponentiating a generalized function in general is impossible, it turns out that in our setting, the correct description of such a procedure is in terms of random measures known as Gaussian multiplicative chaos measures. We now describe some of the basics of the relevant theory.

\subsection{Gaussian Multiplicative Chaos}

Gaussian multiplicative chaos is a theory going back to Kahane \cite{kahane} with the aim of defining what the exponential of a Gaussian random (possibly generalized) function should mean when the covariance kernel of the Gaussian process has a suitable structure, as well as describing some geometric properties of these Gaussian processes.

Kahane proved, that if the covariance kernel has a logarithmic singularity, but otherwise has a particularly nice form, then with a suitable limiting and normalizing procedure, the exponential of the corresponding generalized function can be indeed understood as a random multifractal measure, known as a Gaussian multiplicative chaos measure. For a recent review of the theory, see \cite{rv1} and for a concise proof for existence and uniqueness, see \cite{berestycki}.

Recently, these measures have found applications in constructing random SLE-like planar curves through conformal welding \cite{ajks,sheffield}, quantum Loewner evolution \cite{ms}, the random geometry of two-dimensional quantum gravity \cite{dkrv,ds} -- see also the lecture notes \cite{LQGnotes, rv2}, and even in models of mathematical finance \cite{bkm}. Complex variants of these objects are also connected to the statistical behavior of the Riemann zeta function on the critical line \cite{sw}. Perhaps their greatest importance is the role they are believed to play in describing the scaling limits of random planar maps embedded conformally -- see \cite{ms2,ms3,ms4} and \cite{LQGnotes}. In all of these cases, the covariance kernel of the Gaussian field has a logarithmic singularity on the diagonal.

In this section we will give a brief construction of the measures which are relevant to us. The random distribution we will be interested in is the whole-plane Gaussian free field restricted to the interval $(-1,1)$ with a suitable choice of additive constant. Formally we will want to consider a Gaussian field $X$ defined on $(-1,1)$ such that it has a covariance kernel $\E X(x)X(y)=-\frac{1}{2}\log [2|x-y|]$. It can be shown that it is possible to construct such an object as a random variable taking values in a suitable Sobolev space of  generalized functions, see \cite{fks}. However, we will only need to work with approximations to this distribution which are well defined functions, so we will not need this fact. To motivate our definitions, we first recall  a basic fact about expanding $\log |x-y|$ for $x,y\in(-1,1)$ in terms of Chebyshev polynomials -- see e.g. \cite[Appendix C]{ps}, \cite[Exercise 1.4.4]{for}, or \cite[Lemma 3.1]{gp} for a proof.

\begin{lemma}\label{le:haagerup}
Let $x,y\in(-1,1)$ and $x\neq y$. Then
\begin{equation}\label{eq:logexp}
\log |x-y|=-\log 2-\sum_{n=1}^\infty \frac{2}{n}T_n(x)T_n(y),
\end{equation}

\noindent where $T_n$ is a Chebyshev polynomial of the first kind, i.e. it is the unique polynomial of degree $n$ satisfying $T_n(\cos \theta)=\cos n\theta$ for all $\theta\in[0,2\pi]$.
\end{lemma}

Thus formally, if $(A_k)_{k=1}^\infty$ were i.i.d. standard Gaussians and one defined
$$\mathcal{G}(x)=\sum_{j=1}^\infty \frac{A_j}{\sqrt{j}}T_j(x),
$$
then one would have $\E \mathcal{G}(x)\mathcal{G}(y)=-\frac{1}{2}\log [2|x-y|]$. Motivated by this, we make the following definition.

\begin{definition}
Let $(A_k)_{k=1}^\infty$ be i.i.d. standard Gaussian random variables. For $x\in(-1,1)$ and $M\in \Z_+$, let
\begin{equation}\label{eq:gfieldm}
\mathcal{G}_M(x)=\sum_{j=1}^M \frac{A_j}{\sqrt{j}} T_j(x).
\end{equation}
\end{definition}

We then want to understand $e^{\beta \mathcal{G}}$ (for suitable $\beta$) as a limit related to $e^{\beta \mathcal{G}_M}$ as $M\to \infty$. The precise statement is the following:

\begin{lemma}\label{L:GMC}
Consider the random measure
\begin{equation}\label{eq:gmctrunc}
{\mu}_{\beta}^{(M)}(dx)=e^{\beta {\mathcal{G}}_M(x)-\frac{\beta^2}{2}\E {\mathcal{G}}_M(x)^2}dx
\end{equation}

\noindent on $(-1,1)$. For $\beta\in(-\sqrt{2},\sqrt{2})$, $\mu_{\beta}^{(M)}$ converges weakly almost surely $($when the i.i.d. Gaussians are realized on the same probability space$)$ to a non-trivial random measure $\mu_\beta$ on $(-1,1)$, as $M\to\infty$.
\end{lemma}

This measure $\mu_\beta$ is the limiting object in Theorem \ref{th:main}. The basic idea is that the sequence $\mu_\beta^{(M)}$ is a measure-valued martingale, and it turns out that for $\beta\in(-\sqrt{2},\sqrt{2})$, it is bounded in $L^2$ so by standard martingale theory it has a non-trivial limit. The $L^2$-boundedness is somewhat non-trivial and we will return to the details later.

\begin{remark}\label{rem:bigbeta}
The measure $\mu_\beta$ exists actually for larger values of $|\beta|$ as well. It essentially follows from the standard theory of multiplicative chaos, or alternatively the approach of \cite{berestycki}, that a non-trivial limiting measure exists for $\beta\in(-2,2)$. In fact, comparing with other log-correlated fields, it is natural to expect that with a suitable deterministic normalization, that differs from ours for some values of $\beta$, it is possible to construct a non-trivial limiting object for all $\beta\in \C$. However, for complex $\beta$, the limit might not be in general a measure $($not even a signed measure$)$, but only a distribution. We refer to \cite{lrv} for a study in complex multiplicative chaos and to \cite{mrv} for defining $\mu_\beta$ for large real $\beta$. Our approach for proving convergence relies critically on calculating second moments and it is known for example that the total mass of the measure $\mu_\beta$ has a finite second moment only for $\beta\in(-\sqrt{2},\sqrt{2})$, so our approach is not directly possible for proving a corresponding result in the full range of values of $\beta$ where we would expect the result to hold. However, combining our results, those of \cite{ck}, and the approach of \cite{LamOstSimm} should yield the result for $\beta\in(0,2)$.  This being said, we wish to point out that while the limiting object $\mu_\beta$ should exist for all complex $\beta$, one should not expect that  $\mu_{N,\beta}$ converges to it if the real part of $\beta$ is too negative -- e.g. if $\beta\leq -1$, then with overwhelming probability, $\int_{-1}^1 f(x)|P_N(x)|^{\beta}dx$ will be infinite and one can not hope for convergence. To avoid this type of complications, we focus on non-negative $\beta$.
\end{remark}

\subsection{Outline of the proof}  In this section we define the main objects we analyze in the proof of Theorem \ref{th:main}, and state the main results we need about them. Motivated by the approach in \cite{webb}, we will consider an approximation to $\mu_{N,\beta}$, and we will denote this by $\widetilde{\mu}_{N,\beta}^{(M)}$, where $M$ is an integer parametrizing the approximation. Using known results about the linear statistics of one-cut regular ensembles, it will be clear that as $N\to\infty$ for fixed $M$, $\widetilde{\mu}_{N,\beta}^{(M)}\to \mu_{\beta}^{(M)}$ in distribution. Thus our goal is to control the difference $\mu_{N,\beta}-\widetilde{\mu}_{N,\beta}^{(M)}$, when we first let $N\to\infty$ and then $M\to\infty$.

Let us begin by defining our approximation $\widetilde{\mu}_{N,\beta}^{(M)}$. It is essentially just truncating the Fourier-Chebyshev series of $X_N$, but we have to be slightly careful as the eigenvalues can be outside of $[-1,1]$ with non-zero probability.

\begin{definition}
Fix $M\in \Z_+$ and $\epsilon>0$ {\rm{(}}small and possibly depending on $M${\rm{)}}. Let $\widetilde{T}_j(x)$ be a $C^\infty(\R)$-function with compact support such that $\widetilde{T}_j(x)=T_j(x)$ for each $x\in(-1-\epsilon,1+\epsilon)$. Then define for $x\in(-1,1)$
\begin{equation}\label{eq:rmtfieldtrunc}
\widetilde{X}_{N,M}(x)=-\sum_{k=1}^M\frac{2}{k}\left[\sum_{j=1}^N \widetilde{T}_k\left(\lambda_j\right)\right]T_k(x),
\end{equation}

\noindent and
\begin{equation}\label{eq:rmtmeastrunc}
\widetilde{\mu}_{N,\beta}^{(M)}(dx)=\frac{e^{\beta \widetilde{X}_{N,M}(x)}}{\E e^{\beta \widetilde{X}_{N,M}(x)}}dx.
\end{equation}
\end{definition}

\begin{remark}
Our reasoning here is that if we pretended that all of the $\lambda_j$ are in the interval $(-1,1)$, we could make use of Lemma \ref{le:haagerup}. Then $X_N$ would coincide with the above expansion for $M=\infty$ and $\widetilde{T}_j$ replaced by $T_j$. Outside of the interval, we have to use $\widetilde{T}_k$ instead of $T_k$, as otherwise $\E e^{\beta \widetilde{X}_{N,M}(x)}$ might not exist for all values of $x$ and $M$.
\end{remark}

We will break our main statement down into parts now. The statement of our Theorem \ref{th:main} is equivalent to saying that for each bounded continuous $\varphi:(-1,1)\to [0,\infty)$, $\mu_{N,\beta}(\varphi):=\int_{-1}^1\varphi(x)\mu_{N,\beta}(dx)$ converges in distribution to $\mu_\beta(\varphi)$. It will actually be enough to assume that $\varphi$ has compact support in $(-1,1)$, i.e. to prove vague convergence. We will be more detailed about these statements in the actual proof in Section \ref{sec:mainproof}. The way we will prove vague convergence is to write
\begin{equation*}
\mu_{N,\beta}(\varphi)=[\mu_{N,\beta}(\varphi)-\widetilde{\mu}_{N,\beta}^{(M)}(\varphi)]+\widetilde{\mu}_{N,\beta}^{(M)}(\varphi).
\end{equation*}

By using standard central limit theorems for linear statistics of one-cut regular ensembles, and the definition of $\mu_\beta$, we will see that the second term here tends to $\mu_\beta(\varphi)$ in the limit where first $N\to\infty$, and then $M\to \infty$. Our main result will then follow from showing that the second moment of the first term tends to zero in the same limit. We formulate this as a proposition.

\begin{proposition}\label{prop:main}
If we first let $N\to\infty$ and then $M\to\infty$, then for $\beta\in(0,\sqrt{2})$ and each compactly supported continuous $\varphi:(-1,1)\to [0,\infty)$, $\widetilde \mu_{N,\beta}^{(M)}(\varphi)$ converges in distribution to $\mu_\beta(\varphi)$, and

\begin{equation}
\lim_{M\to\infty}\lim_{N\to\infty}\E |\mu_{N,\beta}(\varphi)-\widetilde{\mu}_{N,\beta}^{(M)}(\varphi)|^2=0.
\end{equation}

\end{proposition}

Proving the second statement takes up most of this article. Expanding the square, we see that what is critical is having uniform asymptotics for $\E e^{\beta X_N(x)}$, $\E e^{\beta \widetilde{X}_{N,M}(x)}$,  $\E e^{\beta (X_N(x)+X_N(y))}$, $\E e^{\beta (\widetilde{X}_{N,M}(x)+\widetilde{X}_{N,M}(y))}$, and $\E e^{\beta(X_N(x)+\widetilde{X}_{N,M}(y))}$. More precisely, we have:
\begin{align*}
\E |\mu_{N,\beta}(\varphi)-\widetilde{\mu}_{N,\beta}^{(M)}(\varphi)|^2 & = \iint \ph(x) \ph(y)  \frac{\E (e^{\beta X_N(x) + \beta X_N(y) })}{\E( e^{\beta X_N(x)})    \E( e^{\beta X_N(y)}) }dx dy\\
& \quad \quad - 2 \iint \ph(x) \ph(y)  \frac{\E (e^{\beta X_N(x) + \beta \widetilde X_{N,M}(y) })}{\E( e^{\beta X_N(x)})    \E( e^{\beta \widetilde X_{N,M}(y)}) }dx dy\\
& \quad  \quad + \iint \ph(x) \ph(y)  \frac{\E (e^{\beta \widetilde X_{N,M}(x) + \beta \widetilde X_{N,M}(y) })}{\E( e^{\beta \widetilde X_{N,M}(x)})    \E( e^{\beta \widetilde X_{N,M}(y)}) }dx dy .
\end{align*}
Each of these expectations here can be expressed as $\E \prod_{j=1}^Nh(\lambda_j)$ for a suitable function $h:\R\to \R$. For instance,
$$
e^{\beta X_N(x) + \beta \widetilde X_{N,M}(y) } = \prod_{j=1}^N |\lambda_j - x|^\beta
e^{  \cT(\lambda_j)} ; \text{ where } \cT(\lambda) = \cT(\lambda;y) = -\beta \sum_{k=1}^M \frac2k \widetilde T_k(\lambda) T_k(y).
$$
As we will recall in Section \ref{sec:hankel}, such quantities can be expressed in terms of Hankel determinants. Moreover, all of these Hankel determinants have a very specific type of symbol: one with so-called \textbf{Fisher--Hartwig} singularities. To explain what this means here, a Hankel matrix is a matrix in which the skew-diagonals are constant. They are closely related to Toeplitz matrices where the diagonals themselves are constant (these arise typically in the study of CUE and related random matrix ensembles rather than the GUE-type ensembles considered in this paper). In the case we will be interested in, the $(i,j)$th coefficient of the Hankel matrix will be of the form $\int_\R x^{i+j}h(x) e^{-NV(x)}dx$ where $h$ is as above. When $h$ is smooth enough and doesn't have any roots, then the asymptotic analysis of such determinants would follow from the classical strong Szeg\H{o} theorem (actually this theorem applies in the Toeplitz case rather than the Hankel case, but here this isn't a crucial distinction).   However in our situation $h$ typically contains at least one root of the form $|x-x_i|^{\beta_i}$, which greatly complicates the task of analysing the corresponding determinant. This type of behavior is an example of a Fisher--Hartwig singularity. (In general a Fisher--Hartwig singularity might also include a jump at $x_i$ corresponding to the symbol also having a term of the form  $e^{\gamma \mathrm{Im} \log (x- x_i)}$.)

The asymptotics of Hankel determinants with Fisher--Hartwig singularities is still very much a subject of active research, and much information is already available using the steepest descent technique due to Deift and Zhou \cite{dz}; see in particular the papers \cite{dik1,dik2,krasovsky,cf} which play an important role in our proof. Yet results in the generality we need seem to still be lacking in the literature. What suffices for us is the following result (which we will only use with $k=1$ or $k=2$, but since there is no added difficulty in proving it for a general value of $k$ we will do so).

\begin{proposition}\label{prop:fh}
Let $\mathcal{T}\in C^\infty(\R)$ be real analytic in some neighborhood of $[-1,1]$ and  have compact support. Let $k\in \Z_+$ be fixed, and let $\beta_1,...,\beta_k\in[0,\infty)$ be fixed. Moreover, let $x_1,...,x_k\in(-1,1)$ be distinct. Finally let $H_N$ be a $N\times N$ random Hermitian matrix drawn from a one-cut regular unitary invariant ensemble with potential $V$. Then for $C(\beta)=2^{\frac{\beta^2}{2}}\frac{G(1+\beta/2)^2}{G(1+\beta)}$, where $G$ is the Barnes $G$ function, we have as $N \to \infty$,
\begin{align}\label{eq:fh}
\E &\left[e^{\sum_{j=1}^N \mathcal{T}(\lambda_j)}\prod_{i=1}^k |\det(H_N-x_i)|^{\beta_i}\right]\\
\notag &=\prod_{j=1}^k C(\beta_j)\left(d(x_j)\frac{\pi}{2}\sqrt{1-x_j^2}\right)^{\frac{\beta_j^2}{4}}\left(\frac{N}{2}\right)^{\frac{\beta_j^2}{4}}e^{(V(x_j)+\ell_V)\frac{\beta_j}{2}N}\prod_{1\le i<j\le k}|2(x_i-x_j)|^{-\frac{\beta_i\beta_j}{2}}\\
\notag &\quad \times e^{N\int_{-1}^1 \Tree(x)d(x)\sqrt{1-x^2}dx+\sum_{j=1}^k \frac{\beta_j}{2}\left[\int_{-1}^1 \frac{\Tree(x)}{\pi \sqrt{1-x^2}}dx-\Tree(x_j)\right]}\\
\notag &\quad \times e^{\frac{1}{4\pi^2}\int_{-1}^1 dy\frac{\Tree(y)}{\sqrt{1-y^2}}P.V.\int_{-1}^1 \frac{\Tree'(x)\sqrt{1-x^2}}{y-x}dx}(1+o(1))
\end{align}

\noindent uniformly on compact subsets of $\lbrace (x_1,...,x_k)\in(-1,1)^k: x_i\neq x_j \ \mathrm{for} \ i\neq j\rbrace$. Here $P.V.\int$ denotes the Cauchy principal value integral. Moreover, if there exists a fixed $M\in \Z_+$, such that in some fixed neighborhood of $[-1,1]$, $\mathcal{T}(x)=\sum_{j=1}^M \alpha_j T_j(x)$, then the above asymptotics are uniform also in compact subsets of $\lbrace (\alpha_1,...,\alpha_M)\in \R^M\rbrace$.
\end{proposition}

\begin{remark}\label{rem:ff}
As mentioned in the introduction, this settles some conjectures due to Forrester and Frankel -- see \cite[Conjecture 5 and Conjecture 8]{ff} for more details. In terms of the potential $V$, we actually improve on the conjectures as these are only stated for polynomial $V$, but concerning the functions $\Tree$, our results are not as general as those appearing in the conjectures of Forrester and Frankel. This being said, one could easily relax some of our regularity assumptions on $\Tree$. In fact, the compact support or smoothness outside of a neighborhood of the interval $[-1,1]$ play essentially no role in our proof, but as this is a simple and clear way of stating the result, we do not attempt to state things in their greatest generality. Moreover, using techniques from \cite{dik2}, one could attempt to generalize our estimates and prove a corresponding result when $\mathcal{T}$ is less smooth also on $[-1,1]$. Again, this is not necessary for our main goal, so we don't pursue this further.

 We also mention that after the first version of this article appeared, Charlier $($in \cite{charlier}$)$ proved an extension of this result to the case where the symbol can also have jump-type singularities.

\end{remark}

We prove our results through Riemann--Hilbert methods. In particular, we first show that with a suitable differential identity, and some analysis of a Riemann--Hilbert problem, we can relate the $\mathcal{T}=0$ case to the $\mathcal{T}\neq 0$ case.  Then with another differential identity (and further analysis of another Riemann--Hilbert problem) we relate the $\Tree=0$, general $V$ -case to the GUE with $\Tree=0$. The asymptotics in the $\Tree=0$ case for the GUE have been obtained by Krasovsky \cite{krasovsky}. Using these, we are able to prove Proposition \ref{prop:fh}.

As we will need uniform asymptotics for $\E e^{\beta X_N(x)+\beta X_N(y)}$ and other terms, Proposition \ref{prop:fh} is not quite enough for us. For uniform estimates, we will rely on a recent result of Claeys and Fahs \cite{cf}, which combined with Proposition \ref{prop:fh} will let us prove Proposition \ref{prop:main}.

Next we review the connection between expectations of the form \eqref{eq:fh}, Hankel determinants, and Riemann--Hilbert problems.

\section{Hankel determinants and  Riemann--Hilbert problems}\label{sec:hankel}

In this section, we recall how the expectations we are interested in can be written as Hankel determinants, which are related to orthogonal polynomials, which in turn can be encoded into a Riemann--Hilbert problem. We also recall certain differential identities we will need for analyzing the expectations we are interested in. While our discussion is very similar to that in e.g. \cite{dik1,dik2}, there are some minor differences as we are dealing with Hankel determinants instead of Toeplitz ones. We choose to give some details for the convenience of a reader with limited experience with Riemann-Hilbert problems.

\subsection{Hankel determinants and orthogonal polynomials}

Terms of the form $\E\prod_{j=1}^N f(\lambda_j)$ can be written in determinantal form due to Andreief's identity -- for a proof, one can use e.g. \cite[Lemma 3.2.3]{agz} with the functions $f_i(x)=f(x)e^{-NV(x)}x^{i-1}$ and $g_i(x)=x^{i-1}$ as well as the product representation of the Vandermonde determinant.

\begin{lemma}\label{le:andr}
Let $f:\R\to \R$ be a nice enough function {\rm{(}}measurable and nice enough decay that all the relevant integrals converge absolutely{\rm{)}}. Then
\begin{equation}\label{eq:andr}
\E \prod_{j=1}^N f(\lambda_j)=\frac{N!}{Z_N(V)}\det\left(\int_\R x^{i+j}f(x) e^{-NV(x)}dx\right)_{i,j=0}^{N-1}.
\end{equation}
where $Z_N(V)$ is as in \eqref{eq:evlaw}. 
\end{lemma}

Let us introduce some notation for the Hankel determinant here.

\begin{definition}
For nice enough functions $f:\R\to \R$, {\rm{(}}so that the integrals exist{\rm{)}} let

\begin{equation}\label{eq:det}
D_k(f)=D_k(f;V)=\det\left(\int_\R x^{i+j}f(x) e^{-NV(x)}dx\right)_{i,j=0}^{k}.
\end{equation}
\end{definition}

As the notation suggests, we will suppress the dependence on $V$ when it's convenient. We suppress the dependence on $N$ always.

It is a well known result in the theory of orthogonal polynomials, that such determinants can be written in terms of orthogonal polynomials. For the convenience of the reader, we offer a proof for the following result.

\begin{lemma}\label{le:poly}
Let $f:\R\to\R$ be positive Lebesgue almost everywhere, have nice enough regularity and growth at infinity, and let $(p_j(x;f,V))_{j=0}^\infty$ be the sequence of real polynomials which have a positive leading order coefficient and which are orthonormal with respect to the measure $f(x)e^{-NV(x)}dx$ on $\R$ {\rm{(}}we will write $p_j(x;f)$ when we wish to suppress the dependence on $V$ and we will always suppress the dependence on $N${\rm{)}}:

\begin{equation}\label{eq:ortho}
\int_\R p_j(x;f)p_k(x;f) f(x)e^{-NV(x)}dx=\delta_{j,k},
\end{equation}

\noindent and $p_j(x;f)=\chi_j(f)x^j+\mathcal{O}(x^{j-1})$ as $x\to\infty$, where $\chi_j(f)>0$. Then

\begin{equation}\label{eq:dchi}
D_k(f)=\prod_{j=0}^k \chi_j(f)^{-2}.
\end{equation}
\end{lemma}

Note that due to our assumptions on $f$, the above polynomials do exist as we can construct them by applying the \emph{determinantal representation} associated with the Gram--Schmidt procedure to the monomials.

\begin{proof}
Consider the space of real polynomials, equipped with  an inner product given by the $L^2$ inner product on $\R$ with weight $f(x)e^{-NV(x)}$. A consequence of the Gram--Schmidt procedure applied to the sequence of monomials in this inner product space is the following: 
 for $j\geq1$ 
\begin{equation}\label{eq:gramdet}
p_{j}(x;f)=\frac{1}{\sqrt{D_{j-1}(f)D_j(f)}}\begin{vmatrix}
\int f(y)e^{-NV(y)}dy & \cdots & \int y^{j} f(y)e^{-NV(y)}dy\\
\vdots & \ddots & \vdots\\
\int y^{j-1}f(y)e^{-NV(y)}dy & \cdots & \int y^{2j-1}f(y)e^{-NV(y)}dy\\
1 & \cdots & x^{j}
\end{vmatrix}.
\end{equation}
where for $j=0$ the determinant is replaced by $1$, and  $D_{-1}(f)=1$.

Note that from our assumption on $f$ and an easy generalization of Lemma \ref{le:andr}, $D_j(f)>0$ for all $j\geq 0$, so these polynomials exist. From \eqref{eq:gramdet} one sees that $\chi_j(f)$ -- the coefficient of $x^j$ in $p_j(x;f)$ --  equals $\sqrt{D_{j-1}(f)/D_j(f)}$. The claim then follows as the product has a telescopic form, and we defined $D_{-1}(f)=1$.
\end{proof}

\subsection{Riemann-Hilbert problems and orthogonal polynomials}

We now recall a result going back to Fokas, Its, and Kitaev \cite{fik} about encoding orthogonal polynomials on the real line into a Riemann-Hilbert problem. In our setting, the relevant result is formulated in the following way.

\begin{proposition}[Fokas, Its, and Kitaev]\label{prop:yrhp}
Let $\mathcal{T}$ be a real valued $C^\infty(\R)$ function with compact support, let $(\beta_j)_{j=1}^k\in [0,\infty)^k$ , $(x_j)_{j=1}^k\in(-1,1)^k$, and $x_i\neq x_j$ for $i\neq j$. Let $V$ be some real analytic function on $\R$ satisfying $\lim_{x\to\pm \infty} V(x)/\log |x|=\infty$. For $\lambda\in \R$, define

\begin{equation}\label{eq:ffh}
f(\lambda)=e^{\mathcal{T}(\lambda)}\prod_{j=1}^k\left|\lambda-x_j\right|^{\beta_j},
\end{equation}

\noindent and let $p_j(x;f)$ be as in Lemma \ref{le:poly}, with the relevant measure being $f(\lambda) e^{-NV(\lambda)}d\lambda$ on $\R$. Consider the $2\times 2$ matrix-valued function

\begin{align}\label{eq:Y_RHP}
Y(z)&=Y_{j}(z;f,V)=\begin{pmatrix}
\frac{1}{\chi_j(f)}p_j(z;f) & \frac{1}{\chi_j(f)}\int_\R\frac{p_j(\lambda;f)}{\lambda-z}\frac{f(\lambda)e^{-NV(\lambda)}d\lambda}{2\pi i}\\
-2\pi i \chi_{j-1}(f)p_{j-1}(z;f) & -\chi_{j-1}(f)\int_\R \frac{p_{j-1}(\lambda;f)}{\lambda-z}f(\lambda)e^{-NV(\lambda)}d\lambda
\end{pmatrix},
\end{align}

\noindent for $z\in \C\setminus \R$. Then $Y$ is the unique solution to the following Riemann-Hilbert problem: find a function $Y:\C\setminus \R\to \C^{2\times 2}$ such that

\begin{itemize}[leftmargin=0.5cm]
\item[1.] $Y$ is analytic.
\item[2.] On $\R$, $Y$ has continuous boundary values $Y_\pm$, i.e. $Y_\pm(\lambda)=\lim_{\epsilon\to 0^+}Y(\lambda\pm i\epsilon)$ exists and is continuous for all $\lambda\in\R$. Moreover, $Y_\pm$ are related by the jump condition

\begin{equation}\label{eq:Yjump}
Y_+(\lambda)=Y_-(\lambda) \begin{pmatrix}
1 & f(\lambda) e^{-NV(\lambda)}\\
0 & 1
\end{pmatrix},  \qquad \lambda\in \R.
\end{equation}

\item[3.] As $z\to\infty$,

\begin{equation}\label{eq:Ynorm}
Y(z)=(I+\mathcal{O}(z^{-1}))\begin{pmatrix}
z^j & 0\\
0 & z^{-j}
\end{pmatrix}.
\end{equation}
\end{itemize}
\end{proposition}

\begin{remark}
Typically for Riemann-Hilbert problems related to Toeplitz and Hankel determinants with Fisher-Hartwig singularities {\rm{(}}e.g. \cite{dik1,dik2,cf}{\rm{)}} one says that the boundary values are continuous on the relevant contour minus the singularities $x_j$, and then imposes conditions on the behavior of $Y$ near the singularities. This is relevant when one of the $\beta_j$ is negative or non-real, but as we will shortly mention, in our case the boundary values are truly continuous on $\R$ and no further condition is needed.
\end{remark}

\begin{proof}[Sketch of proof]
The proof for uniqueness is the standard one: one first looks at some solution to the RHP, say $Y$. From the jump condition, it follows that $\det Y$ is continuous across $\R$, so it is entire. From the behavior of $Y$ at infinity, it follows that $\det Y$ is bounded, so by Liouville's theorem and the behavior at infinity, one sees that $\det Y=1$. In particular, (as a matrix) $Y$ is invertible and the inverse matrix $Y^{-1}$ is analytic in $\C\setminus \R$. Now if $\widetilde{Y}$ is another solution, we see that $\widetilde{Y}Y^{-1}$ is analytic in $\C\setminus \R$ and continuous across $\R$, so it is entire. From the behavior at infinity, $\widetilde{Y}(z)Y(z)^{-1}\to I$ (the $2\times 2$ identity matrix) as $z\to \infty$, so again by Liouville, $\widetilde{Y}=Y$.

Consider then the statement that $Y$ given in terms of the orthogonal polynomials is a solution. The analyticity condition is obvious. The continuity of the boundary values of the first column is obvious since we are dealing with polynomials. For the second column, the Sokhotski-Plemelj theorem implies that the boundary values of the second column can be expressed in terms of $p_j f e^{-NV}$ (or $p_j$ replaced by $p_{j-1}$) and its Hilbert transform (see  e.g. \cite[Chapter V]{titchmarsh} for an introduction to the Hilbert transform). 
The first term is obviously continuous. For the Hilbert transform, we note that $p_j fe^{-NV}$ is H\"older continuous, so as the Hilbert transform preserves H\"older regularity (see \cite[Chapter V.15]{titchmarsh}), we see that the boundary values of $Y$ are continuous.

For the jump condition \eqref{eq:Yjump} and behavior at infinity \eqref{eq:Ynorm}, we refer to analogous problems in \cite[Section 3.2 and Section 7]{deift}.
\end{proof}

We next discuss how deforming $V$ or $\mathcal{T}$ changes $D_{N-1}(f;V)$.

\subsection{Differential identities}

Let us fix our potential $V$ (and drop dependence on it from our notation) and first consider how deforming $\mathcal{T}$ changes $D_{N-1}(f)$.

The proof of the following result is a minor modification of the proof of \cite[Proposition 3.3]{dik2}, but for completeness, we give a proof in Appendix \ref{app:di}. The role of this result is that if we know the asymptotics in the case $\Tree=0$, instead of studying $Y_j$ for all $j$, it's enough to study $Y_N$ though with a one-parameter family of deformations of $\mathcal{T}$.

\begin{lemma}\label{le:di}
Let $\mathcal{T}:\R\to \R$ be a $C^\infty$ function with compact support, let $(\beta_j)_{j=1}^k\in [0,\infty)^k$ , $(x_j)_{j=1}^k\in(-1,1)^k$, and $x_i\neq x_j$ for $i\neq j$. For $t\in[0,1]$ and $\lambda\in \R$, define

\begin{equation}\label{eq:ftdef}
f_t(\lambda)=\left[1-t+t e^{\mathcal{T}(\lambda)}\right]\prod_{j=1}^k |\lambda-x_j|^{\beta_j}.
\end{equation}

Let $Y(z,t)$ be as in \eqref{eq:Y_RHP} with $j=N$, $f=f_t$, and $p_l(x;f)=p_l(x;f_t)$ the orthonormal polynomials with respect to the measure $f_t(\lambda)e^{-NV(\lambda)}d\lambda$ on $\R$. Then

\begin{equation}\label{eq:di1}
\partial_t\log D_{N-1}(f_t)=\frac{1}{2\pi i}\int_\R\left[Y_{11}(x,t)\partial_x Y_{21}(x,t)-Y_{21}(x,t)\partial_x Y_{11}(x,t)\right]\partial_t f_t(x)e^{-NV(x)}dx,
\end{equation}

\noindent where the indices of $Y$ refer to matrix entries.
\end{lemma}

The object we are interested in is $D_{N-1}(f_1)$ which we can analyze by writing

\begin{equation*}
\log D_{N-1}(f_1)=\log D_{N-1}(f_0)+\int_0^1\frac{\partial}{\partial t}\log D_{N-1}(f_t)dt.
\end{equation*}

For the GUE, the asymptotics of $D_{N-1}(f_0)$ -- the case $\Tree=0$ -- were investigated in \cite{krasovsky}, so a consequence of Lemma \ref{le:di} is that if we understand the asymptotics of $Y(z,t)$ well enough, we are able to study the asymptotics of $D_{N-1}(f_1)$ in the GUE case.

The other deformation we will consider is what happens when we interpolate between the potentials $V_0(x)=2x^2$ (the GUE) and $V_1(x)=V(x)$ in the $\Tree=0$ case.

\begin{lemma}\label{le:di2}
Let $(\beta_j)_{j=1}^k\in [0,\infty)^k$, $(x_j)_{j=1}^k\in(-1,1)^k$, and $x_i\neq x_j$ for $i\neq j$. Let $f$ be defined by \eqref{eq:ffh} with $\Tree=0$ and let $V:\R\to \R$ be a real analytic function satisfying $\lim_{x\to \pm \infty}V(x)/\log |x|=\infty$. Define for $s\in[0,1]$

\begin{equation}\label{eq:vs}
V_s(x)=(1-s)2x^2+sV(x).
\end{equation}

Let us then write $Y(z;V_s)$ for $Y$ defined as in \eqref{eq:Y_RHP} with $j=N$, $V=V_s$ and $p_j(x;f)=p_j(x;f,V_s)$. Then using the notation of \eqref{eq:det}

\begin{align}\label{eq:di2}
\partial_s &\log D_{N-1}(f;V_s)\\
\notag &=-N\frac{1}{2\pi i}\int_\R\left[Y_{11}(x;V_s)\partial_x Y_{21}(x;V_s)-Y_{21}(x;V_s)\partial_x Y_{11}(x;V_s)\right]f(x)[\partial_s V_s(x)]e^{-NV_s(x)}dx.
\end{align}
\end{lemma}

Again, we give a proof in Appendix \ref{app:di}. The role of this differential identity is that if we understand the asymptotics of $Y(z;V_s)$ well enough, then by integrating \eqref{eq:di2}, we can move from the GUE asymptotics to the general ones.

We mention that both of these identities are of course true for a much wider class of symbols than what we state in the results (in particular, in Lemma \ref{le:di2} the condition $\Tree=0$ is not necessary for anything). This is simply the generality we use them in. Next we move on to describing how to study the large $N$ asymptotics of $Y(z,t)$ and $Y(z;V_s)$.

\section{Solving the Riemann-Hilbert problem} \label{sec:rh}

In this section we will finally describe the asymptotic behavior of $Y(z,t)$ and $Y(z;V_s)$ as $N\to\infty$. The typical way this is done is through a series of transformations to the RHP, ultimately leading to a RHP where the jump matrix is asymptotically close to the identity matrix as $N \to \infty$, and the behavior at infinity is close to the identity matrix. Then using properties of the Cauchy-kernel, the final RHP can be solved in terms of a Neumann series solution of a suitable integral equation. Moreover, each term in the series expansion is of lower and lower order in $N$. We will go into further details about this part of the problem in Section \ref{sec:rhpasy}, but we will start with transforming the problem.

While we never have both $s,t\in(0,1)$, we will find it notationally convenient to consider $Y(z)$ to be defined as in \eqref{eq:Y_RHP} with $f=f_t$ and $V=V_s$. We suppress all of this in our notation for $Y$. We will also focus on functions $\Tree$ with the regularity claimed in Proposition \ref{prop:fh} which was stronger than what we stated in the differential identities.

\subsection{Transforming the Riemann-Hilbert problem} Let us introduce some further notation to simplify things later on. Let $\Tree$ satisfy the conditions of Proposition \ref{prop:fh}, and let

\begin{equation}\label{eq:mathcaltt}
\mathcal{T}_t(\lambda)=\log(1-t+te^{\mathcal{T}(\lambda)})
\end{equation}

\noindent so that in the notation of Lemma \ref{le:di}

\begin{equation*}
f_t(\lambda)=e^{\mathcal{T}_t(\lambda)}\prod_{j=1}^k |\lambda-x_j|^{\beta_j},
\end{equation*}

\noindent and let us assume that the singularities are ordered: $x_j<x_{j+1}$.

The series of transformations we will now start implementing is a minor modification of that in \cite[Section 4]{krasovsky}.

\subsubsection{The first transformation}\label{sec:trans1} Our first transformation will change the asymptotic behavior of the solution to the RHP so that it is close to the identity as $z\to\infty$, as well as cause the distance between the jump matrix and the identity matrix to be exponentially small in $N$ when we're off of the interval $[-1,1]$. The proofs of the statements of this section are either elementary or straightforward generalizations of standard ones in the RHP-literature, but for the convenience of readers unfamiliar with the literature, they are sketched in Appendix \ref{app:trans1}. Let us now make the relevant definitions.

\begin{definition}\label{def:gandt}
In the notation of \eqref{eq:equilib}, for $s \in[0,1]$ as above, let

\begin{equation}\label{eq:dsdef}
d_s(\lambda)=(1-s)\frac{2}{\pi}+s d(\lambda),
\end{equation}

\noindent and for $z\in\C\setminus(-\infty,1]$, let

\begin{equation}\label{eq:gsdef}
g_s(z)=\int_{-1}^1 d_s(\lambda)\sqrt{1-\lambda^2}\log(z-\lambda)d\lambda,
\end{equation}

\noindent where the branch of the logarithm is the principal one. We also define

\begin{equation}\label{eq:lsdef}
\ell_s=(1-s)(-1-2\log 2)+s\ell_V,
\end{equation}

\noindent where $\ell_V$ is the constant from \eqref{eq:el1} and \eqref{eq:el2}. Finally, for $z\in \C\setminus \R$, let

\begin{equation}\label{eq:tdef}
T(z)=e^{-N\ell_s\sigma_3/2}Y(z)e^{-N(g_s(z)-\ell_s/2)\sigma_3},
\end{equation}

\noindent where

\begin{equation*}
\sigma_3=\begin{pmatrix}
1 & 0\\
0 & -1
\end{pmatrix} \qquad and \qquad e^{q \sigma_3}=\begin{pmatrix}
e^q & 0\\
0 & e^{-q}
\end{pmatrix}.
\end{equation*}

\end{definition}

Before describing the jump structure and normalization of $T$ near infinity, we first point out some simple facts about the boundary values of $g_s$ on $\R$ which follow from its definition and \eqref{eq:el1} (details may be found in  Appendix \ref{app:trans1}). 

\begin{lemma}\label{le:gbv}
For $\lambda\in \R$, let  $g_{s,\pm}(\lambda)=\lim_{\epsilon\to 0^+}g_s(\lambda\pm i \epsilon)$. Then for $\lambda\in(-1,1)$ and $s\in[0,1]$

\begin{equation}\label{eq:gbv1}
g_{s,+}(\lambda)+g_{s,-}(\lambda)=V_s(\lambda)+\ell_s.
\end{equation}

There exist $M,C>0$ {\rm{(}}independent of $s${\rm{)}} so that for $\lambda\in \R\setminus[-1,1]$,

\begin{equation}\label{eq:gbv2}
g_{s,+}(\lambda)+g_{s,-}(\lambda)-V_s(\lambda)-\ell_s\leq \begin{cases}
-C(|\lambda|-1)^{3/2}, & |\lambda|-1\in(0,M)\\
-\log |\lambda|, & |\lambda|-1>M
\end{cases}.
\end{equation}

For $\lambda\in \R$

\begin{equation}\label{eq:gbv3}
g_{s,+}(\lambda)-g_{s,-}(\lambda)=\begin{cases}
2\pi i, & \lambda<-1\\
2\pi i\int_{\lambda}^1 d_s(x)\sqrt{1-x^2}dx, & |\lambda|<1\\
0, & \lambda>1
\end{cases}.
\end{equation}
\end{lemma}

The function $g_{s,+}-g_{s,-}$ along with an analytic continuation of it will play a significant role in our analysis of the Riemann-Hilbert problem, so we give it a name.

\begin{definition}\label{def:hdef}
Let $U\subset \C$ be an open neighborhood of $\R$ into which $d$ has an analytic continuation. For $z\in U\setminus ((-\infty,-1]\cup [1,\infty))$ and $s\in[0,1]$, let

\begin{equation}\label{eq:hdef}
h_s(z)=-2\pi i\int_{1}^z d_s(w)\sqrt{1-w^2}dw,
\end{equation}

\noindent where the square root is according to the principal branch {\rm{(}}i.e. $\sqrt{1-w^2}=e^{\frac{1}{2}\log(1-w^2)}$ and the branch of the logarithm is the principal one{\rm{)}}, and the contour of integration is such that it stays in $U$ and does not cross $(-\infty,-1]\cup[1,\infty)$.
\end{definition}

The function $h_s$ will often appear in the form $e^{\pm Nh_s}$ and to estimate the size of such an exponential, we will need to know the sign of $\mathrm{Re}(h_s)$. For this, we use the following elementary fact.

\begin{lemma}\label{le:hsign}
In a small enough open neighborhood of $(-1,1)$ {\rm{(}}independent of $s${\rm{)}} in the complex plane,
\begin{equation*}
\mathrm{Re}(h_s(z))>0 \qquad  if \qquad \mathrm{Im}(z)>0
\end{equation*}
 and
\begin{equation*}
\mathrm{Re}(h_s(z))<0 \qquad if \qquad \mathrm{Im}(z)<0
\end{equation*}
for all $s\in[0,1]$, and if we restrict to a fixed set in the upper half plane  such that the set is bounded away from the real axis, but inside this neighborhood of $(-1,1)$, we have e.g. $\mathrm{Re}(h_s(z))\geq\epsilon>0$ for some $\epsilon>0$ independent of $s$. A similar result holds in the lower half plane.
\end{lemma}

Again, see Appendix \ref{app:trans1} for details on the proof of this and the next result, which describes the Riemann-Hilbert problem $T$ solves.

\begin{lemma}\label{le:trhp}
The function $T:\C\setminus \R\to \C^{2\times 2}$ defined by \eqref{eq:tdef} is the unique solution to the following Riemann-Hilbert problem.

\begin{itemize}[leftmargin=0.5cm]
\item[1.] $T:\C\setminus \R\to \C^{2\times 2}$ is analytic.
\item[2.] On $\R$, $T$ has continuous boundary values $T_\pm$ and these are related by the jump conditions

\begin{equation}\label{eq:tjump1}
T_+(\lambda)=T_-(\lambda)\begin{pmatrix}
e^{-Nh_s(\lambda)} & f_t(\lambda)\\
0 & e^{Nh_s(\lambda)}
\end{pmatrix}, \qquad \lambda\in(-1,1)
\end{equation}

\noindent and

\begin{equation}\label{eq:tjump2}
T_+(\lambda)=T_-(\lambda)\begin{pmatrix}
1 & f_t(\lambda)e^{N(g_{s,+}(\lambda)+g_{s,-}(\lambda)-\ell_s-V_s(\lambda))}\\
0 & 1
\end{pmatrix}, \qquad \lambda\in \R\setminus[-1,1].
\end{equation}

\item[3.] As $z\to\infty$,

\begin{equation}\label{eq:tnorm}
T(z)=I+\mathcal{O}(|z|^{-1}).
\end{equation}
\end{itemize}
\end{lemma}

The jump matrix given by \eqref{eq:tjump1} and \eqref{eq:tjump2} already looks good for $\lambda\notin [-1,1]$, in the sense that it is exponentially close to the identity, (compare \eqref{eq:tjump2} with \eqref{eq:gbv2}).
However, the issue is that across $(-1,1)$, the jump matrix is not close to the identity in any way. We will next address this issue by performing a second transformation.

\subsubsection{The second transformation}\label{sec:trans2}

As customary in this type of problems, the next step is to ``open lenses". That is, we will add further jumps to the problem off of the real line. Due to a nice factorization property of the jump matrix for $T$, the new jump matrix will be close to the identity on the new jump contours when we are not too close to the points $\pm 1$ or $x_j$.

Before going into the details of this, we will define an analytic continuation of $f_t$ into a subset of $\C$. Recall from our assumptions in Proposition \ref{prop:fh} that on $(-1-\epsilon,1+\epsilon)$, $\mathcal{T}(x)$ is real analytic. Thus $\mathcal{T}$ certainly has an analytic continuation to some neighborhood of $[-1,1]$. Moreover as it is real on $[-1,1]$, we see that in some small enough complex
neighborhood of $[-1,1]$ (which is independent of $t$), $1-t+te^{\mathcal{T}(z)}$ has no zeroes for any $t\in[0,1]$. Thus $\mathcal{T}_t$ (see \eqref{eq:mathcaltt}) has an analytic continuation to this neighborhood for all $t\in[0,1]$. We use this to define the analytic continuation of $f_t$.

\begin{definition}\label{def:ftcont}
Let $U_{[-1,1]}$ be some neighborhood of $[-1,1]$ which is independent of $t$ and in which $\mathcal{T}_t$ is analytic for $t\in[0,1]$. In this domain, and for $1\leq l\leq k-1$, let

\begin{equation}\label{eq:ftcontdef}
f_t(z)=e^{\mathcal{T}_t(z)}\times\begin{cases}
\prod_{j=1}^k (x_j-z)^{\beta_j}, & \mathrm{Re}(z)<x_1\\
\prod_{j=1}^l (x_j-z)^{\beta_j}\prod_{j=l+1}^k (z-x_j)^{\beta_j}, & \mathrm{Re}(z)\in(x_l,x_{l+1})\\
\prod_{j=1}^k (z-x_j)^{\beta_j}, & \mathrm{Re}(z)>x_k
\end{cases},
\end{equation}

\noindent where the powers are according to the principal branch.

\end{definition}

We will now impose some conditions on our new jump contours. Later on, we will be more precise about what we exactly want from them, but for now, we will ignore the details.

\begin{definition}\label{def:lenses}
For $j=1,...,k+1$, let $\Sigma_{j}^+$ $(\Sigma_{j}^-)$, be a smooth curve in the upper {\rm{(}}lower{\rm{)}} half plane from $x_{j-1}$ to $x_j$, where we understand $x_0$ as $-1$ and $x_{k+1}$ as $1$. The curves are oriented from $x_{j-1}$ to $x_j$ and independent of $t,$ $s,$ and $N$. Moreover, they are contained in $U_{[-1,1]}$.

The domain between $\Sigma_j^+$ and $\Sigma_j^{-}$ is called a lens. The domain between $\Sigma_j^+$ and $\R$ is called the top part of the lens, and that between $\Sigma_j^-$ and $\R$ the bottom part of the lens. See Figure \ref{fig:lenses} for an illustration.
\end{definition}

\begin{remark}
Our definition here and our coming construction implicitly assume that $\beta_j\neq 0$ for all $j$. If one $($or more$)$ $\beta_j=0$, one simply ignores the corresponding $x_j$ $($so e.g. one connects $x_{j-1}$ to $x_{j+1}$ with a curve in the upper half plane etc$)$.
\end{remark}

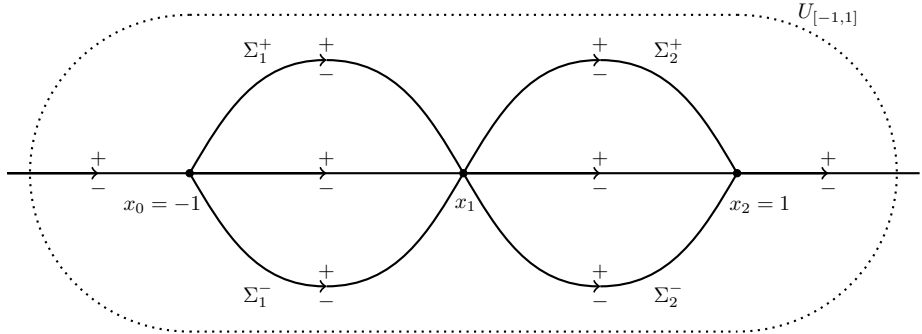
\begin{figure}[h!]
\begin{center}
\begin{tikzpicture}[scale=0.6, every node/.style={scale=0.8}]
\draw[thick, ->] (-4, 0) -- (0,0) (-4, 0) --(-2, 0) node[below]{$-$} node[above]{$+$};
\draw[thick, ->] (12, 0) -- (16,0) (12, 0) -- (14,0) node[below]{$-$} node[above]{$+$};
\draw[thick, ->] (0,0) -- (6,0) (0,0) -- (3, 0) node[below]{$-$} node[above]{$+$};
\draw[thick, ->] (6,0) -- (12, 0) (6,0) -- (9, 0) node[below]{$-$} node[above]{$+$};
\draw[fill = black] (0,0) circle [radius = 0.08]
			(6,0) circle [radius = 0.08]
			(12,0) circle[radius = 0.08];
\node at (-0.6, -0.7) {$x_0 = -1$};
\node at (6.05, -0.7) {$x_1$};
\node at (12.5, -0.7) {$x_2 = 1$};

\draw[thick, dotted] (0,3.5) arc [radius=3.5, start angle=90, end angle= 270]
			(12,3.5) arc [radius=3.5, start angle=90, end angle= -90];
\draw[thick, dotted] (0, 3.5) -- (12, 3.5) (0, -3.5) -- (12, -3.5);
\node at (14,3.5) {$U_{[-1, 1]}$};

\draw[thick, ->] (0,0) to [out = 60, in = 180] (3,2.5) node[below]{$-$} node[above]{$+$}; \node at (1.5, 2.7) {$\Sigma_1^+$};
\draw[thick, ->] (0,0) to [out = -60, in = -180] (3,-2.5) node[below]{$-$} node[above]{$+$}; \node at (1.5, -2.7) {$\Sigma_1^-$};
\draw[thick, ->] (6,0) to [out = 60, in = 180] (9,2.5) node[below]{$-$} node[above]{$+$}; \node at (10.5, 2.7) {$\Sigma_2^+$};
\draw[thick, ->] (6,0) to [out = -60, in = -180] (9,-2.5) node[below]{$-$} node[above]{$+$};\node at (10.5, -2.7) {$\Sigma_2^-$};
\draw[thick](3, 2.5) to [out = 0, in = 120] (6,0) 	(3, -2.5) to [out = 0, in = -120] (6,0)
		(9, 2.5) to [out = 0, in = 120] (12,0) (9, -2.5) to [out = 0, in = -120] (12,0);

\end{tikzpicture}
\caption{Opening of lenses, $k = 1$. The signs indicate the orientation of the curves: the $+$ side is the left side of the curve and $-$ the right.}
\label{fig:lenses}
\end{center}
\end{figure}

We use these contours in our next transformation.

\begin{definition}\label{def:s}
For $z\notin \Sigma:=\cup_{j=1}^{k+1}(\Sigma_j^+\cup\Sigma_j^-)\cup \R$, let

\begin{equation} \label{eq:T_to_S}
S(z)=\begin{cases}
T(z), & \mathrm{outside\ of \ the\  lenses}\\
T(z)\begin{pmatrix}
1 & 0\\
-f_t(z)^{-1} e^{-Nh_s(z)} & 1
\end{pmatrix}, & \mathrm{top \ part \ of \ the \ lenses}\\
T(z)\begin{pmatrix}
1 & 0\\
f_t(z)^{-1} e^{Nh_s(z)} & 1
\end{pmatrix}, & \mathrm{bottom \ part \ of \ the \ lenses}
\end{cases}.
\end{equation}
\end{definition}

\begin{remark}
Note that $S$ depends on our choice of the contours $\Sigma$ {\rm{(}}as well as $s,t,$ and $N${\rm{)}}, but we suppress this in our notation. We also point out that as $f_t$ has zeroes at the singularities, the entries in the first column of $S(z)$ blow up when $z$ approaches a singularity from within the lens. Moreover, we see that we have  discontinuities at the points $\pm 1$. Thus the boundary values are no longer continuous on $\R$, but on $\R\setminus \lbrace x_j: j=0,...,k+1\rbrace$, where again $x_0=-1$ and $x_{k+1}=1$.
\end{remark}

Using the definition of $S$, the RHP for $T$, and the fact that

\begin{equation*}
\begin{pmatrix}
e^{-Nh_s(\lambda)} & f_t(\lambda) \\
0 & e^{Nh_s(\lambda)}\end{pmatrix}
=\begin{pmatrix}
1 & 0\\
{e^{Nh_s(\lambda)} }{f_t(\lambda)^{-1}}& 1
\end{pmatrix}
\begin{pmatrix}
0 & f_t(\lambda)\\
-f_t(\lambda)^{-1} & 0
\end{pmatrix}\begin{pmatrix}
1 & 0\\
{e^{-Nh_s(\lambda)}f_t(\lambda)^{-1}}& 1
\end{pmatrix}
\end{equation*}

\noindent it is simple to check what the Riemann--Hilbert problem for $S$ should be; we omit the proof.

\begin{lemma}\label{le:srhp}
$S$ is the unique solution to the following Riemann--Hilbert problem:

\begin{itemize}[leftmargin=0.5cm]
\item[1.] $S:\C\setminus \Sigma\to \C^{2\times 2}$ is analytic.

\item[2.] $S$ has continuous boundary values on $\Sigma\setminus \lbrace x_j\rbrace_{j=0}^{k+1}$ and they are related by the jump conditions

\begin{equation} \label{eq:S_jump_lenses}
S_+(\lambda)=S_-(\lambda)\begin{pmatrix}
1 & 0\\
f_t(\lambda)^{-1}e^{\mp N h_s(\lambda)} & 1
\end{pmatrix}, \quad \lambda\in \cup_{j=1}^{k+1}\Sigma_j^{\pm}\setminus \lbrace x_l\rbrace_{l=0}^{k+1},
\end{equation}

\begin{equation} \label{eq:S_jump_main}
S_+(\lambda)=S_-(\lambda)\begin{pmatrix}
0 & f_t(\lambda)\\
-f_t(\lambda)^{-1} & 0
\end{pmatrix}, \quad \lambda\in(-1,1)\setminus \lbrace x_j\rbrace_{j=1}^{k},
\end{equation}

\noindent and

\begin{equation}
S_+(\lambda)=S_-(\lambda)\begin{pmatrix}
1 & f_t(\lambda)e^{N(g_{s,+}(\lambda)+g_{s,-}(\lambda)-\ell_s-V_s(\lambda))}\\
0 & 1
\end{pmatrix}, \quad \lambda\in\R\setminus[-1,1].
\end{equation}

In \eqref{eq:S_jump_lenses} the $\mp$ and $\pm$ notation means that we have $e^{-Nh_s}$ in the jump matrix when we cross $\Sigma_j^+$ and $e^{Nh_s}$ when we cross $\Sigma_j^-$.

\item[3.] $S(z)=I+\mathcal{O}(|z|^{-1})$ as $z\to\infty$.

\item[4.] For $j=1,...,k$, $S(z)$ is bounded as $z\to x_j$ from outside of the lenses, but when $z\to x_j$ from inside of the lenses,

\begin{equation}\label{eq:srhpsing}
S(z)=\begin{pmatrix}
\mathcal{O}\big(|z-x_j|^{-\beta_j}\big) & \mathcal{O}(1)\\
\mathcal{O}\big(|z-x_j|^{-\beta_j}\big) & \mathcal{O}(1)
\end{pmatrix}.
\end{equation}

\noindent Moreover, $S$ is bounded at $\pm 1$.

\end{itemize}
\end{lemma}

We are now in a situation where if we are on one of the $\Sigma_j^\pm$ or on $\R\setminus[-1,1]$ and not close to one of the points $\pm 1$ or $x_j$, then the distance of the jump matrix from the identity matrix is exponentially small in $N$. We thus need to do something close to the points $\pm 1$ and $x_j$ as well as on the interval $(-1,1)$ to get a small norm problem, i.e. one that can be solved in terms of a Neumann series.

The way to proceed here is to construct functions which are solutions to approximations of the Riemann-Hilbert problem where we expect the approximations to be good if we are close to one of the points $\pm 1$ or $x_j$, or then alternatively when we are far away from them and we expect the approximate problem related to the behavior on $(-1,1)$ to determine the behavior of $S$. We then construct an ansatz to the original problem in terms of these approximations. This will lead to a small norm problem.

These approximations are often called parametrices, and we will start with the solution far away from the points $\pm 1$ and $x_j$. This case is often called the global parametrix.

\subsection{The global parametrix}\label{sec:global} Our goal is to find a function $P^{(\infty)}(z)$ such that it has the same jumps as $S(z)$ across $(-1,1)$, is analytic elsewhere, and has the correct behavior at infinity. We won't go into great detail about how such problems are solved, but we will build on  similar problems solved in \cite[Section 4.2]{krasovsky} (see also for example \cite[Section 5]{kmlvav}). We will simply state the result here and sketch a proof in Appendix \ref{app:global}. Later on we will need some regularity properties of the solution considered here so we will state and prove the relevant facts here.

We now define our global parametrix.

\begin{definition}\label{def:global}
Let us write for $z\notin(-\infty,1]$

\begin{equation}\label{eq:rdef}
r(z)=(z-1)^{1/2}(z+1)^{1/2}
\end{equation}

\noindent and

\begin{equation}\label{eq:adef}
a(z)=\frac{(z-1)^{1/4}}{(z+1)^{1/4}},
\end{equation}

\noindent where the powers are taken according to the principal branch. Then for $t\in[0,1]$ and $z\notin(-\infty,1]$, let

\begin{equation}\label{eq:ddef}
\mathcal{D}_t(z)=(z+r(z))^{-\mathcal{A}}\exp\left[\frac{r(z)}{2\pi}\int_{-1}^1 \frac{\mathcal{T}_t(\lambda)}{\sqrt{1-\lambda^2}}\frac{1}{z-\lambda}d\lambda\right]\prod_{j=1}^k(z-x_j)^{\beta_j/2}
\end{equation}

\noindent where $\mathcal{A}=\sum_{j=1}^k \beta_j/2$ and the powers are according to the principal branch. Finally, for $z\notin(-\infty,1]$ and $t\in[0,1]$, define the global parametrix

\begin{equation}\label{eq:global}
P^{(\infty)}(z)=P^{(\infty)}(z,t)=\frac{1}{2}\mathcal{D}_t(\infty)^{\sigma_3}\begin{pmatrix}
a(z)+a(z)^{-1} & -i(a(z)-a(z)^{-1})\\
i(a(z)-a(z)^{-1}) & a(z)+a(z)^{-1}
\end{pmatrix}\mathcal{D}_t(z)^{-\sigma_3},
\end{equation}

\noindent where $\mathcal{D}_t(\infty)=\lim_{z\to\infty}\mathcal{D}_t(z)=2^{-\mathcal{A}}e^{\frac{1}{2\pi}\int_{-1}^1 \frac{\mathcal{T}_t(\lambda)}{\sqrt{1-\lambda^2}}d\lambda}$.
\end{definition}

\begin{remark}\label{rem:globalholo}
It's simple to check that $r$ and $a$ are continuous across $(-\infty,-1)$ so they can be analytically continued to $\C\setminus[-1,1]$. Using the fact that $r(\lambda)$ is negative for $\lambda<-1$, one can check that also $\mathcal{D}_t$ is continuous across $(-\infty,-1)$, so in fact $P^{(\infty)}$ is analytic in $\C\setminus[-1,1]$.

We also point out that as $\mathcal{T}_0(\lambda)=0$  {\rm{(}}recall \eqref{eq:mathcaltt}{\rm{)}} we can also write

\begin{equation}\label{eq:globalalt}
P^{(\infty)}(z,t)=e^{\frac{\sigma_3}{2\pi}\int_{-1}^1 \frac{\mathcal{T}_t(\lambda)}{\sqrt{1-\lambda^2}}d\lambda}P^{(\infty)}(z,0)e^{-\sigma_3\frac{r(z)}{2\pi}\int_{-1}^1 \frac{\mathcal{T}_t(\lambda)}{\sqrt{1-\lambda^2}}\frac{d\lambda}{z-\lambda}}.
\end{equation}
\end{remark}

The relevance of this parametrix stems from the following lemma.
\begin{lemma}\label{le:global}
For each $t\in[0,1]$, $P^{(\infty)}(\cdot) = P^{(\infty)}(\cdot, t)$ satisfies the following Riemann--Hilbert problem.

\begin{itemize}[leftmargin=0.5cm]
\item[1.] $P^{(\infty)}:\C\setminus [-1,1]\to \C^{2\times 2}$ is analytic.
\item[2.] $P^{(\infty)}$ has continuous boundary values on $(-1,1)\setminus \lbrace x_j\rbrace_{j=1}^k$, and satisfies the jump condition

\begin{equation}\label{eq:globaljump}
P^{(\infty)}_+(\lambda)=P^{(\infty)}_-(\lambda)\begin{pmatrix}
0 & f_t(\lambda)\\
-f_t(\lambda)^{-1} & 0
\end{pmatrix}, \qquad \lambda\in(-1,1)\setminus\lbrace x_j\rbrace_{j=1}^k.
\end{equation}

\item[3.] As $z\to \infty$,

\begin{equation}\label{eq:globalnorm}
P^{(\infty)}(z)=I+\mathcal{O}(|z|^{-1}).
\end{equation}
\end{itemize}

\end{lemma}

See Appendix \ref{app:global} for a proof.
Later on, we will need some estimates on the regularity of the Cauchy transform appearing in \eqref{eq:ddef} near the interval $[-1,1]$. The fact we need is the following one.

\begin{lemma}\label{le:cauchybounds}
The function

\begin{equation*}
z\mapsto r(z)\int_{-1}^1\frac{\mathcal{T}_t(\lambda)}{\sqrt{1-\lambda^2}}\frac{1}{z-\lambda}d\lambda
\end{equation*}

\noindent is bounded uniformly in $t\in[0,1]$ and $z$ in a small enough neighborhood of $[-1,1]$. Moreover, if in a neighborhood of $[-1,1]$, $\Tree$ is a real polynomial of fixed degree, and if we restrict its coefficients to be in some bounded set, then we have uniform boundedness of the above function in the coefficients of $\Tree$ as well.
\end{lemma}

\begin{proof}
Let us fix a neighborhood of $[-1,1]$ such that for all $t\in[0,1]$, $\mathcal{T}_t$ is analytic in the closure of this neighborhood (this exists by similar reasoning as in the beginning of Section \ref{sec:trans2}). Now write

\begin{align*}
\int_{-1}^1\frac{\mathcal{T}_t(\lambda)}{\sqrt{1-\lambda^2}}\frac{1}{z-\lambda}d\lambda&=\int_{-1}^1\frac{\mathcal{T}_t(\lambda)-\mathcal{T}_t(z)}{z-\lambda}\frac{1}{\sqrt{1-\lambda^2}}d\lambda+\mathcal{T}_t(z)\int_{-1}^{1}\frac{1}{\sqrt{1-\lambda^{2}}}\frac{1}{z-\lambda}d\lambda.
\end{align*}

As $\mathcal{T}_t$ is analytic, the first term is of order $\mathcal{O}(\sup_{t\in[0,1]}||\mathcal{T}_t'||_\infty)$ (the prime referring to the $z$-variable and the sup-norm is over $z$ in the neighborhood we are considering) which is a finite constant depending on our neighborhood of $[-1,1]$ and the function $\mathcal{T}$. In the polynomial case, one can easily check that it is bounded uniformly in the coefficients when they are restricted to a compact set. The second integral can be calculated exactly:

\begin{align*}
\int_{-1}^1 \frac{1}{\sqrt{1-\lambda^2}}\frac{1}{z-\lambda}d\lambda&=\frac{\pi}{r(z)}.
\end{align*}

This can be seen for example by expanding the Cauchy kernel for large $|z|$ as a geometric series. The integrals resulting from this are simple to calculate and one can then also calculate the remaining sum exactly. The resulting quantity agrees with $\pi/r(z)$ on $(1,\infty)$ so by analyticity, the statement holds. The claim now follows from the uniform boundedness of $\mathcal{T}_t$ (for which the uniform boundedness in the polynomial case is again easy to check).
\end{proof}

\subsection{Local parametrices near the singularities}\label{sec:locals}

We now wish to find functions approximating $S(z)$ well near the points $x_j$. We will thus look for functions that satisfy the same jump conditions as $S(z)$ in some fixed neighborhoods of the points $x_j$ for $j=1,...,k$, but we will also want these approximations to be consistent with the global approximation, so we will replace a normalization at infinity with a matching condition, where we demand that the two approximations are close to each other on the boundary of the neighborhood we are looking at at. Our argument is built on \cite[Section 4.3]{krasovsky}, which in turn relies on \cite[Section 4]{vanlessen}. Again, we state the relevant facts here and give some further details in Appendix \ref{app:locals}.

In this case, we will have to introduce a bit more notation before defining our actual object. We first introduce a change of coordinates that will blow up in a neighborhood of a singularity in a good way.

\begin{definition}\label{def:zeta}
Fix some $\delta>0$ $($independent of $N$, $s$, and $t)$. Let us write $U_{x_j}$ for the open $\delta$-disk surrounding $x_j$. We assume that $\delta$ is small enough that the following conditions are satisfied:

\begin{itemize}
\item[i)] $|x_i-x_j|>3\delta$ for $i\neq j$.
\item[ii)] $|x_j\pm 1|>3\delta$ for all $j\in\lbrace 1,...,k\rbrace$.
\item[iii)] For all $j$, $U_{x_j}'$ -- the open $3\delta/2$-disk around $x_j$ -- is contained in $U$, which is some neighborhood of $\R$ into which $d$ has an analytic continuation $($see e.g. Definition \ref{def:hdef}$)$.
\end{itemize}

For $z\in U_{x_j}'$, let

\begin{equation}\label{eq:zetadef}
\zeta_s(z)=\pi N \int_{x_j}^z\left[\frac{2}{\pi}(1-s)+sd(w)\right]\sqrt{1-w^2}dw,
\end{equation}

\noindent where the root is according to the principal branch, and the integration contour does not leave $U_{x_j}'$.
\end{definition}

\begin{remark}\label{rem:zetainj}
The reason for introducing the two neighborhoods $U_{x_j}$ and $U_{x_j}'$, is that we will want the local parametrices to be analytic functions approximately agreeing with $P^{(\infty)}$ on the boundary of $U_{x_j}$, but to ensure that they behave nicely near the boundary, we will construct them such that they are analytic in $U_{x_j}'$.

We also point out that by taking $\delta$ smaller if needed, $\zeta_s$ can be seen to be injective as $d$ is positive on $[-1,1]$. More precisely, we see that $\zeta_s'(x_j)>c N$ for some constant $c$ which is independent of $s$ $($but not necessarily of $\delta)$ and $|\zeta_s''(z)|\leq C N$ uniformly in $z\in U_{x_j}'$ for some $C>0$ independent of $s$ $($but not necessarily of $\delta)$. From this one sees that $\zeta_s$ is injective in a small enough $(N$- and $s$-independent$)$ neighborhood of $x_j$. 
\end{remark}

In addition to this change of coordinates, we will need to add further jumps to make our jump contour more symmetric, in order to obtain an approximate problem with a known solution.

\begin{definition}\label{def:wandphi}
For $z\in U_{x_j}'$, let

\begin{align}\label{eq:wdef}
W_j(z)&=W_j(z,t)\\
\notag &=e^{\mathcal{T}_t(z)/2}\prod_{l=1}^{j-1}(z-x_l)^{\beta_l/2}\prod_{l=j+1}^k (x_l-z)^{\beta_l/2}\times\begin{cases}
(z-x_j)^{\beta_j/2}, & |\mathrm{arg}\ \zeta_s(z)|\in(\pi/2,\pi)\\
(x_j-z)^{\beta_j/2}, & |\mathrm{arg}\ \zeta_s(z)|\in(0,\pi/2)
\end{cases},
\end{align}

\noindent where the roots are principal branch roots. Moreover, let

\begin{equation}\label{eq:phidef}
\phi_s(z)=\begin{cases}
\frac{h_s(z)}{2}, & \mathrm{Im}(z)>0\\
-\frac{h_s(z)}{2}, & \mathrm{Im}(z)<0
\end{cases}.
\end{equation}
\end{definition}

The precise form of $\zeta_s$ will be important for us to be able to see that the local parametrices indeed approximately agree with $P^{(\infty)}$ on the boundary of $U_{x_j}$. We also point out that for small enough $\delta$, $\zeta_s$ is one-to-one, and it preserves the real axis (along with the orientation of the plane as it's conformal).

We also point out that $W_j$ is almost identical to $f_t^{1/2}$, apart from the fact that it introduces some further branch cuts to it: along the imaginary axis in the $\zeta_s$-plane, as well as on the real axis (recall that $f_t$ has no branch cut along the real axis). These further branch cuts are useful in transforming the Riemann-Hilbert problem for the parametrix into one with certain constant jump matrices along a very special contour. This problem has been studied in \cite{vanlessen}.

We are now able to clarify our choice of the contours $\Sigma_j^{\pm}$ apart from the behavior near the end points $\pm 1$.

\begin{definition}\label{def:lenses2}
Let $(\Sigma_l^{\pm})_l$ be such that

\begin{equation}\label{eq:sigmaj1}
\zeta_s\left(\Sigma_{j-1}^{\pm}\cap U_{x_j}'\right)=\left[e^{\pm 3\pi i/4}\times[0,\infty)\right]\cap \zeta_s\left(U_{x_j}'\right)
\end{equation}

\noindent and

\begin{equation}\label{eq:sigmaj2}
\zeta_s\left(\Sigma_{j}^{\pm}\cap U_{x_j}'\right)=\left[e^{\pm \pi i/4}\times[0,\infty)\right]\cap \zeta_s\left(U_{x_j}'\right).
\end{equation}

Outside of $U_{x_j}'$ $($apart from close to $\pm 1)$, we take $(\Sigma_l^{\pm})_l$ to be smooth, without self-intersections and the distance between them and the real axis to be bounded away from zero and of order $\delta$, and such that the contours are contained in $U$ -- the neighborhood of $\R$ into which $d$ has an analytic continuation. For an illustration, see Figure \ref{fig:cont1}.
\end{definition}

Using the injectivity of $\zeta_s$ we argued in Remark \ref{rem:zetainj} and the Koebe quarter theorem, it is immediate that $\Sigma_j^\pm$ and $\Sigma_{j-1}^\pm$ are well defined for large enough $N$ and small enough $\delta$ (large and small enough being independent of $s$).

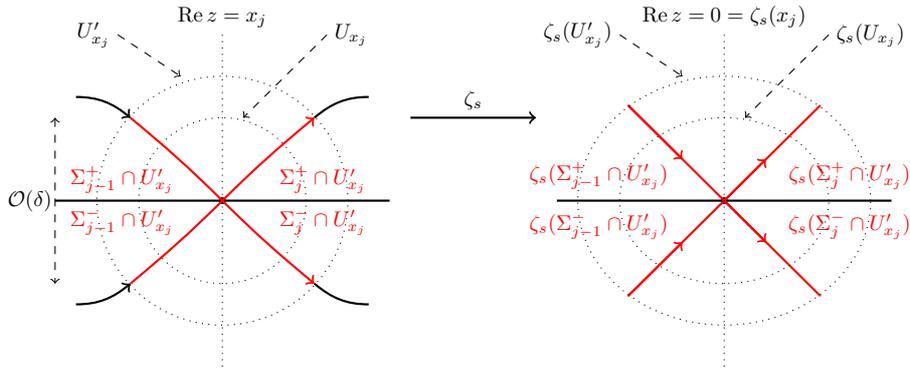
\begin{figure}[ht]
\begin{center}
\begin{tikzpicture}[scale=0.55, every node/.style={scale=0.8}]

\draw[thick] (-4, 0) -- (4, 0);
\draw[fill = black] (0,0) circle[radius = 0.07];
\draw[dotted] (0, 4) node[above]{$\Re z = x_j$} -- (0, -4);
\draw[dotted] (0,0) circle [radius = 2] circle[radius = 3];
\draw[thick,red] (-2.2, 2) to [out = -40, in = 135] (0,0); \node at (-2.4, 0.5) {\color{red}$\Sigma_{j-1}^+\cap U'_{x_j}$};
\draw[thick, ->] (-3.5, 2.5) to [out = 0, in = 140] (-2.2, 2);
\draw[thick, red] (-2.2, -2) to [out = 40, in = -135] (0,0); \node at (-2.4, -0.5) {\color{red}$\Sigma_{j-1}^-\cap U'_{x_j}$};
\draw[thick, ->] (-3.5, -2.5) to [out = 0, in = -140] (-2.2, -2);
\draw[thick] (2.2, 2) to [out = 40, in = 180] (3.5,2.5);
\draw[thick, ->, red] (0, 0) to [out = 45, in = 220] (2.2, 2); \node at (2.4, 0.5) {\color{red}$\Sigma_{j}^+\cap U'_{x_j}$};
\draw[thick] (2.2, -2) to [out = -40, in = -180] (3.5,-2.5);
\draw[thick, ->, red] (0, 0) to [out = -45, in = -220] (2.2, -2); \node at (2.4, -0.5) {\color{red}$\Sigma_{j}^-\cap U'_{x_j}$};
\draw[dashed, <->] (-4, 2) -- (-4, -2); \node at (-4.6, 0) {$\mathcal{O}(\delta)$};
\draw[dashed, ->] (-2.5, 4) node[left] {$U'_{x_j}$} -- (-1, 3);
\draw[dashed, ->] (2.5, 4) node[right] {$U_{x_j}$} -- (0.5, 2);

\draw[thick, ->] (4.5, 2) -- (6, 2) node[above]{$\zeta_s$} -- (7.5, 2);

\draw[thick] (8,0) -- (16,0);
\draw[fill = black] (12,0) circle [radius = 0.07];
\draw[thick, ->, red] (9.7, 2.3) -- (11, 1);  \node at (9, 0.6) {\color{red}$\zeta_s(\Sigma_{j-1}^+ \cap U'_{x_j})$};
\draw[thick, ->, red] (9.7, -2.3) -- (11, -1); \node at (9, -0.6) {\color{red}$\zeta_s(\Sigma_{j-1}^- \cap U'_{x_j})$};
\draw[thick, ->, red] (12, 0) -- (13, 1); \draw[thick, ->, red] (12,0) -- (13,-1); \node at (15, 0.6) {\color{red}$\zeta_s(\Sigma_{j}^+ \cap U'_{x_j})$};
\draw[thick, red] (9.7, 2.3) -- (14.3, -2.3) (9.7, -2.3) -- (14.3, 2.3); \node at (15, -0.6) {\color{red}$\zeta_s(\Sigma_{j}^- \cap U'_{x_j})$};
\draw[dotted] (12, -4) -- (12, 4) node[above] {$\Re z = 0 = \zeta_s(x_j)$};
\draw[dotted] (12,0) ellipse (3.5 and 3) (12,0) ellipse (2.5 and 2) ;

\draw[dashed, ->] (9.5, 4) node[left] {$\zeta_s(U'_{x_j})$} -- (11, 3);
\draw[dashed, ->] (14.5, 4) node[right] {$\zeta_s(U_{x_j})$} -- (12.5, 2);

\end{tikzpicture}

\caption{Choice of the jump contours near the singularities.}
\label{fig:cont1}
\end{center}
\end{figure}

We still need one further ingredient before defining our local parametrix. This is a solution to a model Riemann-Hilbert problem -- a problem where the jump contours and matrices are particularly simple and a solution can be given explicitly in terms of suitable special functions. We will give a rather compact definition here with a more detailed description in Appendix \ref{app:locals}.

\begin{definition}\label{def:psi}
Let us denote by Roman numerals the octants of the complex plane -- so we write $\mathrm{I}=\lbrace r e^{i\theta}: r>0,\theta\in(0,\pi/4)\rbrace$ and so on. Denote by $\Gamma_l$ the boundary rays of these octants: for $ 1\le l \le 8$, $\Gamma_l=\{ re^{i\frac{\pi}{4}(l-1)}, r >0\}$, oriented as in Figure \ref{fig:modelrhp}. 

For $\zeta\in \mathrm{I}$, let

\begin{equation}\label{eq:psiI}
\Psi(\zeta)=\frac{1}{2}\sqrt{\pi \zeta}\begin{pmatrix}
H_{\frac{\beta_j+1}{2}}^{(2)}(\zeta) & -i H_{\frac{\beta_j+1}{2}}^{(1)}(\zeta)\\
H_{\frac{\beta_j-1}{2}}^{(2)}(\zeta) & -iH_{\frac{\beta_j-1}{2}}^{(1)}(\zeta)
\end{pmatrix} e^{-\left(\frac{\beta_j}{2}+\frac{1}{4}\right)\pi i \sigma_3},
\end{equation}

\noindent where $H_\nu^{(i)}$ are Hankel functions and the root is according to the principal branch. In other octants, $\Psi$ satisfies the following Riemann-Hilbert problem:

\begin{itemize}[leftmargin=0.5cm]
\item[1.] $\Psi: \C\setminus \cup_{l=1}^8 \overline{\Gamma_l}\to \C^{2\times 2}$ is analytic.

\item[2.] $\Psi$ has continuous boundary values on each $\Gamma_l$ and satisfies the following jump condition $($again for the orientation, see Figure \ref{fig:modelrhp}$)$ $\Psi_+(\zeta)=\Psi_-(\zeta)K(\zeta)$ for $\zeta\in \cup_{l=1}^8 \Gamma_l$, where

\begin{equation}\label{eq:psijump}
K(\zeta)=\begin{cases}
\begin{pmatrix}
0 & 1\\
-1 & 0
\end{pmatrix}, &\zeta\in \Gamma_1\cup \Gamma_5\\
\begin{pmatrix}
1& 0\\
e^{-\pi i \beta_j} & 1
\end{pmatrix}, & \zeta\in \Gamma_2\cup\Gamma_6\\
e^{\pi i \frac{\beta_j}{2}\sigma_3}, & \zeta\in \Gamma_3\cup \Gamma_7\\
\begin{pmatrix}
1& 0\\
e^{\pi i \beta_j} & 1
\end{pmatrix}, & \zeta\in \Gamma_4\cup\Gamma_8
\end{cases}
\end{equation}
\end{itemize}
\end{definition}

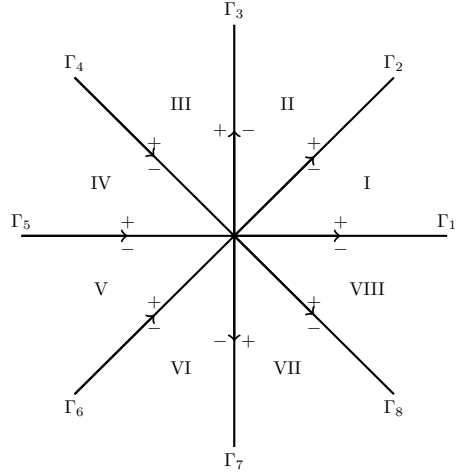
\begin{figure}[ht]
\begin{center}
\begin{tikzpicture}[scale=0.7, every node/.style={scale=0.7}]

\draw[thick, ->] (-4, 0) -- (-2,0) node[above] {$+$} node[below]{$-$};
\draw[thick, ->] (0, 0) -- (2,0) node[above] {$+$} node[below]{$-$};
\draw[thick, ->] (-3, 3) -- (-1.5, 1.5) node[below] {$-$} node[above]{$+$};
\draw[thick, ->] (0, 0) -- (1.5, -1.5) node[below] {$-$} node[above]{$+$};
\draw[thick, ->] (-3, -3) -- (-1.5, -1.5) node[above] {$+$} node[below]{$-$};
\draw[thick, ->] (0, 0) -- (1.5, 1.5) node[above] {$+$} node[below]{$-$};
\draw[thick, ->] (0, 0) -- (0, 2) node[left]{$+$} node[right]{$-$};
\draw[thick, ->] (0, 0) -- (0, -2) node[left]{$-$} node[right]{$+$};
\draw[thick] (-4,0) -- (4, 0) (0, -4) -- (0, 4) (-3, 3) -- (3, -3) (-3, -3) -- (3, 3);

\node at (2.5, 1) {I};
\node at (1, 2.5) {II};
\node at (-1, 2.5) {III};
\node at (-2.5, 1) {IV};
\node at (-2.5, -1) {V};
\node at (-1, -2.5) {VI};
\node at (1, -2.5) {VII};
\node at (2.5, -1) {VIII};

\node[above] at (4,0) {$\Gamma_1$};
\node[above] at (3,3) {$\Gamma_2$};
\node[above] at (0,4) {$\Gamma_3$};
\node[above] at (-3,3) {$\Gamma_4$};
\node[above] at (-4,0) {$\Gamma_5$};
\node[below] at (-3,-3) {$\Gamma_6$};
\node[below] at (0,-4) {$\Gamma_7$};
\node[below] at (3,-3) {$\Gamma_8$};

\end{tikzpicture}
\caption{Jump contour of the model RHP}
\label{fig:modelrhp}
\end{center}
\end{figure}

Uniqueness of such a $\Psi$ can be argued in a similar manner as usual. First of all, one can check that for $\zeta\in \mathrm{I}$, $\det \Psi(\zeta)=1$. As the jump matrices all have unit determinant, $\det \Psi$ is analytic in $\C\setminus \lbrace 0\rbrace$, so  $\det\Psi(\zeta)=1$ for $\zeta\in \C$ (one can check that $\zeta=0$ is a removable singularity). Consider then some other solution to the problem, say $\widetilde{\Psi}$. As $\det \Psi=\det\widetilde{\Psi}=1$,  $\Psi(\zeta) \widetilde{\Psi}(\zeta)^{-1}$ is analytic in $\C\setminus \cup_{l=1}\overline{\Gamma_l}$ and equals $I$ for $\zeta\in \mathrm{I}$. Again it follows from the jump structure that $\Psi(\zeta) \widetilde{\Psi}(\zeta)^{-1}$ continues analytically to $\C\setminus \lbrace 0\rbrace$ so it must equal $I$ everywhere. For an explicit description of the solution, see Appendix \ref{app:locals}.

The local parametrices will then be formulated in terms of this function $\Psi$, a coordinate change given by $\zeta_s$, the function $W_j$, and an analytic ($\C^{2\times 2}$-valued) ``compatibility matrix" $E$, which is needed for the matching condition to be satisfied. We now make the relevant definitions.

\begin{definition}\label{def:local}
For $z\in U_{x_j}'\cap \lbrace \mathrm{Im}(z)>0\rbrace$, write

\begin{equation}\label{eq:edef1}
E(z)=E(z,t,s)=P^{(\infty)}(z,t)W_j(z,t)^{\sigma_3} e^{N\phi_{s,+}(x_j)\sigma_3}e^{-(1\mp \beta_j)\pi i \sigma_3/4}\frac{1}{\sqrt{2}}\begin{pmatrix}
1 & i\\
i & 1
\end{pmatrix}
\end{equation}

\noindent where the $-$ sign is in the domain $\lbrace z\in \C: \mathrm{arg}(\zeta_s(z))\in (0,\pi/2)\rbrace$ and the $+$ sign is in the domain $\lbrace z\in \C: \mathrm{arg}(\zeta_s(z))\in (\pi/2,\pi)\rbrace$. For $z\in U_{x_j}'\cap \lbrace \mathrm{Im}(z)<0\rbrace$, write

\begin{equation}\label{eq:edef2}
E(z)=P^{(\infty)}(z)W_j(z)^{\sigma_3} \begin{pmatrix}
0 & 1\\
-1 & 0
\end{pmatrix}e^{N\phi_{s,+}(x_j)\sigma_3}
e^{-(1\mp \beta_j)\pi i \sigma_3/4}\frac{1}{\sqrt{2}}\begin{pmatrix}
1 & i\\
i & 1
\end{pmatrix}
\end{equation}

\noindent where $-$ sign is in the domain $\lbrace z\in \C: \mathrm{arg}(\zeta_s(z))\in (-\pi/2,0)\rbrace$ and the $+$ sign is in the domain $\lbrace z\in \C: \mathrm{arg}(\zeta_s(z))\in (-\pi,-\pi/2)\rbrace$.

\vspace{0.3cm}

Finally, for $z\in U_{x_j}'\setminus \Sigma$, let

\begin{equation}\label{eq:localdef}
P^{(x_j)}(z)=P^{(x_j)}(z,s,t)=E(z,s,t)\Psi(\zeta_s(z))W_{j}(z,t)^{-\sigma_3}e^{-N\phi_s(z)\sigma_3}.
\end{equation}
\end{definition}

\begin{remark}\label{rem:eholo}
Using \eqref{eq:wdef} -- the definition of $W_j$ -- as well as \eqref{eq:globaljump} -- the jump conditions of $P^{(\infty)}$, one can check that $E$ has no jumps in $U_{x_j}'$. Moreover, using the behavior of both functions near $x_j$, one can check that $E$ does not have an isolated singularity at $x_j$, so $E$ is analytic in $U_{x_j}'$.

We also point out that it follows directly from the definitions, i.e. \eqref{eq:wdef}, \eqref{eq:edef1}, \eqref{eq:edef2}, and \eqref{eq:localdef}, that for $z\in U_{x_j}'\setminus \Sigma$

\begin{equation}\label{eq:localalt}
P^{(x_j)}(z,t,s)=P^{(\infty)}(z,t)e^{\frac{1}{2}\mathcal{T}_t(z)\sigma_3}\left[P^{(\infty)}(z,0)\right]^{-1}P^{(x_j)}(z,0,s)e^{-\frac{1}{2}\mathcal{T}_t(z)\sigma_3}.
\end{equation}
\end{remark}

The main claim about $P^{(x_j)}$ is the following, whose proof we sketch in Appendix \ref{app:locals}.

\begin{lemma}\label{le:localrhp}
The function $P^{(x_j)}$ satisfies the following Riemann-Hilbert problem.

\begin{itemize}[leftmargin=0.5cm]
\item[1.] $P^{(x_j)}:U_{x_j}'\setminus \Sigma\to \C^{2\times 2}$ is analytic.
\item[2.] $P^{(x_j)}$ has continuous boundary values on $\Sigma\cap U_{x_j}'\setminus \lbrace x_j\rbrace$ and these satisfy the following jump conditions $($with the same orientation as for $S$ and same convention for the sign in $e^{\mp Nh_s(\lambda)})$: for $\lambda\in(U_{x_j}'\setminus \lbrace x_j\rbrace)\cap(\Sigma_{j-1}^+\cup\Sigma_{j-1}^-\cup \Sigma_j^+\cup\Sigma_j^{-1})$

\begin{equation}\label{eq:localjumplenses}
P^{(x_j)}_+(\lambda)=P^{(x_j)}_-(\lambda)\begin{pmatrix}
1 & 0\\
f_t(\lambda)^{-1} e^{\mp Nh_s(\lambda)} & 1
\end{pmatrix},
\end{equation}

\noindent and for $\lambda\in \R \cap U_{x_j}'\setminus \lbrace x_j\rbrace$

\begin{equation}\label{eq:localjumpr}
P^{(x_j)}_+(\lambda)=P^{(x_j)}_-(\lambda)\begin{pmatrix}
0 & f_t(\lambda)\\
-f_t(\lambda)^{-1} & 0
\end{pmatrix}.
\end{equation}
\item[3.] $P^{(x_j)}(z)$ is bounded as $z\to x_j$ from outside of the lenses, but when $z\to x_j$ from inside of the lenses

\begin{equation}\label{eq:localrhpsing}
P^{(x_j)}(z)=\begin{pmatrix}
\mathcal{O}(|z-x_j|^{-\beta_j}) & \mathcal{O}(1)\\
\mathcal{O}(|z-x_j|^{-\beta_j}) & \mathcal{O}(1)\\
\end{pmatrix}.
\end{equation}

\item[4.] For $z\in \partial U_{x_j}$

\begin{equation}\label{eq:localmatch}
P^{(x_j)}(z)\left[P^{(\infty)}(z)\right]^{-1}=I+\mathcal{O}(N^{-1}),
\end{equation}

\noindent where the $\mathcal{O}(N^{-1})$-term is a $2\times 2$ matrix whose entries are $\mathcal{O}(N^{-1})$ uniformly in $z,s,t,$ $\lbrace |x_i-x_j|\geq 3\delta \ \mathrm{for \ } i\neq j\rbrace$, and $\lbrace |1\pm x_j|\geq 3\delta \ \mathrm{for \ all \ } j\in\lbrace 1,...,k\rbrace\rbrace$. If in a neighborhood of $[-1,1]$, $\Tree$ is a real polynomial of fixed degree, the error is also uniform in the coefficients once they are restricted to some bounded set.
\end{itemize}
\end{lemma}

For our second differential identity, we will actually need more precise information about $P^{(x_j)}$ on $\partial U_{x_j}$. While we will only use it in the $\Tree=0$ case, it is not more difficult to formulate the result in the general case.

\begin{lemma}\label{le:2ndorderlocals}
For $z\in \partial U_{x_j}$

\begin{equation}\label{eq:2ndorderlocals}
P^{(x_j)}(z)\left[P^{(\infty)}(z)\right]^{-1}=I+\frac{\beta_j}{4\zeta_s(z)}E(z)\begin{pmatrix}
0 & 1+\frac{\beta_j}{2}\\
1-\frac{\beta_j}{2}& 0
\end{pmatrix}E(z)^{-1}+\mathcal{O}\left(N^{-2}\right),
\end{equation}

\noindent where the $\mathcal{O}(N^{-2})$-term is a $2\times 2$ matrix whose entries are $\mathcal{O}(N^{-2})$ uniformly in $z,s,$ and $\lbrace |x_i-x_j|\geq 3\delta \ \mathrm{for \ } i\neq j\rbrace$ and $\lbrace |1\pm x_j|\geq 3\delta \ \mathrm{for \ all \ } j\in\lbrace 1,...,k\rbrace\rbrace$.
\end{lemma}

The $t=0$, $s=0$ case of these results has been proven in \cite[Section 4.3]{krasovsky}, though without focus on the uniformity relevant to us. Due to this, we will again sketch a proof in Appendix \ref{app:locals}.

\subsection{Local parametrices at the edge of the spectrum} \label{sec:locale}
The reasoning here is similar to the previous section -- we wish to find a function approximating $S$ near the points $\pm 1$. We will do this by approximating the Riemann-Hilbert problem and imposing a matching condition. Our argument will follow \cite[Section 4.4]{krasovsky}, which in turn relies on \cite{dkmlvz}. We will focus on the approximation at $1$, as the one at $-1$ is analogous. Again we will provide a sketch of the relevant proofs in Appendix \ref{app:locale}. We will begin by introducing the relevant coordinate change in this case (analogous to $\zeta_s$ in the previous section).

\begin{definition}\label{def:xi}
Let $\delta>0$ satisfy the conditions of Definition \ref{def:zeta}. Denote by $U_{1}$ a $\delta$-disk around $1$ and $U_{1}'$ denote a $3\delta/2$-disk around $1$. We assume that $\delta$ is small enough that $d$ has an analytic extension to $U_{1}'$. Moreover, we assume $\delta$ is small enough -- though independent of $s$ -- so that with a suitable choice of the branch, the function

\begin{equation}\label{eq:xidef}
\xi_s(z)=\left[-\frac{3}{2}N\phi_s(z)\right]^{2/3}
\end{equation}

\noindent is analytic and injective in $U_1'$, for all $s\in[0,1]$.
\end{definition}

We will justify that this is indeed possible in Appendix \ref{app:locale}. This conformal coordinate change allows us to define what $\Sigma_{k+1}^\pm$ looks like near $1$. Let $\delta>0$ be small enough to satisfy the conditions of Definition \ref{def:xi} and so that $\mathcal{T}_t$ is analytic in $U_1'$ for all $t\in[0,1]$. We will define the local parametrix in $U_1'$ and impose the matching condition on $\partial U_1$. Let us thus define $\Sigma_{k+1}^\pm$ in $U_1'$.

\begin{definition}\label{def:sigmae}
Inside $U_1'$, let $\Sigma_{k+1}^\pm$ be such that

\begin{equation}\label{eq:sigmae}
\xi_s(\Sigma_{k+1}^\pm\cap U_1')=\left[e^{\pm 2\pi i/3}\times [0,\infty)\right]\cap \xi_s(U_1').
\end{equation}
\end{definition}

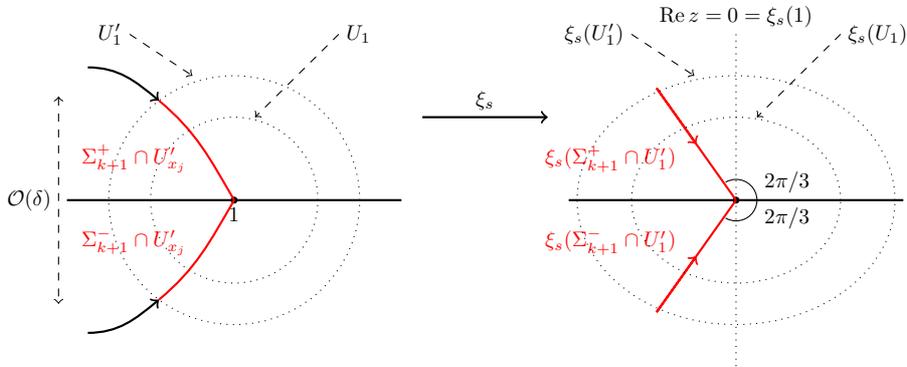
\begin{figure}[ht]
\begin{center}
\begin{tikzpicture}[scale=0.55, every node/.style={scale=0.8}]

\draw[thick] (-4, 0) -- (4, 0);
\draw[fill = black] (0,0) node[below]{$1$} circle[radius = 0.07];
\draw[dotted] (0,0) circle [radius = 2] circle[radius = 3];
\draw[thick,red] (-1.8, 2.4) to [out = -40, in = 120] (0,0); \node at (-2.4, 1) {\color{red}$\Sigma_{k+1}^+\cap U'_{x_j}$};
\draw[thick, ->] (-3.5, 3.2) to [out = 0, in = 140] (-1.8, 2.4);
\draw[thick,red] (-1.8, -2.4) to [out = 40, in = -120] (0,0); \node at (-2.4, -1) {\color{red}$\Sigma_{k+1}^-\cap U'_{x_j}$};
\draw[thick, ->] (-3.5, -3.2) to [out = 0, in = -140] (-1.8, -2.4);

\draw[dashed, <->] (-4.2, 2.5) -- (-4.2, -2.5); \node at (-4.9, 0) {$\mathcal{O}(\delta)$};
\draw[dashed, ->] (-2.5, 4) node[left] {$U'_{1}$} -- (-1, 3);
\draw[dashed, ->] (2.5, 4) node[right] {$U_{1}$} -- (0.5, 2);

\draw[thick, ->] (4.5, 2) -- (6, 2) node[above]{$\xi_s$} -- (7.5, 2);

\draw[thick] (8,0) -- (16,0);
\draw[fill = black] (12,0) circle [radius = 0.07];
\draw[thick, ->, red] (10.1, 2.7) -- (12,0) (10.1, 2.7) -- (11.05, 1.35);  \node at (9, 1) {\color{red}$\xi_s(\Sigma_{k+1}^+ \cap U'_{1})$};
\draw[thick, ->, red] (10.1, -2.7) -- (12,0) (10.1, -2.7) -- (11.05, -1.35); \node at (9, -1) {\color{red}$\xi_s(\Sigma_{k+1}^- \cap U'_{1})$};

\draw[dotted] (12, -4) -- (12, 4) node[above] {$\Re z = 0 = \xi_s(1)$};
\draw[dotted] (12,0) ellipse (3.8 and 3) (12,0) ellipse (2.5 and 2) ;
\draw (12.5,0) node[above right] {$2\pi / 3$} arc (0:120:0.5);
\draw (12.5,0) node[below right] {$2\pi / 3$} arc (0:-120:0.5);

\draw[dashed, ->] (9.5, 4) node[left] {$\xi_s(U'_{1})$} -- (11, 3);
\draw[dashed, ->] (14.5, 4) node[right] {$\xi_s(U_{1})$} -- (12.5, 2);

\end{tikzpicture}

\caption{Choice of the jump contours near the edge of the spectrum.}
\label{fig:cont}
\end{center}
\end{figure}

\begin{remark}
The angle $2\pi/3$ is slightly arbitrary here. In \cite{dkmlvz} the model Riemann-Hilbert problem relevant to us is constructed for a family of angle parameters $\sigma\in(\pi/3,\pi)$, and any angle here would work just as well for us, but we choose this for concreteness.

Also we point out that the above definition is fine as we know that $\xi_s$ is injective and we can apply the Koebe quarter theorem to ensure that the preimage of the rays is non-empty.
\end{remark}

We are now also in a position to define our local parametrix. As in the previous section, we need for this a solution to a certain model RHP considered in \cite{dkmlvz} as well as a function which is analytic in $U_{x_j}'$ which is required for the matching condition to hold.

\begin{definition}\label{def:locale}

Let us write $\mathrm{I}=\lbrace re^{i\theta}: r>0,\theta\in(0,2\pi/3)\rbrace$, $\mathrm{II}=\lbrace re^{i\theta}: r>0,\theta\in(2\pi /3,\pi)\rbrace$, $\mathrm{III}=\lbrace re^{i\theta}: r>0,\theta\in(-\pi,-2\pi/3)\rbrace$, and $\mathrm{IV}=\lbrace re^{i\theta}: r>0,\theta\in(-2\pi/3,0)\rbrace$. Then define

\begin{equation}\label{eq:qdef}
Q(\xi)=\begin{cases}
\begin{pmatrix}
\mathrm{Ai}(\xi) & \mathrm{Ai}(e^{4\pi i /3}\xi)\\
\mathrm{Ai}'(\xi) & e^{4\pi i /3} \mathrm{Ai}'(e^{4\pi i /3}\xi)
\end{pmatrix}e^{-\pi i \sigma_3/6}, & \xi \in \mathrm{I}\\
\begin{pmatrix}
\mathrm{Ai}(\xi) & \mathrm{Ai}(e^{4\pi i /3}\xi)\\
\mathrm{Ai}'(\xi) & e^{4\pi i /3} \mathrm{Ai}'(e^{4\pi i /3}\xi)
\end{pmatrix}e^{-\pi i \sigma_3/6}\begin{pmatrix}
1 & 0\\
-1 & 1
\end{pmatrix}, & \xi \in \mathrm{II}\\
\begin{pmatrix}
\mathrm{Ai}(\xi) & -e^{4\pi i /3}\mathrm{Ai}(e^{4\pi i /3}\xi)\\
\mathrm{Ai}'(\xi) &  -\mathrm{Ai}'(e^{4\pi i /3}\xi)
\end{pmatrix}e^{-\pi i \sigma_3/6}\begin{pmatrix}
1 & 0\\
1 & 1
\end{pmatrix}, & \xi \in \mathrm{III}\\
\begin{pmatrix}
\mathrm{Ai}(\xi) & -e^{4\pi i /3}\mathrm{Ai}(e^{4\pi i /3}\xi)\\
\mathrm{Ai}'(\xi) &  -\mathrm{Ai}'(e^{4\pi i /3}\xi)
\end{pmatrix}e^{-\pi i \sigma_3/6}, & \xi \in \mathrm{IV}
\end{cases},
\end{equation}

\noindent where $\mathrm{Ai}$ is the Airy function.

Morover, define another ``compatibility matrix"

\begin{equation}\label{eq:Fdef}
F(z)=F(z,t,s)=P^{(\infty)}(z,t)f_t(z)^{\sigma_3/2}e^{i\pi \sigma_3/4}\sqrt{\pi}\begin{pmatrix}
1 & -1\\
1 & 1
\end{pmatrix}\xi_s(z)^{\sigma_3/4} e^{-\pi i /12},
\end{equation}

\noindent where the roots are principal branch roots, and

\begin{equation}\label{eq:localedef}
P^{(1)}(z)=P^{(1)}(z,t,s)=F(z)Q(\xi_s(z))e^{-N\phi_s(z)\sigma_3}f_t(z)^{-\sigma_3/2}.
\end{equation}
\end{definition}

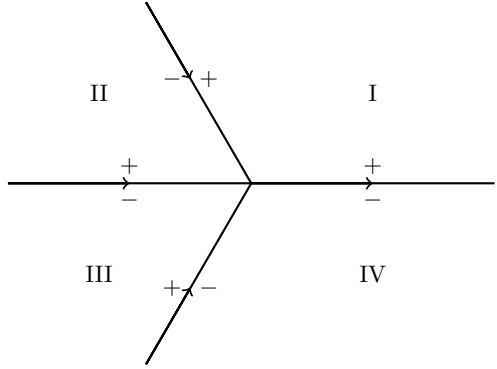
\begin{figure}[ht]
\begin{center}
\begin{tikzpicture}[scale=0.8, every node/.style={scale=0.9}]

\draw[thick, ->] (-4, 0) -- (0,0) (-4,0) -- (-2,0) node[above]{$+$} node[below]{$-$};
\draw[thick, ->] (0, 0) -- (4, 0) (0,0) -- (2,0) node[above]{$+$} node[below]{$-$};
\draw[thick, ->] (-1.73, 3) -- (0,0) (-1.73, 3) -- (-1, 1.73) node[right]{$+$} node[left]{$-$};
\draw[thick, ->] (-1.73, -3) -- (0,0) (-1.73, -3) -- (-1, -1.73) node[left]{$+$} node[right]{$-$};
\node at (2,1.5) {I};
\node at (-2.5,1.5) {II};
\node at (-2.5,-1.5) {III};
\node at (2,-1.5) {IV};
\end{tikzpicture}
\end{center}
\caption{Jump contour of $Q(\xi)$}
\end{figure}

\begin{remark}
Note that we can write

\begin{equation}\label{eq:localalt2}
P^{(1)}(z,t,s)=P^{(\infty)}(z,t)e^{\mathcal{T}_t(z)\sigma_3/2}\left[P^{(\infty)}(z,0)\right]^{-1}P^{(1)}(z,0,s) e^{-\mathcal{T}_t(z)\sigma_3/2}.
\end{equation}
\end{remark}

Again the relevant fact about this function is that it satisfies a suitable Riemann-Hilbert problem. Part of this is the fact that $F$ in \eqref{eq:Fdef} is an analytic function in $U_1'$. As before, we sketch the proof in Appendix \ref{app:locale}.

\begin{lemma}\label{le:localerhp}
The function $F$ from \eqref{eq:Fdef} is analytic in $U_1'$ and the function $P^{(1)}(z)$ satisfies the following Riemann-Hilbert problem.

\begin{itemize}[leftmargin=0.5cm]
\item[1.] $P^{(1)}(z)$ is analytic in $U_1'\setminus (\Sigma_{k+1}^+\cup \Sigma_{k+1}^-\cup \R)$.
\item[2.] For $\lambda\in (-1,1)\cap U_1'$, $P^{(1)}$  satisfies

\begin{equation}\label{eq:localej1}
P^{(1)}_+(\lambda)=P^{(1)}_-(\lambda)\begin{pmatrix}
0 & f_t(\lambda)\\
-f_t(\lambda)^{-1} & 0
\end{pmatrix}.
\end{equation}

\noindent For $\lambda\in(1,\infty)\cap U_1'$, $P^{(1)}$ satisfies

\begin{equation}\label{eq:localej2}
P^{(1)}_+(\lambda)=P^{(1)}_-(\lambda)\begin{pmatrix}
1 & f_t(\lambda) e^{N(g_{+,s}(\lambda)+g_{s,-}(\lambda)-V_s(\lambda)-\ell_s)}\\
0 & 1
\end{pmatrix}.
\end{equation}

\noindent For $\lambda\in \Sigma_{k+1}^{\pm}$, $P^{(1)}$ satisfies

\begin{equation}\label{eq:localej3}
P^{(1)}_+(\lambda)=P^{(1)}_-(\lambda)\begin{pmatrix}
1 & 0\\
f_t(\lambda)^{-1} e^{\mp Nh_s(\lambda)} & 1
\end{pmatrix}.
\end{equation}

\item[3.] For $z\in \partial U_1$, $P^{(1)}$ satisfies the following matching condition,

\begin{equation}
P^{(1)}(z)\left[P^{(\infty)}(z)\right]^{-1}=I+\mathcal{O}(N^{-1}),
\end{equation}

\noindent where the entries of the $\mathcal{O}(N^{-1})$ matrix are $\mathcal{O}(N^{-1})$ uniformly in $z\in \partial U_1$, uniformly in $\lbrace x_i\rbrace$ for $|x_i-x_j|\geq 3\delta$ for $i\neq j$ and $|x_i\pm 1|\geq 3\delta$ for $j\in\lbrace 1,...,k\rbrace$, uniformly in $t\in[0,1]$, and uniformly in $s\in[0,1]$. If in a neighborhood of $[-1,1]$, $\Tree$ is a real polynomial with fixed degree, the error is also uniform in the coefficients once they are restricted to some bounded set.
\end{itemize}
\end{lemma}

Again we will need finer asymptotics for our second differential identity and we will formulate them in the $\Tree=0$ case.

\begin{lemma}\label{le:locale2ndorderasy}
For $z\in \partial U_{1}$

\begin{align*}
&P^{(1)}(z)\left[P^{(\infty)}(z)\right]^{-1}\\
&=I+P^{(\infty)}(z)f(z)^{\sigma_3/2}e^{i\pi\sigma_3/4}\frac{1}{8}\begin{pmatrix}
\frac{1}{6} & 1\\
-1 & -\frac{1}{6}
\end{pmatrix}e^{-i\pi\sigma_3/4}f(z)^{-\sigma_3/2}\left[P^{(\infty)}(z)\right]^{-1}\xi_s(z)^{-3/2}+\mathcal{O}(N^{-2})
\end{align*}

\noindent where the $\mathcal{O}(N^{-2})$-term is a $2\times 2$ matrix whose entries are $\mathcal{O}(N^{-2})$ uniformly in $z,s,$ and $\lbrace |x_i-x_j|\geq 3\delta \ \mathrm{for \ } i\neq j\rbrace$ and $\lbrace |1\pm x_j|\geq 3\delta \ \mathrm{for \ all \ } j\in\lbrace 1,...,k\rbrace\rbrace$. \end{lemma}

\begin{remark}\label{rem:localematch}
Using the definition of $F$, one can check that this can be written also as

\begin{align*}
P^{(1)}(z)\left[P^{(\infty)}(z)\right]^{-1}=I+F(z)\begin{pmatrix}
0 & \frac{5}{48}\xi_s(z)^{-2}\\
-\frac{7}{48}\xi_s(z)^{-1} & 0
\end{pmatrix}F(z)^{-1} +\mathcal{O}(N^{-2}).
\end{align*}

From the previous representation of the matching condition matrix, one can easily see that the subleading term is indeed of order $N$. The benefit of this representation is that as $F$ and $F^{-1}$ are analytic in $U_1$, the subleading term is analytic in $U_1\setminus \lbrace 1\rbrace$ and has $($at most$)$ a second order pole at $z=1$.

\end{remark}

\subsection{The final transformation and asymptotic analysis of the problem}\label{sec:rhpasy}

We now perform the final transformation of the problem, and solve it asymptotically. The proofs of these statements are essentially standard in the RHP literature, but we don't know of a reference where the exact calculations we need exist and also issues such as uniformity in our relevant parameters are essential for us, but not usually stressed in the literature. Thus we provide proofs in Appendix \ref{app:rrhp}.

\begin{definition}\label{def:r}
Let us fix some small $\delta>0$ $($``small" being independent of $t$ and $s$ and detailed in Section \ref{sec:locals} and Section \ref{sec:locale}$)$, and write $U_{\pm 1}$ for a $\delta$-disk around $\pm 1$ and $U_{x_j}$ for a $\delta$-disk around $x_j$. We also assume that for $i\neq j$, $|x_i-x_j|\geq 3\delta$ and for all $i\neq 0,k+1$, $|x_i\pm 1|\geq 3\delta$. We then define

\begin{equation}\label{eq:r}
R(z)=\begin{cases}
S(z)\left[P^{(-1)}(z)\right]^{-1}, & z\in U_{-1}\setminus \Sigma\\
S(z)\left[P^{(x_j)}(z)\right]^{-1}, & z\in U_{x_j}\setminus \Sigma \ \mathrm{for \ some \ } j\\
S(z)\left[P^{(1)}(z)\right]^{-1}, & z\in U_{1}\setminus \Sigma\\
S(z) \left[P^{(\infty)}(z)\right]^{-1}, & z\in \C\setminus \overline{U_{-1}\bigcup\cup_{j=1}^k U_{x_j} \bigcup U_1\bigcup \Sigma}
\end{cases}.
\end{equation}

\end{definition}

We now state what is the Riemann--Hilbert solved by $R$ -- for details, see Appendix \ref{app:rrhp}.

\begin{lemma}\label{le:rrhp}
For the $\delta$ in Definition \ref{def:r}, define

\begin{align}\label{eq:gamma}
\Gamma_\delta&=\left(\R\setminus[-1-\delta,1+\delta]\right)\bigcup\left(\cup_{j=1}^{k+1}(\Sigma_j^+\cup \Sigma_j^-)\setminus \overline{U_{-1}\cup\cup_{j=1}^k U_{x_j}\cup U_1}\right)\\
\notag & \qquad \bigcup\left(\partial U_{-1}\cup \cup_{j=1}^k \partial U_{x_j} \cup \partial U_1\right),
\end{align}

\noindent where $\R$ and the lenses are oriented as before. $\partial U_{x_j}$ and $\partial U_{\pm 1}$ are oriented in a clockwise manner -- see Figure \ref{fig:R_RHP}. Then $R$ is the unique solution to the following Riemann-Hilbert problem:

\begin{itemize}[leftmargin=0.5cm]
\item[1.] $R:\C\setminus \Gamma_\delta\to \C^{2\times 2}$ is analytic.
\item[2.] $R$ satisfies the jump conditions $R_+(\lambda) = R_-(\lambda) J_R(\lambda)$ $($with lenses and $\R$ oriented as before, and the circles are oriented clockwise$)$, where the jump matrix $J_R$ take the following form:

\begin{itemize}
\item[$($i$)$] For $\lambda\in \R\setminus [-1-\delta,1+\delta]$,

\begin{align}\label{eq:rjump1}
J_R(\lambda)&= P^{(\infty)}(\lambda)\begin{pmatrix}
1 & f_t(\lambda) e^{N(g_{s,+}(\lambda)+g_{s,-}(\lambda)-V_s(\lambda)-\ell_s)}\\
0 & 1
\end{pmatrix}\left[P^{(\infty)}(\lambda)\right]^{-1}.
\end{align}

\item[$($ii$)$] For $\lambda\in \cup_{j=1}^{k+1}\Sigma_j^\pm\setminus\overline{U_{-1}\cup\cup_{j=1}^k U_{x_j}\cup U_1}$,

\begin{equation}\label{eq:rjump2}
J_R(\lambda)=P^{(\infty)}(\lambda)\begin{pmatrix}
1 & 0\\
f_t(\lambda)^{-1}e^{\mp Nh_s(\lambda)} & 1
\end{pmatrix}\left[P^{(\infty)}(\lambda)\right]^{-1}.
\end{equation}

\item[$($iii$)$]

For $\lambda\in \partial U_{x_j}\setminus \cup_{j=1}^{k+1}(\Sigma_j^+\cup \Sigma_j^-)$,

\begin{equation} \label{eq:R_jump3}
J_R(\lambda)=P^{(x_j)}(\lambda)\left[P^{(\infty)}(\lambda)\right]^{-1}.
\end{equation}

\item[$($iv$)$] For $\lambda\in \partial U_{\pm 1}\setminus (\R\cup \cup_{j=1}^{k+1}(\Sigma_j^+\cup \Sigma_j^-)$,

\begin{equation} \label{eq:R_jump4}
J_R(\lambda)=P^{(\pm 1)}(\lambda)\left[P^{(\infty)}(\lambda)\right]^{-1}.
\end{equation}
\end{itemize}

\item[3.] As $z\to \infty$,

\begin{equation}
R(z)=I+\mathcal{O}(|z|^{-1}).
\end{equation}
\end{itemize}
\end{lemma}

The first ingredient to solving this Riemann--Hilbert problem is to show that the jump matrix of $R(z)$ is close to the identity matrix in a suitable sense.

\begin{lemma}\label{le:rjumpmat}
For $z \in \Gamma_\delta$, write $J_R(z)=I+\Delta_R(z)=I+\Delta$ for the jump matrix of $R$ as described in Lemma \ref{le:rrhp}. 
Then for any $p\geq 1$, and large enough $N$ $($``large enough" depending only on $V)$

\begin{equation*}
||\Delta||_{L^p(\Gamma_\delta)}=\mathcal{O}(N^{-1})
\end{equation*}

\noindent where the norm is any matrix norm, the $L^p$-spaces are with respect to the Lebesgue measure on the jump contour, and the $\mathcal{O}(N^{-1})$ term is uniform in everything relevant $($i.e., $(x_i)$ for $|x_i-x_j|\geq 3\delta$, for $i\neq 0,k+1$: $|x_i\pm 1|\geq 3\delta$, in $s,t\in[0,1]$, and if $\Tree$ is a real polynomial in a neighborhood of $[-1,1]$, then in its coefficients when restricted to a bounded set; but may depend on $\delta)$.
\end{lemma}

\begin{figure}
\begin{center}
\begin{tikzpicture}[scale=0.6, every node/.style={scale=0.8}]
\draw[fill = black] (0,0) circle [radius = 0.05] node[below]{$-1$};
\draw[fill = black] (6,0) circle [radius = 0.05] node[below] {$x_1$};
\draw[fill = black] (12,0) circle [radius = 0.05] node[below]{$1$};
\draw[thick, ->, blue] (0.75,1.3) to [out = 60, in = 180] (3,2.5) node[above]{$\Sigma_1^+ \backslash (\cup_{j} U_{x_j})$};
\draw[thick, ->, blue] (0.75,-1.3) to [out = -60, in = -180] (3,-2.5) node[below]{$\Sigma_1^-\backslash (\cup_{j} U_{x_j})$};
\draw[thick, ->, blue] (6.75, 1.3) to [out = 60, in = 180] (9,2.5) node[above]{$\Sigma_2^+ \backslash (\cup_{j} U_{x_j})$};
\draw[thick, ->, blue] (6.75, -1.3) to [out = -60, in = -180] (9,-2.5) node[below]{$\Sigma_2^- \backslash (\cup_{j} U_{x_j})$};
\draw[thick, blue](3, 2.5) to [out = 0, in = 120] (5.25,1.3) 	(3, -2.5) to [out = 0, in = -120] (5.25, -1.3)
		(9, 2.5) to [out = 0, in = 120] (11.25,1.3) (9, -2.5) to [out = 0, in = -120] (11.25,-1.3);
\draw[thick, brown] (0,0) circle [radius=1.5]  (6,0) circle [radius=1.5]  (12,0) circle [radius=1.5];
\node at (0, 2) {\color{brown}$\partial U_{-1}$}; \node at (6, 2) {\color{brown}$\partial U_{x_1}$}; \node at (12, 2) {\color{brown}$ \partial U_1$};
\draw[thick, ->, brown] (-1.5, 0) arc [radius = 1.5, start angle = 180, end angle = 90];
\draw[thick, ->, brown] (4.5, 0) arc [radius = 1.5, start angle = 180, end angle = 90];
\draw[thick, ->, brown] (10.5, 0) arc [radius = 1.5, start angle = 180, end angle = 90];

\draw[thick, ->, red] (13.5, 0) -- (16.5, 0) (13.5, 0) -- (15,0);
\draw[thick, ->, red] (-4.5, 0) -- (-1.5, 0) (-4.5, 0) -- (-3,0);
\end{tikzpicture}
\end{center}
\caption{\label{fig:R_RHP} The jump contour of the Riemann--Hilbert problem for $R$, in the case $k = 1$.}
\end{figure}
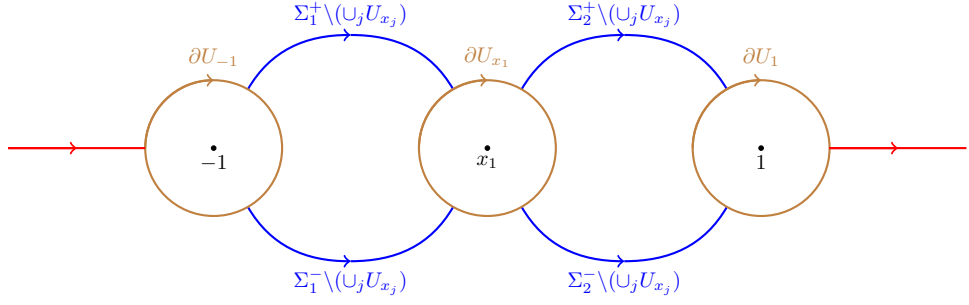

See Appendix \ref{app:rrhp} for a proof. 
We will want to show that $R$ is close to the identity, and the tool which allows us to do this is the following
representation of $R$ as a solution to a suitable integral equation involving its jump matrix.

\begin{proposition}\label{prop:neumann}
Let $\delta>0$ be small enough $($``small enough" being independent of $s$ and $t)$. For $N$ sufficiently large $($again independent of $s$ and $t)$, the unique solution of the Riemann--Hilbert problem for $R$ $($see Lemma \ref{le:rrhp}$)$ is given by

\begin{equation}\label{eq:R_Neuman}
R = I + C[\Delta + (I- C_\Delta)^{-1}(C_\Delta(I)) \Delta]
\end{equation}

\noindent where
\begin{align*}
C(f) : = \frac{1}{2\pi i} \int_{\Gamma_\delta} f(s) \frac{ds}{s-z}
\end{align*}

\noindent is the Cauchy operator on $\Gamma_\delta$, and $C_\Delta(f) = C_-(f \Delta)$ where $C_-(f) = \lim_{z \to s} C(f)$ as $z$ approaches a point $s \in \Gamma_\delta \backslash \{ \mathrm{intersection\ points}\}$ from the $-$side of $\Gamma_\delta$ $($for the orientation, see Lemma \ref{le:rrhp}$)$.
\end{proposition}

Finally, what we want to show is that $R(z)=I+\mathcal{O}(N^{-1})$ uniformly in everything relevant and use this as well as the explicit form of our parametrices to analyze our differential identities. The precise statement we need is the following one.

\begin{theorem}\label{th:rasy}
For small enough $\delta>0$ $($again small enough being independent of relevant quantities$)$ and large enough $N$ $($large enough being independent of everything relevant$)$ with respect to any matrix norm $|\cdot |$, there exists a $c>0$ such that

\begin{align*}
|R(z) - I| \le \frac{c}{N} \qquad \mathrm{and} \qquad |R'(z)|\le \frac{c}{N}
\end{align*}

\noindent uniformly in $(x_i)$ for $|x_i-x_j|\geq 3\delta$, $|x_i\pm 1|\geq 3\delta$ for $i\neq 0,k+1$, $t,s \in [0,1]$, $z \in \mathbb{C} \backslash \Gamma_\delta$, and if $\Tree$ is a real polynomial in a neighborhood of $[-1,1]$, then the error is uniform in its coefficients when these are restricted to a bounded set.

Moreover, for $\Tree=0$, we have

\begin{align*}
R(z) = I+R_1(z) + o(1/N), \qquad R'(z) = R_1'(z) + o(1/N)
\end{align*}

\noindent uniformly in $(x_i)$ for $|x_i-x_j|\geq 3\delta$, $|x_i\pm 1|\geq 3\delta$ for $i\neq 0,k+1$, $s \in [0,1]$, and $z \in \mathbb{C} \backslash (\Gamma_\delta\cup \cup_{j=0}^{k+1} U_{x_j})$. Here $R_1(z) = \sum_{j=0}^{k+1} R_1^{(x_j)}(z)$ with 

\begin{equation*}
R_1^{(x_j)}(z) =\begin{cases}
\frac{1}{x_j-z}\frac{\beta_j}{4\pi N d_s(x_j) \sqrt{1-x_j^2}} E^{(x_j)}(x_j) \begin{pmatrix} 0 & 1 + \frac{\beta_j}{2} \\ 1 - \frac{\beta_j}{2} & 0 \end{pmatrix} \left[E^{(x_j)}(x_j)\right]^{-1}, & j\in\lbrace 1,...,k\rbrace\\
-\underset{w = -1}{\mathrm{Res}} \frac{1}{w-z} F^{(-1)}(w) \begin{pmatrix} 0 & -\frac{5}{48 \xi_s^{(-1)}(w)^{2}}  \\ - \frac{7}{48 \xi_s^{(-1)}(w)}  & 0 \end{pmatrix} \left[F^{(-1)}(w)\right]^{-1}, & j = 0\\
-\underset{w = 1}{\mathrm{Res}} \frac{1}{w-z} F^{(1)}(w) \begin{pmatrix} 0 & \frac{5}{48\xi_s^{(1)}(w)^{2}}  \\ - \frac{7}{48 \xi_s^{(1)}(w)}  & 0 \end{pmatrix} \left[F^{(1)}(w)\right]^{-1}, & j = k+1
\end{cases}.
\end{equation*}
where $E$ and $F$ are the ``compatibility matrices" from Definitions \ref{def:local} and \ref{def:locale}.
\noindent In particular, we have
\begin{align*}
&\Jcal^{(x_j)}(z) : = \left([P^{(\infty)}(z)]^{-1} \left[R_1^{(x_j)}\right]'(z) P^{(\infty)}(z) \right)_{22}\\
& =\begin{cases}
\frac{1}{4} \frac{1}{(z - x_j )^2}\frac{i\beta_j}{4 \pi Nd_s(x_j)\sqrt{1-x_j^2}}
\bigg[\frac{a(z)^2}{a_+(x_j)^2} (c_{x_j, s}^2 + c_{x_j, s}^{-2} - \beta_j) \\
\qquad \qquad - \frac{a_+(x_j)^2}{a(z)^2}(c_{x_j, s}^2 + c_{x_j, s}^{-2} + \beta_j)\bigg] , & j \in \{1, \dots, k\}\\
-  \frac{1}{(z+1)^2} \frac{\sqrt{2}i}{8N}\left\{a(z)^{-2} \left[ \frac{5+ 96 \mathcal{A}^2}{48G_s^{(-1)}(-1)}- \frac{5\left[ G_s^{(-1)}\right]'(-1)}{12G_s^{(-1)}(1)^2} \right] -a(z)^2 \frac{7}{24G_s^{(-1)}(-1)} \right\}\\ \qquad \qquad + \frac{1}{(z+1)^3}\frac{5\sqrt{2}i}{48NG_s^{(-1)}(1)}  a(z)^{-2}, & j = 0\\
-\frac{1}{(z-1)^2} \frac{\sqrt{2}}{8N} \left\{a(z)^2 \left[\frac{5 + 96 \mathcal{A}^2}{48 G_s^{(1)}(1)} - \frac{5\left[G_s^{(1)}\right]'(1)}{12G_s^{(1)}(1)^2} \right] - a(z)^{-2} \frac{7}{24G_s^{(1)}(1)}\right\}\\ \qquad \qquad - \frac{1}{(z-1)^3} \frac{5\sqrt{2}}{48NG_s^{(1)}(1)} a(z)^2, & j=k+1
\end{cases}
\end{align*}

\noindent where
\begin{align*}
c_{x_j, s} &= \left(x_j + i \sqrt{1-x_j^2}\right)^{\mathcal{A}} \exp \left(-i \sum_{k > j} \beta_k \pi / 2 + N \phi_{s, +}(x_j) - (1+\beta_j) \pi i / 4 \right),\\
G_s^{(-1)}(-1) & = -i\pi \sqrt{2}d_s(-1), \quad \left[G_s^{(-1)}\right]'(-1) = -\frac{3\pi i}{10\sqrt{2}}[4d_s'(-1) - d_s(-1)],\\
G_s^{(1)}(1) & = \pi \sqrt{2}d_s(1), \quad \left[G_s^{(1)}\right]'(1) = \frac{3\pi}{10\sqrt{2}}[4d_s'(1) + d_s(1)].
\end{align*}

\end{theorem}

\begin{remark}
As discussed in \cite{krasovsky}, using the asymptotic expansions of the Airy function and Bessel functions, the matching conditions of the local parametrices can be extended into asymptotic expansions in inverse powers of $N$. These then can be used to prove a full asymptotic expansion for $R$ and $R'$. We don't have use for this, so we won't discuss it further.
\end{remark}

\section{Integrating the differential identities}\label{sec:di}

In this section we will use our asymptotic solution and precise form of the parametrices to analyze the asymptotics of the differential identities \eqref{eq:di1} and \eqref{eq:di2}, and finally integrate them. We will start with \eqref{eq:di1}.

\subsection{\texorpdfstring{The differential identity \eqref{eq:di1}}{The first differential identity}}

Here we will give a (slightly simplified) variant of the argument in \cite[Section 5.3]{dik2} to integrate the differential identity \eqref{eq:di1}. As there are minor modifications due to the differences in the models and the argument being relevant for \eqref{eq:di2}, we present a full proof here. The main goal we wish to prove is the following.

\begin{proposition}\label{prop:di1int}
Let $V$ be one-cut regular, $\Tree$ as in Proposition \ref{prop:fh}, and $\delta>0$ small enough, but independent of $N$. Then as $N\to\infty$,

\begin{align}\label{eq:di1int}
\log \frac{D_{N-1}(f_1;V)}{D_{N-1}(f_0;V)}&=N\int_{-1}^1 \Tree(x)d(x)\sqrt{1-x^2}dx+\frac{\mathcal{A}}{\pi}\int_{-1}^1 \frac{\Tree(x)}{\sqrt{1-x^2}}dx-\sum_{j=1}^k \frac{\beta_j}{2}\Tree(x_j)\\
\notag &\quad +\frac{1}{4\pi^2}\int_{-1}^1 dy\frac{\Tree(y)}{\sqrt{1-y^2}}P.V.\int_{-1}^1 \frac{\Tree'(x)\sqrt{1-x^2}}{y-x}dx+o(1)
\end{align}

\noindent where $o(1)$ is uniform in $\lbrace (x_j)_{j=1}^k: |x_i-x_j|\geq 3\delta, i\neq j \ and \ |x_i\pm 1|\geq 3\delta \ \forall i\rbrace$, and if in a neighborhood of $[-1,1]$, $\Tree$ is a real polynomial of fixed degree, then the error is also uniform in the coefficients of $\Tree$ when these are restricted to a bounded set.
\end{proposition}

The way we will do this is we'll express the integrand in \eqref{eq:di1} in a slightly different way which will allow deforming our integration contour in such a way that we can express $Y$ in terms of $R$ and the global parametrix $P^{(\infty)}$. The expression will be such that to leading order, we can treat $R$ as the identity, and using the global parametrix, we can perform the relevant integrals explicitly. 

Let us begin with expressing our integral in terms of the global parametrix. We first remind the reader that we denoted by $U_{[-1,1]}$ a fixed (independent of $N$ and $t$) complex neighborhood of $[-1,1]$ into which $\Tree_t$ had an analytic continuation for all $t\in[0,1]$. We also assumed that the lenses and neighborhoods $(U_{x_j})_{j=0}^{k+1}$ were inside $U_{[-1,1]}$. 

\begin{lemma}\label{le:di1global}
Let $\tau_+:[0,1]\to \lbrace z\in \C: \mathrm{Im}(z)\geq 0\rbrace\cap U_{[-1,1]}$ be a smooth simple curve independent of $N$. We also assume that $\tau_+(0)<-1$, $\tau_+(1)>1$, and that $\tau(s)$ is outside of the lenses and neighborhoods $(U_{x_j})_{j=0}^{k+1}$ for all $s$. We also define $\tau_-$ in a similar way but in the lower half plane and with the assumption that $\tau_-(0)=\tau_+(0)$ as well as $\tau_-(1)=\tau_+(1)$. See Figure \ref{fig:int_contour} for an illustration. 

Then for $t\in[0,1]$

\begin{align*}
\frac{1}{2\pi i}&\int_\R \left[Y_{11}(x,t)\partial_x Y_{21}(x,t)-Y_{21}(x,t)\partial_x Y_{11}(x,t)\right]\partial_t f_t(x) e^{-NV(x)}dx\\
&=N\int_{-1}^1 d(x)\sqrt{1-x^2}\frac{\partial_t f_t(x)}{f_t(x)}dx+\frac{1}{2\pi i}\left[\int_{\tau_+}-\int_{\tau_-}\right]\frac{\mathcal{D}_t'(z)}{\mathcal{D}_t(z)}\frac{\partial_t f_t(z)}{f_t(z)}dz+o(1),
\end{align*}

\noindent where $o(1)$ is uniform in $t\in[0,1]$, $\lbrace (x_j)_{j=1}^k: |x_i-x_j|\geq 3\delta, i\neq j \ and \ |x_i\pm 1|\geq 3\delta \ \forall i\rbrace$, and if in a neighborhood of $[-1,1]$, $\Tree$ is a real polynomial of fixed degree, then the error is also uniform in the coefficients of $\Tree$ when these are restricted to a bounded set.
\end{lemma}

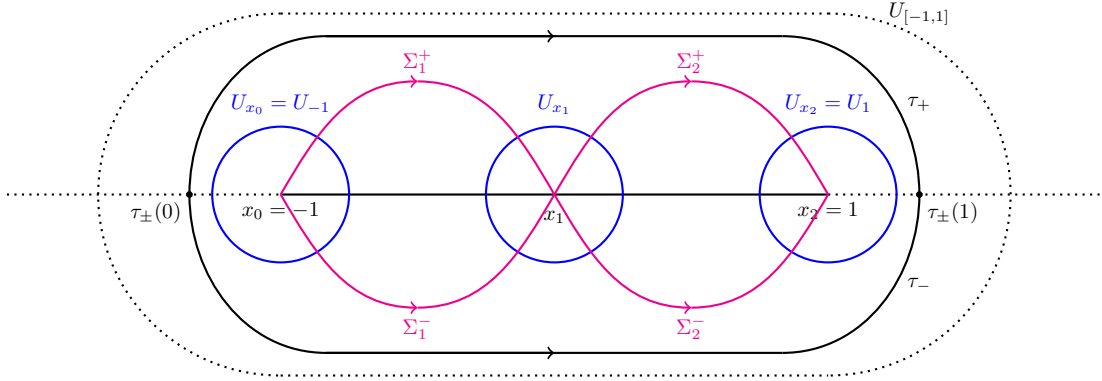
\begin{figure}[h!]
\begin{center}
\begin{tikzpicture}[scale=0.6, every node/.style={scale=0.8}]
\draw[thick, dotted] (-6, 0) -- (0,0) node[below]{$x_0 = -1$} (12, 0) node[below]{$x_2 = 1$} -- (18,0);
\draw[thick, dotted] (0,4) arc [radius=4, start angle=90, end angle= 270]
			(12,4) arc [radius=4, start angle=90, end angle= -90];
\draw[thick, dotted] (0, 4) -- (12, 4) (0, -4) -- (12, -4);
\node at (14, 4) {$U_{[-1,1]}$};

\draw[thick, fill] (-2,0) circle [radius = 0.05] node[below left]{$\tau_\pm(0)$};
\draw[thick, fill] (14,0) circle [radius = 0.05] node[below right]{$\tau_\pm(1)$};
\draw[thick] (-2, 0) to [out = 90, in = 180] (1, 3.5);
\draw[thick, ->] (1, 3.5) -- (11, 3.5) (1, 3.5) -- (6, 3.5);
\draw[thick] (11, 3.5) to [out = 0, in = 90] (14, 0);
\node at (14, 2) {$\tau_+$};

\draw[thick] (-2, 0) to [out = -90, in = -180] (1, -3.5);
\draw[thick, ->] (1, -3.5) -- (11, -3.5) (1, -3.5) -- (6, -3.5);
\draw[thick] (11, -3.5) to [out = 0, in = -90] (14, 0);
\node at (14, -2) {$\tau_-$};

\draw[thick] (0,0) -- (12,0); \node at(6,-0.5){$x_1$};
\draw[blue, thick] (0,0) circle [radius=1.5]  (6,0) circle [radius=1.5]  (12,0) circle [radius=1.5];
\draw[magenta, thick, ->] (0,0) to [out = 60, in = 180] (3,2.5) node[above]{$\Sigma_1^+$};
\draw[magenta, thick, ->] (0,0) to [out = -60, in = -180] (3,-2.5) node[below]{$\Sigma_1^-$};
\draw[magenta, thick, ->] (6,0) to [out = 60, in = 180] (9,2.5) node[above]{$\Sigma_2^+$};
\draw[magenta, thick, ->] (6,0) to [out = -60, in = -180] (9,-2.5) node[below]{$\Sigma_2^-$};
\draw[magenta, thick](3, 2.5) to [out = 0, in = 120] (6,0) 	(3, -2.5) to [out = 0, in = -120] (6,0)
		(9, 2.5) to [out = 0, in = 120] (12,0) (9, -2.5) to [out = 0, in = -120] (12,0);
\node at (0, 2) {\color{blue}$U_{x_0} = U_{-1}$}; \node at (6, 2) {\color{blue}$U_{x_1}$}; \node at (12, 2) {\color{blue}$U_{x_2} = U_1$};
\end{tikzpicture}
\caption{\label{fig:int_contour} Deforming the integration contour, $k = 1$. }
\end{center}
\end{figure}

\begin{proof}
Let us write $Y'=\partial_x Y$. We first note that an elementary calculation using \eqref{eq:Yjump} and the fact that the first column of $Y$ consists of polynomials which have no jump across $\R$,  show that for $\lambda\in \R$, 

\begin{equation}\label{eq:direform}
f_te^{-NV}(Y_{11}Y_{21}'-Y_{21}Y_{11}')=\left(Y_{22,-}Y_{11}'-Y_{12,-}Y_{21}'\right)-\left(Y_{22,+}Y_{11}'-Y_{12,+}Y_{21}'\right).
\end{equation}

Now recall that $Y_{12,\pm}$ and $Y_{22,\pm}$ have continuous boundary values on $\R$ so we see that the terms $Y_{22}Y_{11}'-Y_{12}Y_{21}'$ are analytic in $\C\setminus \R$ and are continuous up to the boundary. Moreover, by our construction, $f_t(z)^{-1}\partial_t f_t(z)$ is analytic in $U_{[-1,1]}$. We can thus argue by Cauchy's integral theorem to deform the integration contour. In particular, plugging \eqref{eq:direform} into \eqref{eq:di1}, we find 

\begin{align*}
\frac{1}{2\pi i}&\int_\R \left[Y_{11}(x,t)\partial_x Y_{21}(x,t)-Y_{21}(x,t)\partial_x Y_{11}(x,t)\right]\partial_t f_t(x) e^{-NV(x)}dx\\
&=\frac{1}{2\pi i}\int_{(-\infty,\tau_+(0)]\cup [\tau_+(1),\infty)} \left[Y_{11}(x,t)Y_{21}'(x,t)-Y_{21}(x,t)Y_{11}'(x,t)\right]\partial_t f_t(x) e^{-NV(x)}dx\\
&\quad -\frac{1}{2\pi i}\left[\int_{\tau_+}-\int_{\tau_-}\right]\left(Y_{22}(z,t)Y_{11}'(z,t)-Y_{12}(z,t)Y_{21}'(z,t)\right)\frac{\partial_t f_t(z)}{f_t(z)}dz.
\end{align*}

Notice that
\begin{align*}
Y_{11}Y_{21}' -Y_{21}Y_{11}' = [Y^{-1} Y']_{21},
\qquad Y_{22}Y_{11}' - Y_{12}Y_{21}' = [Y^{-1}Y']_{11}.
\end{align*}

\noindent Unravelling our transformations, we note as we are not inside the lenses or the neighborhoods, we have on $\R\setminus [\tau_+(0),\tau_+(1)]$ and on $\tau_\pm$

\begin{align}
\notag Y^{-1}Y'
& = \left[e^{N\ell_1\sigma_3 / 2} S e^{N(g_1-\ell_1/2)\sigma_3}\right]^{-1} \left[e^{N\ell_1\sigma_3 / 2} S e^{N(g_1-\ell_1/2)\sigma_3}\right]'\\
\notag & = Ng_1' {\sigma_3} + e^{-N(g_1-\ell_1/2)\sigma_3}S^{-1}S' e^{N(g_1-\ell_1/2)\sigma_3}\\
\label{eq:YY} & = Ng_1' {\sigma_3} + e^{-N(g_1-\ell_1/2)\sigma_3}\left[\left(P^{(\infty)}\right)^{-1} R^{-1} \left(RP^{(\infty)}\right)'\right] e^{N(g_1-\ell_1/2)\sigma_3}
\end{align}

\noindent where we have used the global parametrix in the last equality. Since the $P^{(\infty)}$-RHP implies that $P^{(\infty)}(z)$ is complex analytic when $z \not \in [-1,1]$, $I+\mathcal{O}(|z|^{-1})$ as $z \to \infty$, and $\det P^{(\infty)} \equiv 1$, we see that both $\left(P^{(\infty)}\right)^{-1}$ and $\left(P^{(\infty)}\right)'$ are bounded when we are away from a (complex) neighbourhood of $[-1, 1]$. One can easily check that they are in fact uniformly bounded in all our relevant parameters. Combined with the estimates 
\begin{align*}
R(z, t) = I + \mathcal{O}(N^{-1}), \qquad R'(z, t) = \mathcal{O}(N^{-1})
\end{align*}

\noindent in Theorem \ref{th:rasy}, we have $S^{-1}S' = \left(P^{(\infty)}\right)^{-1}\left(P^{(\infty)}\right)' + \mathcal{O}(N^{-1})$.

Consider first the integral along $\R\setminus [\tau_+(0),\tau_+(1)]$. Using the specific form \eqref{eq:global} of $P^{(\infty)}$, \eqref{eq:YY}, and the fact that terms containing $R$ give something $o(1)$, a direct calculation shows that 

\begin{align*}
& [Y(z, t)^{-1}Y'(z, t)]_{21}
= e^{N(2g_1(z)-\ell_1)}\left[P^{(\infty)}_{11}(z, t) \partial_z P_{21}^{(\infty)}(z, t) - P_{21}^{(\infty)}(z, t)\partial_z P_{11}^{(\infty)}(z, t) +o(1)\right]\\
& \quad = \frac{ie^{N(2g_1(z)-\ell_1)}}{4 \mathcal{D}_t^2(z)} \left[
((a(z)^2 + a(z)^{-2})(a(z)^2 - a(z)^{-2})' - (a(z)^2 - a(z)^{-2})(a(z)^2 + a(z)^{-2})' + o(1)
\right]\\
& \quad = \frac{ie^{N(2g_1(z)-\ell_1)}}{\mathcal{D}_t(z)^2} \left[ \frac{1}{z^2-1} + o(1) \right].
\end{align*}

\noindent Thus
\begin{align*}
& \left[Y_{11}(x,t)\partial_x Y_{21}(x,t)-Y_{21}(x,t)\partial_x Y_{11}(x,t)\right]\partial_t f_t(x) e^{-NV(x)}\\
& \qquad \qquad =  \left[\frac{e^{\mathcal{T}(x)} - 1}{\mathcal{D}_t(x)^2 (x^2 - 1)} \prod_{j=1}^k |x - x_j|^{\beta_j} + o(1)\right] e^{N(g_{1,+}(x)+g_{1,-}(x)-\ell_1-V(x))}
\end{align*}

\noindent and one finds from \eqref{eq:gbv2} that as $N\to\infty$, the integral along $\R\setminus [\tau_+(0),\tau_+(1)]$ is $o(1)$ uniformly in everything relevant.

Consider then the integrals along $\tau_\pm$. A similar direct calculation shows that

\begin{align*}
[Y(z, t)^{-1}Y'(z, t)]_{11}
& = Ng_1'(z) + P^{(\infty)}_{22}(z, t) \partial_z P_{11}^{(\infty)}(z, t) - P_{12}^{(\infty)}(z, t)\partial_z P_{21}^{(\infty)}(z, t) +o(1)\\
& = Ng_1'(z) +
\frac{1}{4}\left[\frac{\partial_z \mathcal{D}_t(z)^{-1}}{\mathcal{D}_t(z)^{-1}} \left((a(z)^2 + a(z)^{-2})^2 - (a(z)^2 - a(z)^{-2})^2\right)\right] +o(1)\\
& = Ng_1'(z) - \frac{\mathcal{D}_t'(z)}{\mathcal{D}_t(z)} +o(1)
\end{align*}

\noindent and hence
\begin{align*}
\left(Y_{22}(z,t)Y_{11}'(z,t)-Y_{12}(z,t)Y_{21}'(z,t)\right)\frac{\partial_t f_t(z)}{f_t(z)}
& =Ng_1'(z)\frac{\partial_t f_t(z)}{f_t(z)}-\frac{\mathcal{D}_t'(z)}{\mathcal{D}_t(z)}\frac{\partial_t f_t(z)}{f_t(z)}+o(1),
\end{align*}

\noindent where again $o(1)$ is uniform in everything relevant. This yields the claim once we notice that by contour deformation and \eqref{eq:gbv3}

\begin{align*}
-\frac{1}{2\pi i}\left[\int_{\tau_+}-\int_{\tau_-}\right]g_1'(z)\frac{\partial_t f(z)}{f_t(z)}dz=\int_{-1}^1 d(x)\sqrt{1-x^2}\frac{\partial_t f_t(x)}{f_t(x)}dx. 
\end{align*}

\end{proof}

Our next task is to calculate the $\tau_\pm$ integrals. To do this, we introduce some notation.

\begin{definition}\label{def:qfhnsz}
For $z\in \C\setminus(-\infty,1]$, let

\begin{equation}\label{eq:qfh}
q_{FH}(z)=\log\left[(z+r(z))^{-\mathcal{A}}\prod_{j=1}^k (z-x_j)^{\beta_j/2}\right],
\end{equation}

\noindent where the logarithm is with the principal branch, $\mathcal{A}=\sum_{j=1}^k \beta_j/2$, and $FH$ refers to Fisher-Hartwig.  We also define for $z\in\C\setminus [-1,1]$

\begin{equation}\label{eq:sz}
q_{Sz}(z)=q_{Sz}(z,t)=\frac{r(z)}{2\pi}\int_{-1}^1 \frac{\mathcal{T}_t(\lambda)}{\sqrt{1-\lambda^2}}\frac{1}{z-\lambda}d\lambda,
\end{equation}

\noindent where $r(z)$ is as in \eqref{eq:rdef} and $Sz$ refers to Szeg\H{o}.
\end{definition}

Note that we have $\mathcal{D}_t'/\mathcal{D}_t=q_{FH}'+q_{Sz}'$. We will need the following fact before proving Proposition \ref{prop:di1int}. The following is an analogue of a result in \cite{deift2} in the case of the circle.

\begin{lemma}\label{le:szint}

Write $\tau_\pm$ be as in Lemma \ref{le:di1global}. We have

\begin{equation}\label{eq:szint}
\int_0^1 \frac{1}{2\pi i}\left[\int_{\tau_+}-\int_{\tau_-}\right]q_{Sz}'(z,t)\frac{\partial_t f_t(z)}{f_t(z)}dzdt=-\frac{1}{4\pi^2}\int_{-1}^1 dy\frac{\Tree(y)}{\sqrt{1-y^2}}P.V.\int_{-1}^1 \frac{\Tree'(x)\sqrt{1-x^2}}{x-y}dx.
\end{equation}
\end{lemma}

\begin{proof}
Let us recall that we saw in the proof of Lemma \ref{le:cauchybounds} that off of $[-1,1]$ we can write

\begin{equation*}
q_{Sz}(z,t)=\frac{r(z)}{2\pi}\int_{-1}^1 \frac{\mathcal{T}_t(\lambda)-\mathcal{T}_t(z)}{z-\lambda}\frac{d\lambda}{\sqrt{1-\lambda^2}}+\frac{\mathcal{T}_t(z)}{2}
\end{equation*}

\noindent which implies that $q_{Sz}$ is bounded in a neighborhood of $[-1,1]$ and $q_{Sz}(\pm 1,t)=\frac{1}{2}\mathcal{T}_t(\pm 1)$. Moreover, we see from this that

\begin{align*}
q_{Sz}'(z,t)&=\frac{r'(z)}{2\pi }\int_{-1}^1 \frac{\mathcal{T}_t(\lambda)-\mathcal{T}_t(z)}{z-\lambda}\frac{d\lambda}{\sqrt{1-\lambda^2}}+\frac{r(z)}{2\pi}\int_{-1}^1\frac{\mathcal{T}_t(z)-\mathcal{T}_t(\lambda)-\mathcal{T}_t'(z)(z-\lambda)}{(z-\lambda)^2}\frac{d\lambda}{\sqrt{1-\lambda^2}}+\frac{\mathcal{T}_t'(z)}{2}.
\end{align*}

This in turn implies that $q_{Sz}'$ is bounded except at $z=\pm 1$ where it has singularities of order $|z\mp 1|^{-1/2}$; in particular these are integrable ones. Due to the singularities being integrable, we can perform contour deformation and integrate by parts in the $z$-integral in the left hand side of \eqref{eq:szint}. Noting that $f_t^{-1}\partial_t f_t=\partial_t\Tree_t=:\dot{\Tree}_t$ (we will use a dot here and below to indicate time derivatives below when there is no risk of confusion), we see that

\begin{align}\label{eq:Iint}
I:=\int_0^1 dt\left[\int_{\tau_+}-\int_{\tau_-}\right]\frac{dz}{2\pi i}\dot{\Tree}_t(z)q_{Sz}'(z,t)=-\int_0^1 {dt}\int_{-1}^1 \frac{dx}{2\pi i}\dot{\Tree}_t'(x)\left[q_{Sz,+}(x,t)-q_{Sz,-}(x,t)\right].
\end{align}

Let us write for $x\in(-1,1)$, $s(x)=\sqrt{1-x^2}$. As for $x\in(-1,1)$, $r_\pm (x)=\pm is(x)$, we see by Sokhotski--Plemelj that

\begin{equation*}
q_{Sz,+}(x,t)-q_{Sz,-}(x,t)=is(x) \frac{1}{\pi} P.V. \int_{-1}^1 \frac{\Tree_t(y)}{x-y}\frac{dy}{s(y)}=:is(x)[\mathcal{H}(\mathbf{1}_{(-1,1)}\Tree_t/s)](x),
\end{equation*}

\noindent where $\mathbf{1}_{(-1,1)}$ is the indicator function of the interval $(-1,1)$, and $\mathcal{H}$ denotes the Hilbert transform (note that the Hilbert transform is well defined as $\mathbf{1}_{(-1,1)}\Tree_t/s\in L^p(\R)$ for $p\in[1,2)$).

To simplify notation slightly, let us write $\langle f,g\rangle:=\int_\R f(x)g(x)dx$. Integrating by parts in the $t$ integral in \eqref{eq:Iint} we see that

\begin{align}\label{eq:Iparts}
I&=-\int_0^1\frac{1}{2\pi}\left\langle \dot{\Tree}_t',\mathbf{1}_{(-1,1)}s \mathcal{H}\left(\mathbf{1}_{(-1,1)}\Tree_t/s\right)\right\rangle dt\\
\notag &=-\frac{1}{2\pi}\left\langle \Tree',\mathbf{1}_{(-1,1)}s \mathcal{H}\left(\mathbf{1}_{(-1,1)}\Tree/s\right)\right\rangle+\int_0^1\frac{1}{2\pi}\left\langle \Tree_t',\mathbf{1}_{(-1,1)}s \mathcal{H}\left(\mathbf{1}_{(-1,1)}\dot{\Tree}_t/s\right)\right\rangle dt.
\end{align}

Our aim is now to show that actually $\frac{1}{2\pi}\int_0^1\langle \Tree_t',\mathbf{1}_{(-1,1)}s \mathcal{H}(\mathbf{1}_{(-1,1)}\dot{\Tree}_t/s)\rangle dt=-I$ so we would have $I=-\langle \Tree',\mathbf{1}_{(-1,1)}s \mathcal{H}(\mathbf{1}_{(-1,1)}\Tree/s)\rangle/4\pi$, which we will see to be equivalent to our claim. To see that indeed $\frac{1}{2\pi}\int_0^1\langle \Tree_t',\mathbf{1}_{(-1,1)}s \mathcal{H}(\mathbf{1}_{(-1,1)}\dot{\Tree}_t/s)\rangle dt=-I$, we note first that

\begin{equation*}
\frac{s(x)}{s(y)}\frac{1}{x-y}=\frac{s(y)}{s(x)}\frac{1}{x-y}-\frac{x+y}{s(x)s(y)}
\end{equation*}

\noindent implying that for say a continuous $f:[-1,1]\to \R$ and $x\in(-1,1)$

\begin{equation}\label{eq:Hing1}
s(x) \left[\mathcal{H}\left(\mathbf{1}_{(-1,1)}f/s\right)\right](x)=\frac{1}{s(x)}\left[\mathcal{H}\left(\mathbf{1}_{(-1,1)}fs\right)\right](x)-\frac{1}{\pi}\int_{-1}^1 \frac{x+y}{s(x)s(y)}f(y)dy.
\end{equation}

Using the definition of the Cauchy principal value integral, one can also check easily that for a smooth $f:[-1,1]\to \R$ and $x\in(-1,1)$

\begin{equation}\label{eq:Hing2}
\left[\mathcal{H}(\mathbf{1}_{(-1,1)}fs)\right]'(x)=\left[\mathcal{H}(\mathbf{1}_{(-1,1)}(fs)')\right](x).
\end{equation}

Thus integrating by parts in the $x$ integral, using the fact that $q_+(\pm 1,t)=q_-(\pm 1,t)$,  and \eqref{eq:Hing2}, we see that

\begin{align}\label{eq:tpint}
\left\langle \Tree_t',\mathbf{1}_{(-1,1)}s \mathcal{H}\left(\mathbf{1}_{(-1,1)}\dot{\Tree}_t/s\right)\right\rangle&=\int_{-1}^1 dx \Tree_t(x)\frac{s'(x)}{s(x)^2}\left(\left[\mathcal{H}\left(\mathbf{1}_{(-1,1)}\dot{\Tree}_ts\right)\right](x)-\int_{-1}^1 \frac{x+y}{\pi s(y)}\dot{\Tree}_t(y)dy\right)\\
\notag & \quad -\int_{-1}^1 dx \Tree_t(x)\frac{1}{s(x)}\left(\left[\mathcal{H}\left(\mathbf{1}_{(-1,1)}(\dot{\Tree}_ts)'\right)\right](x)-\int_{-1}^1 \frac{\dot{\Tree}_t(y)}{\pi s(y)}dy\right).
\end{align}

We then note that

\begin{align*}
[\mathcal{H}(\mathbf{1}_{(-1,1)}\dot{\Tree}_ts')](x)-\frac{1}{\pi}\int_{-1}^1 \frac{\dot{\Tree}_t(y)}{s(y)}dy&=\frac{1}{\pi}P.V.\int_{-1}^1 \frac{\dot{\Tree}_t(y)}{s(y)}\left(\frac{-y}{x-y}-1\right)dy\\
&={-}x[\mathcal{H}(\mathbf{1}_{(-1,1)}\dot{\Tree}_t/s)](x)
\end{align*}

\noindent and

\begin{align*}
\left[\mathcal{H}\left(\mathbf{1}_{(-1,1)}\dot{\Tree}_ts\right)\right](x)-\frac{1}{\pi}\int_{-1}^1 \frac{x+y}{ s(y)}\dot{\Tree}_t(y)dy&=\frac{1}{\pi}P.V.\int_{-1}^1 \frac{\dot{\Tree}_t(y)}{s(y)}\frac{\left[s(y)^2-(x^2-y^2)\right]}{x-y}dy\\
&=s(x)^2[\mathcal{H}(\mathbf{1}_{(-1,1)}\dot{\Tree}_t/s)](x).
\end{align*}

Plugging these into \eqref{eq:tpint}, using the fact that $s'(x)=-x/s(x)$ along with the anti-self adjointness of $\mathcal{H}$ we see that

\begin{align}\label{eq:asa}
\frac{1}{2\pi}\int_0^1\left\langle \Tree_t',\mathbf{1}_{(-1,1)}s \mathcal{H}\left(\mathbf{1}_{(-1,1)}\dot{\Tree}_t/s\right)\right\rangle dt&=-\frac{1}{2\pi}\int_0^1\left\langle \Tree_t,\mathbf{1}_{(-1,1)}s^{-1} \mathcal{H}\left(\mathbf{1}_{(-1,1)}\dot{\Tree}_t's\right)\right\rangle dt\\
\notag &=\frac{1}{2\pi}\int_0^1\left\langle \dot{\Tree}_t',\mathbf{1}_{(-1,1)} s\mathcal{H}(\mathbf{1}_{(-1,1)} \Tree_t/s)\right\rangle dt\\
\notag &=-I.
\end{align}

\noindent Note that $1/s\notin L^2(-1,1)$ so we can't use the anti-self adjointness of the Hilbert transform on the space $L^2$, but we use the fact that if $f\in L^p(\R)$ and $g\in L^{p'}(\R)$, where $p'$ is the H\"older conjugate of $p$, then $\int g\mathcal{H}f=-\int f\mathcal{H}g$ -- see e.g. \cite[Theorem 102]{titchmarsh}.

Plugging \eqref{eq:asa} into \eqref{eq:Iparts}, we find our previous claim that

\begin{equation*}
I=-\frac{1}{4\pi}\langle \Tree',\mathbf{1}_{(-1,1)}s \mathcal{H}(\mathbf{1}_{(-1,1)}\Tree/s)\rangle
\end{equation*}

\noindent Making use of the anti-self adjointness of $\mathcal{H}$ again, this translates into

\begin{equation*}
I=\frac{1}{4\pi^2}\int_{-1}^1 dy\frac{\Tree(y)}{\sqrt{1-y^2}}P.V.\int_{-1}^1 \frac{\Tree'(x)\sqrt{1-x^2}}{y-x}dx
\end{equation*}

\noindent which is our claim.

\end{proof}

We are now in a position to finish the proof.

\begin{proof}[Proof of Proposition \ref{prop:di1int}]
We start with the result of Lemma \ref{le:di1global}. Consider first the integral along $[-1,1]$. Here we note that by the definition of $f_t$, $\int_0^1 f_t(x)^{-1}\partial_t f_t(x)dt=\log f_1(x)-\log f_0(x)=\Tree (x)$. This yields the $\mathcal{O}(N)$-term in \eqref{eq:di1int}.

Let us now consider the $\mathcal{D}_t'/\mathcal{D}_t$-terms. The contribution from $q_{Sz}$ is calculated in Lemma \ref{le:szint}, so we need to understand the contribution of $q_{FH}$. As $q_{FH}$ is independent of $t$, we find that 

\begin{align}\label{eq:qfhint}
\int_0^1 dt\left[\int_{\tau_+}-\int_{\tau_-}\right]\frac{dz}{2\pi i}q_{FH}'(z)\frac{\dot{f}_t(z)}{f_t(z)}&=\left[\int_{\tau_+}-\int_{\tau_-}\right]\frac{dz}{2\pi i}\Tree(z)q_{FH}'(z).
\end{align}

Now as

\begin{equation*}
q_{FH}'(z)=-\frac{\mathcal{A}}{r(z)}+\sum_{j=1}^k \frac{\beta_j}{2}\frac{1}{z-x_j}
\end{equation*}

\noindent we see by Cauchy's integral theorem, the fact that $r_\pm(x)={\pm}i\sqrt{1-x^2}$ for $x\in(-1,1)$, and Sokhotski-Plemelj that

\begin{equation}\label{eq:qfhint2}
\int_0^1 dt\left[\int_{\tau_+}-\int_{\tau_-}\right]\frac{dz}{2\pi i}q_{FH}'(z)\frac{\dot{f}_t(z)}{f_t(z)}=\frac{\mathcal{A}}{\pi}\int_{-1}^1 \frac{\Tree(x)}{\sqrt{1-x^2}}dx-\sum_{j=1}^k \frac{\beta_j}{2}\Tree(x_j).
\end{equation}

Thus combining \eqref{eq:qfhint2}, \eqref{eq:szint}, our reasoning about the $\mathcal{O}(N)$ term, and Lemma \ref{le:di1global}, yields

\begin{align*}
\log D_{N-1}(f_1)-\log D_{N-1}(f_0)&=N\int_{-1}^1 \Tree(x)d(x)\sqrt{1-x^2}dx+\frac{\mathcal{A}}{\pi}\int_{-1}^1 \frac{\Tree(x)}{\sqrt{1-x^2}}dx-\sum_{j=1}^k \frac{\beta_j}{2}\Tree(x_j)\\
& \quad +\frac{1}{4\pi^2}\int_{-1}^1 dy\frac{\Tree(y)}{\sqrt{1-y^2}}P.V.\int_{-1}^1 \frac{\Tree'(x)\sqrt{1-x^2}}{y-x}dx+o(1),
\end{align*}

\noindent where $o(1)$ is uniform in everything relevant. This is precisely the claim.
\end{proof}

\subsection{\texorpdfstring{The differential identity \eqref{eq:di2}}{The second differential identity}}

The main goal of this section is to prove the following identity.

\begin{proposition}\label{prop:di2int}
Let $V$ be one-cut regular, $\Tree$ as in Proposition \ref{prop:fh}, $\delta>0$ small enough but independent of $N$. Then as $N\to\infty$,

\begin{align*}
\log D_{N-1}(f_0;V_1)-\log D_{N-1}(f_0;V_0)&=-\frac{N^2}{2}\int_{-1}^1\left(\frac{2}{\pi}+d(x)\right)(V(x)-2x^2)\sqrt{1-x^2}dx\\
&\quad -\mathcal{A}\frac{N}{\pi}\int_{-1}^1 \frac{V(x)-2x^2}{\sqrt{1-x^2}}dx+N\sum_{j=1}^k\frac{\beta_j}{2}(V(x_j)-2x_j^2)\\
&\quad +\sum_{j=1}^k \frac{\beta_j^2}{4}\log \left(\frac{\pi}{2}d(x_j)\right) - { \frac{1}{24}}\log\left(\frac{\pi^2}{4}d(1)d(-1)\right)+o(1),
\end{align*}

\noindent where $o(1)$ is uniform in $\lbrace (x_j)_{j=1}^k: |x_i-x_j|\geq 3\delta, i\neq j \ \mathrm{and} \ |x_i\pm 1|\geq 3\delta \ \forall i\rbrace$.
\end{proposition}

The arguments are largely similar to those related to the differential identity \eqref{eq:di1} so we will be less detailed here. The arguments in the proof of Lemma \ref{le:di1global} can be repeated in this case with the only difference being that we replace $\partial_t f_t$ by $- Nf\partial_s V_s$ and $d$ with $d_s$ etc, apart from approximating $R$ by the identity -- we'll need the $\mathcal{O}(N^{-1})$ contribution from $R$ here as well. We will also need to assume that our lenses and neighborhoods of the singularities are chosen so that $V$ is analytic in some neighborhood of them, but as we assumed $V$ to be real analytic, we can of course do this. We will also assume that $\tau_\pm$ are inside this domain where $V$ can be analytically continued to. Repeating the arguments from the previous section in such a setting leads to the following lemma.

\begin{lemma}\label{le:di2global}
Let $\tau_\pm$ be as in Lemma \ref{le:di1global} with the difference that we assume that the contours are within the domain where $V$ is analytic in.  

Then for $s\in[0,1]$

\begin{align*}
-\frac{N}{2\pi i}&\int_\R \left[Y_{11}(x;V_s)\partial_x Y_{21}(x;V_s)-Y_{21}(x;V_s)\partial_x Y_{11}(x;V_s)\right]f(x)e^{-NV_s(x)}\partial_s V_s(x) dx\\
&=-N^2\int_{-1}^1 d_s(x)\sqrt{1-x^2}\partial_s V_s(x)dx-\frac{N}{2\pi i}\left[\int_{\tau_+}-\int_{\tau_-}\right]\mathcal{J}_s(z)\partial_s V_s(z)dz+o(1),
\end{align*}

\noindent where $o(1)$ is uniform in $s\in[0,1]$, $\lbrace (x_j)_{j=1}^k: |x_i-x_j|\geq 3\delta, i\neq j \ and \ |x_i\pm 1|\geq 3\delta \ \forall i\rbrace$ and 

\begin{equation*}
\mathcal{J}_s(z)=-Y_{22}(z;V_s)Y_{11}'(z;V_s)+Y_{12}(z;V_s)Y_{21}'(z;V_s).
\end{equation*}
\end{lemma}

The proof is essentially identical to that of Lemma \ref{le:di1global} and we omit it. We now consider the asymptotics of the integral of this from $s=0$ to $s=1$. Let us first consider the order $N^2$ term.

\begin{lemma}\label{le:oderN2}
We have

\begin{equation*}
\int_0^1 ds(-N^2)\int_{-1}^1 d_s(x)\partial_s V_s(x)\sqrt{1-x^2}dx=-\frac{N^2}{2}\int_{-1}^1\left(\frac{2}{\pi}+d(x)\right)(V(x)-2x^2)\sqrt{1-x^2}dx.
\end{equation*}
\end{lemma}

\begin{proof}
This follows immediately from the definitions: $\partial_s V_s(x)=V(x)-2x^2$ and $d_s(x)=(1-s)\frac{2}{\pi}+sd(x)$.

\end{proof}

For $\mathcal{J}$-terms, we note that we now need to take into account $\mathcal{O}(N^{-1})$ terms in the expansion of $R$ -- these will result in $\mathcal{O}(1)$ terms in the differential identity. We first focus on the $\mathcal{O}(N)$ terms which come from the $\mathcal{O}(1)$ terms in the expansion of $R$. For this, repeating our argument from the previous section results in the $\mathcal{O}(N)$ term being

\begin{equation*}
\frac{N}{2\pi i}\int_0^1ds\oint_{\gamma}\frac{\mathcal{D}'(x)}{\mathcal{D}(x)}\partial_s V_s(x)dx=\frac{N}{2\pi i}\oint_{\gamma}\frac{\mathcal{D}'(x)}{\mathcal{D}(x)}(V(x)-2x^2)dx,
\end{equation*}

\noindent where $\gamma$ is a nice curve enclosing $[-1,1]$ inside which everything relevant is analytic. We again have $\mathcal{D}'(z)/\mathcal{D}(z)=q_{Sz}'(z,0)+q_{FH}'(z,0)=q_{FH}'(z,0)$ (as $q_{Sz}(z,0)=0$). Recalling that

\begin{equation*}
q_{FH}'(z)=-\frac{\mathcal{A}}{r(z)}+\sum_{j=1}^k\frac{\beta_j}{2}\frac{1}{z-x_j},
\end{equation*}

\noindent an application of Sokhotski-Plemelj shows that the order $N$ terms combine into the following quantity

\begin{align}\label{eq:di2orderN}
\frac{N}{2\pi i}\oint_{\gamma}\frac{\mathcal{D}'(x)}{\mathcal{D}(x)}(V(x)-2x^2)dx&=-\frac{N}{2\pi i}\int_{-1}^1(q_{FH,+}'(x)-q_{FH,-}'(x))(V(x)-2x^2)dx\\
\notag &=-\mathcal{A}\frac{N}{\pi}\int_{-1}^1 \frac{V(x)-2x^2}{\sqrt{1-x^2}}dx+N\sum_{j=1}^k\frac{\beta_j}{2}(V(x_j)-2x_j^2).
\end{align}

Finally, let us consider the $\mathcal{O}(1)$ terms. We will make use of the following lemma (whose variants are surely well known in the literature, but as we don't know of a reference exactly in our setting we will sketch a proof of it).

\begin{lemma}\label{le:prinint}
For $x\in(-1,1)$ and one-cut regular potential $V$,

\begin{equation}\label{eq:1term}
P.V.\int_{-1}^1 V'(\lambda)\frac{\sqrt{1-\lambda^2}}{\lambda-x}d\lambda=-2\pi+2\pi^2d(x)(1-x^2)
\end{equation}

\noindent and

\begin{equation}\label{eq:2term}
\int_{x}^1 d(\lambda)\sqrt{1-\lambda^2}d\lambda=\frac{\sqrt{1-x^2}}{2\pi^2}P.V.\int_{-1}^1 \frac{V(\lambda)}{x-\lambda}\frac{d\lambda}{\sqrt{1-\lambda^2}}+\frac{1}{\pi}\arccos(x).
\end{equation}

\end{lemma}

\begin{proof}
For \eqref{eq:1term}, define the function $H:(\C\setminus[-1,1])\to \C$

\begin{equation*}
H(z)=2\pi(z-1)^{1/2}(z+1)^{1/2}\int_{-1}^{1}\frac{d(\lambda)\sqrt{1-\lambda^2}}{\lambda-z}d\lambda+\int_{-1}^1 \frac{V'(\lambda)\sqrt{1-\lambda^2}}{\lambda-z}d\lambda.
\end{equation*}

Using Sokhotksi-Plemelj and \eqref{eq:el1}, one can check that this function is continuous across $(-1,1)$. One also sees easily that $H$ is bounded at $\pm 1$ so we conclude that it is entire. Finally as $H(\infty)=-2\pi$, Liouville implies that $H(z)=-2\pi$. An application of Sokhotski-Plemelj then implies \eqref{eq:1term}.

We note that as a consequence of \eqref{eq:1term}, one can check that what's required for \eqref{eq:2term} is to prove the identity

\begin{equation}\label{eq:goal}
p(x):=\int_{x}^1\frac{1}{\sqrt{1-y^2}}P.V.\int_{-1}^1 \frac{V'(\lambda)}{\lambda-y}\sqrt{1-\lambda^2}d\lambda dy=\sqrt{1-x^2}P.V. \int_{-1}^1 \frac{V(\lambda)}{x-\lambda}\frac{d\lambda}{\sqrt{1-\lambda^2}}=:q(x).
\end{equation}

One can easily check that these are both smooth functions of $x$ and satisfy $p(1)=q(1)=0$, so it's enough for us to check that $p'(x)=q'(x)$. For this, let us first write

\begin{equation*}
q(x)=\frac{1}{\sqrt{1-x^2}}P.V.\int_{-1}^1 \frac{V(\lambda)}{x-\lambda}\sqrt{1-\lambda^2}d\lambda-\frac{1}{\sqrt{1-x^2}}\int_{-1}^1 \frac{(x+\lambda)V(\lambda)}{\sqrt{1-\lambda^2}}d\lambda.
\end{equation*}

We again make use of the fact that differentiation commutes with the Hilbert transform so one can check that

\begin{align*}
q'(x)&=p'(x)-\frac{1}{\sqrt{1-x^2}}P.V.\int_{-1}^1 \frac{\lambda V(\lambda)}{x-\lambda}\frac{d\lambda}{\sqrt{1-\lambda^2}}+\frac{x}{(1-x^2)^{3/2}}P.V.\int_{-1}^1 \frac{V(\lambda)}{x-\lambda}\sqrt{1-\lambda^2}d\lambda\\
&\quad -\frac{x}{(1-x^2)^{3/2}}\int_{-1}^1 \frac{(x+\lambda)V(\lambda)}{\sqrt{1-\lambda^2}}d\lambda-\frac{1}{\sqrt{1-x^2}}\int_{-1}^1 \frac{V(\lambda)}{\sqrt{1-\lambda^2}}d\lambda\\
&=p'(x)+\frac{x}{\sqrt{1-x^2}}P.V.\int_{-1}^1 \frac{V(\lambda)}{x-\lambda}\left[-\frac{1}{\sqrt{1-\lambda^2}}+\frac{\sqrt{1-\lambda^2}}{1-x^2}-\frac{x^2-\lambda^2}{(1-x^2)\sqrt{1-\lambda^2}}\right]d\lambda\\
&=p'(x).
\end{align*}

We conclude that $p=q$ and \eqref{eq:2term} is true.

\end{proof}

Now to get a hold of the $\mathcal{O}(1)$-terms we are interested in, we need the $\mathcal{O}(N^{-1})$ term in the expansion of $\Jcal_s$ for the $\tau_\pm$-integrals. Again by Theorem \ref{th:rasy}, we know that
\begin{align*}
R(z) = I + \underbrace{R_1(z)}_{\mathcal{O}(N^{-1})} + o(N^{-1}),
\quad \Rightarrow \quad R(z)^{-1} = I - R_1(z) + o(N^{-1})
\end{align*}

\noindent where the claim about $R^{-1}$ follows by Neumann series expansion. Inspecting \eqref{eq:YY}, one realizes that the extra $\mathcal{O}(N^{-1})$ correction is indeed given by

\begin{equation*}
-\left(\left[P^{(\infty)}\right]^{-1}R_1'P^{(\infty)}\right)_{11}.
\end{equation*}

\noindent Let us consider first the contributions from the $R_1^{(x_j)}$ terms with $j\in\lbrace 1,...,k\rbrace$ (recall Theorem \ref{th:rasy} for the definition of this and $\Jcal^{(x_j)}$ below).

\begin{lemma}\label{le:r1xj}
Let $\tau_\pm$ be as in Lemma \ref{le:szint} and $j\in\lbrace 1,...,k\rbrace$. Then

\begin{equation}\label{eq:diorder1}
-\int_0^1 ds\frac{N}{2\pi i}\left[\int_{\tau_+}-\int_{\tau_-}\right]\Jcal^{(x_j)}(z)\partial_s V_s(z)dz=\frac{\beta_j^2}{4}\log\left[\frac{\pi}{2}d(x_j)\right]+\mathcal{O}(N^{-1})
\end{equation}

\noindent uniformly in $x_j\in(-1+\epsilon,1-\epsilon)$.

\end{lemma}

\begin{proof}
Recall first of all from Theorem \ref{th:rasy} that for $j\in\lbrace 1,...k\rbrace$

\begin{align*}
N\Jcal^{(x_j)}(z)&=-\frac{1}{4}\frac{1}{(z - x_j )^2}\frac{i\beta_j^2}{4 \pi d_s(x_j)\sqrt{1-x_j^2}}\left[\frac{a(z)^2}{a_+(x_j)^2}+\frac{a_+(x_j)^2}{a(z)^2}\right]\\
&\quad + \frac{1}{4}\frac{1}{(z - x_j )^2}\frac{i\beta_j (c_{x_j,s}^2+c_{x_j,s}^{-2})}{4 \pi d_s(x_j)\sqrt{1-x_j^2}}\left[\frac{a(z)^2}{a_+(x_j)^2}-\frac{a_+(x_j)^2}{a(z)^2}\right]
\end{align*}

\noindent where

\begin{align*}
c_{x_j, s} &= \left(x_j + i \sqrt{1-x_j^2}\right)^{\mathcal{A}} \exp \left(-i \sum_{k > j} \beta_k \pi / 2 + N \phi_{s, +}(x_j) - (1+\beta_j) \pi i / 4 \right).
\end{align*}

Let us first focus on the $z$-integral in the statement of the lemma. Note first that

\begin{equation}\label{eq:a+}
\frac{a(z)^2}{a_+(x_j)^2}+\frac{a_+(x_j)^2}{a(z)^2}=\frac{2i(1-x_jz)}{(z-1)^{1/2}(z+1)^{1/2}\sqrt{1-x_j^2}}
\end{equation}

\noindent and

\begin{equation}\label{eq:a-}
\frac{a(z)^2}{a_+(x_j)^2}-\frac{a_+(x_j)^2}{a(z)^2}=\frac{2i(x_j-z)}{(z-1)^{1/2}(z+1)^{1/2}\sqrt{1-x_j^2}}.
\end{equation}

Using \eqref{eq:a+} and \eqref{eq:a-} one can check with direct calculations that

\begin{equation*}
\frac{1}{(x_j-z)^2}\left[\frac{a(z)^2}{a_+(x_j)^2}+\frac{a_+(x_j)^2}{a(z)^2}\right]=\frac{2i}{\sqrt{1-x_j^2}}\frac{d}{dz}\frac{(z-1)^{1/2}(z+1)^{1/2}}{z-x_j}
\end{equation*}

\noindent and

\begin{equation*}
\frac{1}{(x_j-z)^2}\left[\frac{a(z)^2}{a_+(x_j)^2}-\frac{a_+(x_j)^2}{a(z)^2}\right]=\frac{2i}{\sqrt{1-x_j^2}}\frac{1}{x_j-z}\frac{1}{(z-1)^{1/2}(z+1)^{1/2}}.
\end{equation*}

Recalling that $\partial_s V_s(z)=V(z)-2z^2$, we thus see by integration by parts, contour deformation, and Sokhotski-Plemelj that

\begin{align}\label{eq:firstterm}
\left[\int_{\tau_+}-\int_{\tau_-}\right]&\frac{1}{(x_j-z)^2}\left[\frac{a(z)^2}{a_+(x_j)^2}+\frac{a_+(x_j)^2}{a(z)^2}\right]\partial_s V_s(z)\frac{dz}{2\pi i}\\
\notag &=-\frac{1}{\pi}\frac{1}{\sqrt{1-x_j^2}}\left[\int_{\tau_+}-\int_{\tau_-}\right]\frac{(z-1)^{1/2}(z+1)^{1/2}}{z-x_j}(V'(z)-4z)dz\\
\notag &=-\frac{2i}{\pi\sqrt{1-x_j^2}}P.V.\int_{-1}^1 (V'(\lambda)-4\lambda)\frac{\sqrt{1-\lambda^2}}{\lambda-x_j}d\lambda
\end{align}

\noindent and simply by Sokhotski-Plemelj that

\begin{align}\label{eq:secondterm}
\left[\int_{\tau_+}-\int_{\tau_-}\right]&\frac{1}{(x_j-z)^2}\left[\frac{a(z)^2}{a_+(x_j)^2}-\frac{a_+(x_j)^2}{a(z)^2}\right]\partial_s V_s(z)\frac{dz}{2\pi i}\\
\notag &=\frac{2}{\pi i}\frac{1}{\sqrt{1-x_j^2}}P.V.\int_{-1}^1 \frac{V(\lambda)-2\lambda^2}{x_j-\lambda}\frac{d\lambda}{\sqrt{1-\lambda^2}}.
\end{align}

Let us first focus on the integral of the first term. We have from \eqref{eq:firstterm} and \eqref{eq:1term}

\begin{align}\label{eq:plusterm}
-\int_0^1 ds &\left[\int_{\tau_+}-\int_{\tau_-}\right]\left(-\frac{1}{4}\frac{1}{(z - x_j )^2}\frac{i\beta_j^2}{4 \pi d_s(x_j)\sqrt{1-x_j^2}}\left[\frac{a(z)^2}{a_+(x_j)^2}+\frac{a_+(x_j)^2}{a(z)^2}\right]\right)\partial_s V_s(z)\frac{dz}{2\pi i}\\
\notag &=\frac{\beta_j^2}{4}\left(d(x_j)-\frac{2}{\pi}\right)\int_0^1 \frac{ds}{d_s(x_j)}\\
\notag &=\frac{\beta_j^2}{4}\log \left[\frac{\pi}{2}d(x_j)\right].
\end{align}

Let us now turn to the second term. We have from \eqref{eq:secondterm} and \eqref{eq:2term} that

\begin{align*}
-\int_0^1 ds &\left[\int_{\tau_+}-\int_{\tau_-}\right]\left(\frac{1}{4}\frac{1}{(z - x_j )^2}\frac{i\beta_j(c_{x_j,s}^2+c_{x_j,s}^{-2})}{4 \pi d_s(x_j)\sqrt{1-x_j^2}}\left[\frac{a(z)^2}{a_+(x_j)^2}-\frac{a_+(x_j)^2}{a(z)^2}\right]\right)\partial_s V_s(z)\frac{dz}{2\pi i}\\
&=-(1-x_j)^{-3/2}\frac{\beta_j}{4}\int_{x_j}^1 \left(d(\lambda)-\frac{2}{\pi}\right)\sqrt{1-\lambda^2}d\lambda \int_0^1 ds \frac{c_{x_j,s}^2+c_{x_j,s}^{-2}}{d_s(x_j)}.
\end{align*}

Let us note that we can write $c_{x_j,s}^2=e^{i\theta_N(x_j)} e^{2\pi i N s\int_{x_j}^1\left(d(\lambda)-\frac{2}{\pi}\right)\sqrt{1-\lambda^2}}d\lambda$, where $e^{i\theta_N(x_j)}$ is a complex number of unit length and independent of $s$. Thus

\begin{align*}
\int_{x_j}^1&\left(d(\lambda)-\frac{2}{\pi}\right)\sqrt{1-\lambda^2}d\lambda \int_0^1 ds \frac{c_{x_j,s}^{\pm 2}}{d_s(x_j)}{=}\pm e^{\pm i\theta_N(x_j)}\frac{1}{2\pi i N}\int_0^1\frac{1}{d_s(x_j)}\frac{d}{ds}e^{\pm 2\pi i N s \int_{x}^1\left(d(\lambda)-\frac{2}{\pi}\right)\sqrt{1-\lambda^2}d\lambda}ds.
\end{align*}

Integrating this by parts, noting that $\frac{d}{ds}d_s(x)=d(x)-\frac{2}{\pi}$ is bounded and $1/d_s(x)^2$ is bounded in $x$ and $s$, we see that

\begin{equation}\label{eq:minusterm}
-\int_0^1 ds \left[\int_{\tau_+}-\int_{\tau_-}\right]\left(\frac{1}{(z - x_j )^2}\frac{i\beta_j(c_{x_j,s}^2+c_{x_j,s}^{-2})}{d_s(x_j)\sqrt{1-x_j^2}}\left[\frac{a(z)^2}{a_+(x_j)^2}-\frac{a_+(x_j)^2}{a(z)^2}\right]\right)\partial_s V_s(z)dz=\mathcal{O}(N^{-1})
\end{equation}

\noindent uniformly in $x_j\in(-1+\epsilon,1-\epsilon)$. Combining \eqref{eq:plusterm} and \eqref{eq:minusterm}, yields the claim \eqref{eq:diorder1}.
\end{proof}

Let us now treat the integrals associated to $\Jcal^{(\pm 1)}$.

\begin{lemma}\label{le:partition}
We have
\begin{align} \label{eq:partition}
-\int_0^1 ds \frac{N}{2 \pi i} \left[\int_{\tau_+}-\int_{\tau_-}\right] \Jcal^{(1)}(z) \partial_s V_s(z) dz
& = -{\frac{1}{24}}\log\left(\frac{\pi}{2}d(1)\right),\\
\notag -\int_0^1 ds \frac{N}{2 \pi i} \left[\int_{\tau_+}-\int_{\tau_-}\right] \Jcal^{(-1)}(z) \partial_s V_s(z) dz
& = -{\frac{1}{24}}\log\left(\frac{\pi}{2}d(-1)\right).
\end{align}
\end{lemma}

\begin{proof}
We only prove the first equality. From Theorem \ref{th:rasy} we have

\begin{align*}
\Jcal^{(1)}(z) &= {-}\frac{1}{(z-1)^2} \frac{2^{1/2}}{8N}\left\{a(z)^2 \left[ \frac{1}{48} \left(G_s^{(1)}(1)\right)^{-1}(5 + 96 \mathcal{A}^2) - \frac{5}{12} \left(G_s^{(1)}(1)\right)^{-2} \left(\left[G_s^{(1)}\right]'(1)\right)\right] \right.\\
& \qquad \qquad \left.- a(z)^{-2} \frac{7}{24}\left(G_s^{(1)}(1)\right)^{-1}\right\} - \frac{1}{(z-1)^3} \frac{5\sqrt{2}}{48NG_s^{(1)}(1)} a(z)^2
\end{align*}

\noindent where $G_s^{(1)}$ is defined in (\ref{eq:Gs}) and we have $G_s^{(1)}(1) = \pi \sqrt{2} d_s(1)$. Note that

\begin{equation*}
\frac{a(z)^2}{(z-1)^2}=-\frac{d}{dz}\frac{(z+1)^{1/2}}{(z-1)^{1/2}}\qquad \mathrm{and} \qquad \frac{a(z)^2}{(z-1)^3}=\frac{1}{3}\frac{d}{dz}\frac{(z-2)(z+1)^{1/2}(z-1)^{1/2}}{(z-1)^2}.
\end{equation*}

Thus integrating by parts, contour deformation, and a simple application of Lemma \ref{le:prinint} imply that

\begin{align*}
\left[\int_{\tau_+}-\int_{\tau_-}\right]\frac{a(z)^2}{(z-1)^2} V(z) dz &= -\left[\int_{\tau_+}-\int_{\tau_-}\right] V(z) \frac{d}{dz} \left(\frac{z+1}{z-1}\right)^{1/2}dz\\
& = \left[\int_{\tau_+}-\int_{\tau_-}\right]V'(z)  \left(\frac{z+1}{z-1}\right)^{1/2}dz\\
&= 2i  \int_{-1}^1 \frac{\sqrt{1-x^2}}{x-1}V'(x)dx={ -} 4\pi i
\end{align*}

\noindent and
\begin{align}\label{eq:aint1}
\left[\int_{\tau_+}-\int_{\tau_-}\right] \frac{a(z)^2}{(z-1)^2} \partial_s V_s(z) dz
= \left[\int_{\tau_+}-\int_{\tau_-}\right] \frac{a(z)^2}{(z-1)^2} (V(z) - 2z^2) dz =  0.
\end{align}

In a similar manner and with an application of Lemma \ref{le:prinint},

\begin{align*}
\left[\int_{\tau_+}-\int_{\tau_-}\right]\frac{a(z)^2}{(z-1)^3} V(z) dz&=-\left[\int_{\tau_+}-\int_{\tau_-}\right] V'(z) \frac{1}{3}\frac{(z-2)(z+1)^{1/2}(z-1)^{1/2}}{(z-1)^2}dz \\
&=-\frac{1}{3}\left[\int_{\tau_+}-\int_{\tau_-}\right] V'(z) \frac{(z+1)^{1/2}(z-1)^{1/2}}{z-1}dz\\
&\quad +\frac{1}{3} \left[\int_{\tau_+}-\int_{\tau_-}\right]V'(z) \frac{(z+1)^{1/2}(z-1)^{1/2}}{(z-1)^2}dz\\
&=-\frac{1}{3}\left[\int_{\tau_+}-\int_{\tau_-}\right]V'(z) \frac{(z+1)^{1/2}(z-1)^{1/2}}{z-1}dz\\
&\quad +\frac{1}{3}\left.\frac{d}{dx}\right|_{x=1}\left[\int_{\tau_+}-\int_{\tau_-}\right] V'(z) \frac{(z+1)^{1/2}(z-1)^{1/2}}{z-x}dz\\
&=\frac{2i}{3}\int_{-1}^1 V'(\lambda)\sqrt{\frac{1+\lambda}{1-\lambda}}d\lambda+\frac{2i}{3}\left.\frac{d}{dx}\right|_{x=1}P.V.\int_{-1}^1 V'(\lambda)\frac{\sqrt{1-\lambda^2}}{\lambda-x}d\lambda\\
&=\frac{4\pi i}{3}-\frac{8\pi^2 i}{3}d(1),
\end{align*}

\noindent which implies

\begin{equation}\label{eq:aint2}
\left[\int_{\tau_+}-\int_{\tau_-}\right]\frac{a(z)^2}{(z-1)^3} \partial_s V_s(z) dz=-\frac{8\pi^2 i}{3}\left(d(1)-\frac{2}{\pi}\right).
\end{equation}

Consider finally the $a(z)^{-2}$ term. One can easily check that

\begin{align*}
\frac{a(z)^{-2}}{(z-1)^2}
& = -\frac{2}{3} \frac{\partial}{\partial z} \left[ \frac{(z-1)^{1/2}(z+1)^{1/2}}{(z-1)^2} \right] + \frac{1}{3} \frac{a(z)^2}{(z-1)^2} .
\end{align*}

\noindent We can safely ignore the second term on the RHS, as we saw that it will integrate to zero. Moreover, we essentially calculated the integral related to the first term already:

\begin{align*}
-\frac{2}{3} \left[\int_{\tau_+}-\int_{\tau_-}\right]& V(z)  \frac{\partial}{\partial z} \left[ \frac{(z-1)^{1/2}(z+1)^{1/2}}{(z-1)^2} \right] dz
= -\frac{16}{3} \pi^2 i d(1)
\end{align*}

\noindent and we find

\begin{equation}\label{eq:aint3}
\left[\int_{\tau_+}-\int_{\tau_-}\right]\frac{a(z)^{-2}}{(z-1)^2} \partial_s V_s(z) dz=-\frac{16\pi^2 i}{3}\left(d(1)-\frac{2}{\pi}\right).
\end{equation}

Putting together \eqref{eq:aint1}, \eqref{eq:aint2}, and \eqref{eq:aint3} a direct calculation leads to

\begin{align*}
-\int_0^1 ds\frac{N}{2\pi i}\left[\int_{\tau_+}-\int_{\tau_-}\right]\Jcal^{(1)}(z)\partial_s V_s(z)dz=-{\frac{1}{24}}\log\left(\frac{\pi}{2}d(1)\right).
\end{align*}

\end{proof}

\begin{proof}[Proof of Proposition \ref{prop:di2int}]
This is simply a combination of Lemma \ref{le:di2global}, Lemma \ref{le:oderN2}, \eqref{eq:di2orderN}, Lemma \ref{le:r1xj}, and Lemma \ref{le:partition}.
\end{proof}

We are now in a position to apply these results.

\section{Proof of Theorem \ref{th:main}} \label{sec:mainproof}

As discussed earlier, we do this through Proposition \ref{prop:main}. Before proving this, we will need to recall Krasovsky's result for the GUE from \cite{krasovsky} and a result of Claeys and Fahs \cite{cf} which we need to control the situation when the singularities are close to each other. Let us begin with Krasovsky's result \cite[Theorem 1]{krasovsky}.

\begin{theorem}[Krasovsky]\label{th:krasovsky}
Let $(x_j)_{j=1}^k$ be distinct points in $(-1,1)$, let $\beta_j>-1$, and let $H_N$ be a GUE matrix $($i.e. $V(x)=2x^2)$. Then as $N\to\infty$

\begin{align*}
\E &\prod_{j=1}^k |\det(H_N-x_j)|^{\beta_j}\\
&=\prod_{j=1}^k C(\beta_j)(1-x_j^2)^{\frac{\beta_j^2}{8}}\left(\frac{N}{2}\right)^{\frac{\beta_j^2}{4}}e^{(2x_j^2-1-2\log 2)\frac{\beta_j}{2}N}\prod_{i<j}|2(x_i-x_j)|^{-\frac{\beta_i\beta_j}{2}}\left(1+\mathcal{O}(\log N/N)\right)
\end{align*}

\noindent uniformly in compact subsets of $\lbrace (x_1,...,x_k)\in(-1,1)^k: x_i\neq x_j \ \mathrm{for} \ i\neq j\rbrace$. Here $C(\beta)=2^{\frac{\beta^2}{2}}\frac{G(1+\beta/2)^2}{G(1+\beta)}$, and $G$ is the Barnes $G$ function.
\end{theorem}

We mention that Krasovsky's result is actually valid for complex $\beta_j$ with real part greater than $-1$, and he used a slightly different normalization, but obtaining this formulation follows after trivial scaling. Also his formulation of the result does not stress the uniformity, but it can easily be checked through uniform bounds on the jump matrices which are similar to the ones we have considered.

Combining this with Proposition \ref{prop:di2int} yields the following result.

\begin{proposition}\label{prop:1cpurefh}
Let $H_N$ be drawn from a one-cut regular ensemble with potential $V$ and support of the equilibrium measure normalized to $[-1,1]$. If $(x_j)_{j=1}^k$ are distinct points in $(-1,1)$ and $\beta_j\geq 0$ for all $j$, then

\begin{align*}
\E \prod_{j=1}^k |\det(H_N-x_j)|^{\beta_j}&=\prod_{j=1}^k C(\beta_j)\left(d(x_j)\frac{\pi}{2}\sqrt{1-x_j^2}\right)^{\frac{\beta_j^2}{4}}\left(\frac{N}{2}\right)^{\frac{\beta_j^2}{4}}e^{(V(x_j)+\ell_V)\frac{\beta_j}{2}N}\\
&\quad \times \prod_{i<j}|2(x_i-x_j)|^{-\frac{\beta_i\beta_j}{2}}\left(1+o(1))\right)
\end{align*}

\noindent uniformly in compact subsets of $\lbrace (x_1,...,x_k)\in(-1,1)^k:x_i\neq x_j \ \mathrm{for} \ i\neq j\rbrace$.

\end{proposition}

\begin{proof}
Let us write $\E_V$ for the expectation with respect to an ensemble with potential $V$. Note that from \eqref{eq:andr} setting $f=1$, we have

\begin{equation*}
\frac{Z_N(V)}{N!}=D_{N-1}(1;V)
\end{equation*}

\noindent so we see from Proposition \ref{prop:di2int} that for $f(\lambda)=\prod_{j=1}^k |\lambda-x_j|^{\beta_j}$ and $V_0(x)=2x^2$

\begin{align}\label{eq:deltaexp}
\log \E_V &\prod_{j=1}^k |\det(H_N-x_j)|^{\beta_j}-\log \E_{V_0} \prod_{j=1}^k |\det(H_N-x_j)|^{\beta_j}\\
\notag &=\log D_{N-1}(f;V)-\log D_{N-1}(f;V_0)-\log D_{N-1}(1;V)+\log D_{N-1}(1;V_0)\\
\notag &=-N\sum_{j=1}^k\frac{\beta_j}{2}\left[\frac{1}{\pi}\int_{-1}^1 \frac{V(x)-2x^2}{\sqrt{1-x^2}}dx-(V(x_j)-2x_j^2)\right]+\sum_{j=1}^k\frac{\beta_j^2}{4}\log \left(\frac{\pi}{2}d(x_j)\right)+o(1),
\end{align}

\noindent where we have the desired uniformity.

Let us now recall the logarithmic potential of the arcsine law (see e.g. \cite[Section 1.3: Example 3.5]{st}): $\frac{1}{\pi}\int_{-1}^1\log|x-y|/\sqrt{1-x^2}dx=-\log 2$ for all $y\in(-1,1)$. This along with \eqref{eq:el1} imply that

\begin{equation*}
\frac{1}{\pi}\int_{-1}^1 \frac{V(x)}{\sqrt{1-x^2}}dx+\ell_V=-2\log 2.
\end{equation*}

This in turn implies that

\begin{equation*}
(2x_j^2-1-2\log 2)-\frac{1}{\pi}\int_{-1}^1 \frac{V(x)-2x^2}{\sqrt{1-x^2}}dx+(V(x_j)-2x_j^2)=V(x_j)+\ell_V.
\end{equation*}

Combining this with Theorem \ref{th:krasovsky} and \eqref{eq:deltaexp} yields the claim.
\end{proof}

We now recall the result of Claeys and Fahs that we will need, namely \cite[Theorem 1.1]{cf}.

\begin{theorem}[Claeys and Fahs]\label{th:cf}
Let $V$ be one-cut regular and let the support of the associated equilibrium measure be $[a,b]$ with $a<0<b$. Let $\beta>0$, $u>0$, and $f_u(x)=|x^2-u|^\beta$. Then

\begin{align*}
\log D_{N-1}(f_u;V)&=\log D_{N-1}(f_0;V)+\int_0^{s_{N,u}}\frac{\sigma_\beta(s)-\beta^2}{s}ds+\frac{\beta}{2}s_{N,u}\\
&\quad +N\frac{\beta}{2}(V(\sqrt{u})+V(-\sqrt{u})-2V(0))+\mathcal{O}(\sqrt{u})+\mathcal{O}(N^{-1})
\end{align*}

\noindent uniformly as $u\to 0$ and $N\to\infty$. Here

\begin{equation*}
s_{N,u}=-2\pi i N\int_{-\sqrt{u}}^{\sqrt{u}}d(s)\sqrt{(s-a)(b-s)}ds
\end{equation*}

\noindent and $\sigma_\beta(s)$ is analytic on $-i\R_+$, independent of $V,$ $N$, and $u$ and satisfies:

\begin{equation}\label{eq:sigmaasy}
\sigma_\beta(s)=\begin{cases}
\beta^2+o(1), & s\to -i0^+\\
\frac{\beta^2}{2}-\frac{\beta}{2}s+\mathcal{O}(|s|^{-1}), & s\to -i\infty
\end{cases}
\end{equation}

\noindent Moreover, the integral involving $\sigma_\beta$ is taken along $-i\R_+$.

\end{theorem}

Much more is in fact known about $\sigma_\beta$. For example, it is known to satisfy a Painlev\'e V equation. A generalization of it was studied extensively in \cite{ck}. Theorem \ref{th:cf} and Proposition \ref{prop:1cpurefh} let us prove the convergence of $\E[\mu_N(f)^2]$ -- the argument is similar to analogous ones in \cite{cf,ck}.

\begin{proposition}\label{prop:2mom}
Let $\varphi:(-1,1)\to [0,\infty)$ be continuous and have compact support. Moreover, let $\beta\in(0,\sqrt{2})$. Then

\begin{equation*}
\lim_{N\to\infty}\E [\mu_{N,\beta}(\varphi)^2]=\int_{-1}^1\int_{-1}^1 \varphi(x)\varphi(y)(2|x-y|)^{-\frac{\beta^2}{2}}dxdy
\end{equation*}
\end{proposition}

\begin{proof}
This is very similar to the proof of \cite[Corollary 1.11]{cf} where a more general statement was proven for the GUE. Let us fix some small $\epsilon>0$, $\alpha\in(\beta^2/2,1)$, and write the relevant moment in the following way:

\begin{align*}
\E [\mu_N(\varphi)^2]&=\left[\int_{|x-y|\geq \epsilon}+\int_{2N^{-\alpha}\leq |x-y|<\epsilon}+\int_{|x-y|\leq 2N^{-\alpha}}\right]\varphi(x)\varphi(y)\\
& \qquad \times \frac{\E\left[|\det (H_N-x)|^\beta|\det(H_N-y)|^\beta\right]}{\E|\det(H_N-x)|^\beta\E |\det(H_N-y)|^\beta }dxdy\\
&=:A_{N,1}(\epsilon)+A_{N,2}(\epsilon)+A_{N,3}.
\end{align*}

It follows immediately from Proposition \ref{prop:1cpurefh} that if there is some $\epsilon>0$ such that $|x-y|\geq \epsilon$ and $x,y\in(-1+\epsilon,1-\epsilon)$ then uniformly in such $x,y$

\begin{equation*}
\frac{\E\left[|\det (H_N-x)|^\beta|\det(H_N-y)|^\beta\right]}{\E|\det(H_N-x)|^\beta\E |\det(H_N-y)|^\beta }=\frac{1}{(2|x-y|)^{\frac{\beta^2}{2}}}(1+o(1)).
\end{equation*}

As $\varphi$ has compact support in $(-1,1)$, this is precisely the situation for the integral in $A_{N,1}(\epsilon)$. We conclude that

\begin{equation*}
\lim_{N\to\infty}A_{N,1}(\epsilon)=\int_{|x-y|\geq \epsilon}\varphi(x)\varphi(y)\frac{1}{(2|x-y|)^{\frac{\beta^2}{2}}}dxdy\stackrel{\epsilon\to 0^+}{\longrightarrow}\int_{-1}^1\int_{-1}^1\varphi(x)\varphi(y)\frac{1}{(2|x-y|)^{\frac{\beta^2}{2}}}dxdy.
\end{equation*}

Let us now consider $A_{N,3}$. Here we find by Cauchy--Schwarz and Proposition \ref{prop:1cpurefh} that there exists some finite $B(\beta)$ (uniform in the relevant $x,y$) such that

\begin{align*}
\frac{\E_V[|\det(H_N-x)|^\beta |\det(H_N-y)|^\beta]}{\E_V[|\det(H_N-x)|^\beta]\E_V[|\det(H_N-y)|^\beta]} &\leq \frac{\sqrt{\E_V[|\det(H_N-x)|^{2\beta}]\E_V[|\det(H_N-y)|^{2\beta}]}}{\E_V[|\det(H_N-x)|^\beta]\E_V[|\det(H_N-y)|^\beta]}\\
&\leq B(\beta)N^{\beta^2/2}
\end{align*}

\noindent so we see that as $N\to \infty$

\begin{equation*}
A_{N,3}=\int_{|x-y|\leq 2N^{-\alpha}}\varphi(x)\varphi(y)\frac{\E_V[|\det(H_N-x)|^\beta |\det(H_N-y)|^\beta]}{\E_V[|\det(H_N-x)|^\beta]\E_V[|\det(H_N-y)|^\beta]} dxdy\lesssim N^{-\alpha+\frac{\beta^2}{2}}\to 0
\end{equation*}

\noindent since we chose $\alpha>\beta^2/2$.

Thus to conclude the proof, it's enough to show that

\begin{equation*}
\lim_{\epsilon\to 0^+}\limsup_{N\to\infty} A_{N,2}(\epsilon)=0.
\end{equation*}

Let us begin doing this by noting that if we write $u=\frac{(x-y)^2}{4}$ and $V_{x,y}(\lambda)=V(\lambda+(x+y)/2)$, then in the notation of Theorem \ref{th:cf}

\begin{align*}
\E_V \left[|\det (H_N-x)|^\beta |\det(H_N-y)|^\beta\right]=\frac{D_{N-1}(f_u;V_{x,y})}{D_{N-1}(1;V)}.
\end{align*}

This follows from \eqref{eq:evlaw} through the change of variables $\lambda_i=\mu_i+\frac{x+y}{2}$.
 The goal is to make use of Theorem \ref{th:cf} to estimate $D_{N-1}(f_u;V_{x,y})$. There are several issues we need to check to justify this. First of all, we need $V_{x,y}$ to be one-cut regular and the interior of the support of its equilibrium measure to contain the point $0$. This is simple to justify as one can check from the Euler-Lagrange equations that the equilibrium measure associated to $V_{x,y}$ is simply $d(u+\frac{x+y}{2})\sqrt{1-(u+\frac{x+y}{2})^2}du$ and its support is $[-1-\frac{x+y}{2},1-\frac{x+y}{2}]$. The remaining conditions for one-cut regularity are easy to check with this representation.

It is less obvious that we can use Theorem \ref{th:cf} to study the asymptotics of $D_{N-1}(f_u;V_{x,y})$ as now $V_{x,y}$ depends on $x$ and $y$ and we would need the errors in the theorem to be uniform in $V$ as well. As mentioned in \cite{cf} for the GUE, for $x,y\in(-1+\epsilon,1-\epsilon)$, this can be checked by going through the relevant estimates in the proof. This is true also for general one-cut regular ensembles. As checking this may be non-trivial for a reader with little background in Riemann-Hilbert problems, we outline how to do this in Appendix \ref{app:cfcheck}.

We may therefore use Theorem \ref{th:cf}, and so we have

\begin{align*}
\log & \ \E_V[|\det(H_N-x)|^\beta |\det(H_N-y)|^\beta]\\
&=\log D_{N-1}(f_0;V_{x,y})-\log D_{N-1}(1;V)+\int_0^{s_{N,u}}\frac{\sigma_\beta(s)-\beta^2}{s}ds+\frac{\beta}{2}s_{N,u}\\
&\quad +N\frac{\beta}{2}(V_{x,y}(\sqrt{u})+V_{x,y}(-\sqrt{u})-2V_{x,y}(0))+\mathcal{O}(\sqrt{u})+\mathcal{O}(N^{-1}),
\end{align*}

\noindent where the error estimates are uniform in $|x-y|<\epsilon$ and $x,y\in(-1+\epsilon,1-\epsilon)$. Note that now

\begin{align*}
s_{N,u}&=-2\pi iN\int_{-\sqrt{u}}^{\sqrt{u}}d_{x,y}(s)\sqrt{1-\left(s+\frac{x+y}{2}\right)^2}ds\\
&=-4\pi i N\sqrt{u} d\left(\frac{x+y}{2}\right)\sqrt{1-\left(\frac{x+y}{2}\right)^2}+\mathcal{O}(N u)
\end{align*}

\noindent again uniformly in the relevant values of $x$ and $y$.

Recall that we're considering $u$ such that $\sqrt{u}<2\epsilon$ but $\sqrt{u}>N^{-\alpha}$ with $\frac{\beta^2}{2}<\alpha<1$. We then have $s_{N,u}\to -i\infty$ uniformly in the relevant $x,y$ and using \cite[equation (1.26)]{ck} one has

\begin{equation*}
\lim_{N\to\infty}\left[\int_0^{s_{N,u}}\frac{\sigma_{\beta}(s)-\beta^2}{s}ds+\frac{\beta}{2}s_{N,u}+\frac{\beta^{2}}{2}\log |s_{N,u}|\right]=\log \frac{G(1+\frac{\beta}{2})^{4}G(1+2\beta)}{G(1+\beta)^{4}}
\end{equation*}

\noindent uniformly for $x,y\in(-1+\epsilon,1-\epsilon)$ and $2N^{-\alpha}<|x-y|<\epsilon$.

On the other hand, reversing our mapping from $V$ to $V_{x,y}$, we see that

\begin{equation*}
\log D_{N-1}(f_0;V_{x,y})-\log D_{N-1}(1;V)=\log \E_V\left|\det\left(H_N-\frac{x+y}{2}\right)\right|^{2\beta}.
\end{equation*}

Thus we see that uniformly for $x,y\in(-1+\epsilon,1-\epsilon)$ and $2N^{-\alpha}<|x-y|<\epsilon$

\begin{align*}
\log & \ \E_V[|\det(H_N-x)|^\beta |\det(H_N-y)|^\beta]\\
&=\log \E_V\left|\det\left(H_N-\frac{x+y}{2}\right)\right|^{2\beta}+\log \frac{G(1+\frac{\beta}{2})^{4}G(1+2\beta)}{G(1+\beta)^{4}}\\
&\quad -\frac{\beta^2}{2}\log \left[4\pi N\sqrt{u}d\left(\frac{x+y}{2}\right)\sqrt{1-\left(\frac{x+y}{2}\right)^2}\right]+N\frac{\beta}{2} (V_{x,y}(\sqrt{u})+V_{x,y}(-\sqrt{u})-2V_{x,y}(0))\\
&\quad +\mathcal{O}(\sqrt{u})+o(1),
\end{align*}

\noindent where $o(1)$ means something that tends to zero as $N\to\infty$. Using these estimates, we can write for such $x,y$

\begin{align*}
&\frac{\E_V[|\det(H_N-x)|^\beta |\det(H_N-y)|^\beta]}{\E_V[|\det(H_N-x)|^\beta]\E_V[|\det(H_N-y)|^\beta]}\\
&\quad =\frac{G(1+\frac{\beta}{2})^{4}G(1+2\beta)}{G(1+\beta)^{4}}\frac{\E_V\left|\det\left(H_N-\frac{x+y}{2}\right)\right|^{2\beta}}{\E_V[|\det(H_N-x)|^\beta]\E_V[|\det(H_N-y)|^\beta]}\\
&\quad\quad \times N^{-\frac{\beta^2}{2}}(2|x-y|)^{-\frac{\beta^2}{2}}\left[\pi d\left(\frac{x+y}{2}\right)\sqrt{1-\left(\frac{x+y}{2}\right)^2}\right]^{-\frac{\beta^2}{2}}\\
&\quad \quad \times e^{\frac{N\beta}{2}\left(V_{x,y}(\sqrt{u})+V_{x,y}(-\sqrt{u})-2V_{x,y}(0)\right)} e^{\mathcal{O}(\sqrt{u})}(1+o(1))
\end{align*}

\noindent uniformly in $x,y\in(-1+\epsilon,1-\epsilon)$ and $2N^{-\alpha}<|x-y|<\epsilon$. Plugging in Proposition \ref{prop:1cpurefh}, we see that this becomes

\begin{align*}
&\frac{\E_V[|\det(H_N-x)|^\beta |\det(H_N-y)|^\beta]}{\E_V[|\det(H_N-x)|^\beta]\E_V[|\det(H_N-y)|^\beta]} \\
&\quad=\frac{\left(d\left(\frac{x+y}{2}\right)\sqrt{1-\left(\frac{x+y}{2}\right)^2}\right)^{\beta^2/2}}{\left(d(x)\sqrt{1-x^2}d(y)\sqrt{1-y^2}d(y)\right)^{\frac{\beta^2}{4}}}(2|x-y|)^{-\frac{\beta^2}{2}}(1+o(1))(1+\mathcal{O}(\sqrt{u}))\\
&\quad =(2|x-y|)^{-\frac{\beta^2}{2}}(1+o(1))(1+\mathcal{O}(\sqrt{u})).
\end{align*}

We conclude that

\begin{equation*}
\lim_{\epsilon\to 0^+}\limsup_{N\to\infty}\int_{2N^{-\alpha}<|x-y|<\epsilon}\varphi(x)\varphi(y)\frac{\E_V[|\det(H_N-x)|^\beta |\det(H_N-y)|^\beta]}{\E_V[|\det(H_N-x)|^\beta]\E_V[|\det(H_N-y)|^\beta]} dxdy=0,
\end{equation*}

\noindent which was the missing part of the proof.
\end{proof}

Next we need to study the cross term $\E \mu_{N,\beta}(\varphi)\widetilde{\mu}_{N,\beta}^{(M)}(\varphi)$ along with the fully truncated term $\E[\widetilde{\mu}_{N,\beta}^{(M)}(\varphi)^2]$. For this, we need Proposition \ref{prop:fh}, so let us finish the proof of it.

\begin{proof}[Proof of Proposition \ref{prop:fh}]
We have now

\begin{equation*}
\E e^{\sum_{j=1}^N\Tree (\lambda_j)}\prod_{j=1}^k |\det (H_N-x_j)|^{\beta_j}=\frac{D_{N-1}(f;V)}{D_{N-1}(1;V)},
\end{equation*}

\noindent where $f(\lambda)=f_1(\lambda)=e^{\Tree(\lambda)}\prod_{j=1}^k|\lambda-x_j|^{\beta_j}$. Since we know the asymptotics of this for $\Tree=0$, we can apply Proposition \ref{prop:di1int} to get the relevant asymptotics for $\Tree\neq 0$:

\begin{align*}
\frac{D_{N-1}(f_1;V)}{D_{N-1}(1;V)}&=\frac{D_{N-1}(f_0;V)}{D_{N-1}(1;V)}e^{N\int_{-1}^1 \Tree(x)d(x)\sqrt{1-x^2}dx+\sum_{j=1}^k \frac{\beta_j}{2}\left[\int_{-1}^1 \frac{\Tree(x)}{\pi \sqrt{1-x^2}}dx-\Tree(x_j)\right]}\\
&\quad \times e^{\frac{1}{4\pi^2}\int_{-1}^1 dy\frac{\Tree(y)}{\sqrt{1-y^2}}P.V.\int_{-1}^1 \frac{\Tree'(x)\sqrt{1-x^2}}{y-x}dx}(1+o(1))
\end{align*}

\noindent uniformly in everything relevant. Applying Proposition \ref{prop:1cpurefh} to this yields the claim.
\end{proof}

We now apply this to understanding the remaining terms.

\begin{proposition}\label{prop:mixedmom}
Let $\beta\in(0,\sqrt{2})$ and $\varphi:(-1,1)\to [0,\infty)$ be continuous with compact support. Then for fixed $M\in \Z_+$

\begin{equation*}
\lim_{N\to\infty} \E [\mu_{N,\beta}(\varphi)\widetilde{\mu}_{N,\beta}^{(M)}(\varphi)]=\lim_{N\to\infty}\E[\widetilde{\mu}_{N,\beta}^{(M)}(\varphi)^2]=\int_{-1}^1 \int_{-1}^1 \varphi(x)\varphi(y)e^{\beta^2\sum_{k=1}^M \frac{1}{k}T_k(x)T_k(y)}dxdy.
\end{equation*}
\end{proposition}

\begin{proof}
Let us first consider the cross term. We write this as

\begin{equation*}
\E [\mu_{N,\beta}(\varphi)\widetilde{\mu}_{N,\beta}^{(M)}(\varphi)]=\int_{-1}^1\int_{-1}^1 \varphi(x)\varphi(y)\frac{\E |\det(H_N-x)|^\beta e^{\beta \widetilde{X}_{N,M}(y)}}{\E |\det(H_N-x)|^\beta \E e^{\beta \widetilde{X}_{N,M}(y)}}dxdy.
\end{equation*}

Let us begin by calculating the numerator. Note that as we have only one singularity, Proposition \ref{prop:fh} gives us asymptotics which are uniform in $x$ throughout the whole integration region. To apply Proposition \ref{prop:fh}, we point out that we now have $\Tree(\lambda)=\Tree(\lambda;y)=-\beta\sum_{k=1}^{M}\frac{2}{k}\widetilde{T}_k(\lambda)T_k(y)$. We need uniformity in $y$, but this is ensured by the fact that in a neighborhood of $[-1,1]$, $\Tree$ is a polynomial of fixed degree and its coefficients are uniformly bounded for fixed $M$. Using the facts that $\int_{-1}^1 T_k(y)/\sqrt{1-y^2}dy=0$ for $k\geq 1$, $P.V.\frac{1}{\pi}\int_{-1}^1 T_k'(y)\sqrt{1-y^2}/(x-y)dy=kT_k(x)$, and the orthogonality of the Chebyshev polynomials: $2\int_{-1}^1 T_k(\lambda)T_l(\lambda)/(\pi \sqrt{1-\lambda^2})d\lambda=\delta_{k,l}$ for $k,l\geq 1$, we see that

\begin{align*}
\E [|\det(H_N-x)|^\beta e^{\beta \widetilde{X}_{N,M}(y)}]&=\E [|\det(H_N-x)|^\beta] e^{-\beta N\sum_{k=1}^M \frac{2}{k}T_k(y)\int_{-1}^1 T_k(\lambda)d(\lambda)\sqrt{1-\lambda^2}d\lambda}\\
&\quad \times e^{\frac{\beta^2}{2}\sum_{k=1}^M \frac{1}{k}T_k(y)^2+\beta^2\sum_{k=1}^M \frac{1}{k}T_k(x)T_k(y)}(1+o(1))
\end{align*}

\noindent uniformly in $x,y\in(-1+\epsilon,1-\epsilon)$. We see that the $\E[|\det(H_N-x)|^\beta]$-term in the denominator will cancel, but we still need to understand the $\E e^{\beta \widetilde{X}_{N,M}(y)}$-term. This now has no singularity, so we get the asymptotics from Proposition \ref{prop:fh} by setting $\beta_j=0$ for all $j$. Thus we find with a similar argument that

\begin{equation*}
\E e^{\beta \widetilde{X}_{N,M}(y)}=e^{-\beta N\sum_{k=1}^M \frac{2}{k}T_k(y)\int_{-1}^1 T_k(\lambda)d(\lambda)\sqrt{1-\lambda^2}d\lambda+\frac{\beta^2}{2}\sum_{k=1}^M \frac{1}{k}T_k(y)^2}(1+o(1)),
\end{equation*}

\noindent uniformly in $y$, and we conclude that

\begin{equation*}
\lim_{N\to\infty}\E [\mu_{N,\beta}(\varphi)\widetilde{\mu}_{N,\beta}^{(M)}(\varphi)]=\int_{-1}^1 \int_{-1}^1 f(x)f(y)e^{\beta^2\sum_{k=1}^M \frac{1}{k}T_k(x)T_k(y)}dxdy.
\end{equation*}

For the fully truncated term one argues in a similar way: in this case 
$$\Tree(\lambda)=\Tree(\lambda;x,y)=-\beta \sum_{j=1}^M\frac{2}{j}\widetilde{T}_j(\lambda)(T_j(x)+T_j(y))$$ 
and only the part quadratic in $\Tree$ affects the leading order asymptotics. Going through the calculations one finds

\begin{equation*}
\lim_{N\to\infty}\E [\widetilde{\mu}_{N,\beta}^{(M)}(\varphi)^2]=\int_{-1}^1 \int_{-1}^1 \varphi(x)\varphi(y)e^{\beta^2\sum_{k=1}^M \frac{1}{k}T_k(x)T_k(y)}dxdy.
\end{equation*}
\end{proof}

Before proving Proposition \ref{prop:main}, we need to know that $\mu_\beta$ exists, namely we need to prove Lemma \ref{L:GMC}.

\begin{proof}[Proof of Lemma \ref{L:GMC}]
As discussed earlier, this boils down to showing that $(\mu_\beta^{(M)}(\varphi))_{M=1}^\infty$ is bounded in $L^2$ for continuous $\varphi:[-1,1]\to [0,\infty)$. From the definition of $\mu_\beta^{(M)}$ (see \eqref{eq:gmctrunc}), we see that

\begin{equation*}
\E[\mu_\beta^{(M)}(\varphi)^2]=\int_{-1}^1 \int_{-1}^1 \varphi(x)\varphi(y)e^{\beta^2 \sum_{j=1}^M \frac{1}{j}T_j(x)T_j(y)}dxdy.
\end{equation*}

Now from Proposition \ref{prop:2mom} and Proposition \ref{prop:mixedmom}, we see that if $\varphi$ had compact support in $(-1,1)$, then

\begin{align*}
0\leq \lim_{N\to \infty}\E[(\mu_{N,\beta}(\varphi)-\widetilde{\mu}_{N,\beta}^{(M)}(\varphi))^2]&=\int_{-1}^1\int_{-1}^1 \frac{\varphi(x)\varphi(y)}{|2(x-y)|^{\beta^2/2}}dxdy\\
&\quad -\int_{-1}^1\int_{-1}^1 \varphi(x)\varphi(y) e^{\beta^2\sum_{k=1}^M \frac{1}{k}T_k(x)T_k(y)}dxdy,
\end{align*}

\noindent so for fixed $M\in \Z_+$ and continuous, compactly supported in $(-1,1)$, non-negative $\varphi$

\begin{equation*}
\int_{-1}^1\int_{-1}^1 \varphi(x)\varphi(y) e^{\beta^2\sum_{k=1}^M \frac{1}{k}T_k(x)T_k(y)}dxdy\leq \int_{-1}^1\int_{-1}^1 \frac{\varphi(x)\varphi(y)}{|2(x-y)|^{\beta^2/2}}dxdy<\infty
\end{equation*}

\noindent as $\beta^2/2<1$. For continuous $\varphi:[-1,1]\to [0,\infty)$, we get the same inequality simply by approximating $\varphi$ by a compactly supported one. We conclude that $\mu_\beta^{(M)}(\varphi)$ is indeed bounded in $L^2$ and thus (as it is a martingale as a function of $M$), a limit $\mu_\beta (\varphi)$ exists in $L^2(\mathbb{P})$.
\end{proof}

We are now in a position to prove Proposition \ref{prop:main}.

\begin{proof}[Proof of Proposition \ref{prop:main}]
As noted, Proposition \ref{prop:2mom} and Proposition \ref{prop:mixedmom} imply that

\begin{equation*}
\lim_{N\to\infty}\E[(\mu_{N,\beta}(\varphi)-\widetilde{\mu}_{N,\beta}^{(M)}(\varphi))^2]=\int_{-1}^1\int_{-1}^1 \varphi(x)\varphi(y)\left[\frac{1}{|2(x-y)|^{\beta^2/2}}-e^{\beta^2\sum_{k=1}^M \frac{1}{k}T_k(x)T_k(y)}\right]dxdy.
\end{equation*}

As this is a limit of a second moment, it is non-negative and we see that

\begin{equation*}
\limsup_{M\to\infty}\int_{-1}^1\int_{-1}^1 \varphi(x)\varphi(y)e^{\beta^2\sum_{k=1}^M \frac{1}{k}T_k(x)T_k(y)}dxdy\leq \int_{-1}^1\int_{-1}^1 \varphi(x)\varphi(y)(2|x-y|)^{-\frac{\beta^2}{2}}dxdy.
\end{equation*}

On the other hand, Lemma \ref{le:haagerup} and Fatou's lemma imply that

\begin{equation*}
\int_{-1}^1\int_{-1}^1 \varphi(x)\varphi(y)(2|x-y|)^{-\frac{\beta^2}{2}}dxdy\leq \liminf_{M\to\infty}\int_{-1}^1\int_{-1}^1 \varphi(x)\varphi(y)e^{\beta^2\sum_{k=1}^M \frac{1}{k}T_k(x)T_k(y)}dxdy,
\end{equation*}

\noindent so we see actually that

\begin{equation*}
\lim_{M\to\infty}\lim_{N\to\infty}\E[(\mu_{N,\beta}(\varphi)-\widetilde{\mu}_{N,\beta}^{(M)}(\varphi))^2]=0.
\end{equation*}

We still need to prove that when we first let $N\to\infty$ and then $M\to\infty$, $\widetilde{\mu}_{N,\beta}^{(M)}(\varphi)$ converges in law to $\mu_\beta(\varphi)$. As $\mu_\beta(\varphi)$ is constructed as a limit of $\mu_\beta^{(M)}(\varphi)$, this will follow from showing that $\widetilde{\mu}_{N,\beta}^{(M)}(\varphi)$ converges to $\mu_{\beta}^{(M)}(\varphi)$ in law if we let $N\to\infty$ for fixed $M$. For this, consider the function $F:\R^M\to [0,\infty)$

\begin{equation*}
F(u_1,...,u_M)=\int_{-1}^1\varphi(\lambda) e^{\beta\sum_{k=1}^M\frac{1}{\sqrt{k}}u_kT_k(\lambda)-\frac{\beta^2}{2}\sum_{k=1}^M \frac{1}{k}T_k(\lambda)^2}d\lambda.
\end{equation*}

We now have

\begin{equation*}
F\left(\left(-\frac{2}{\sqrt{k}}\mathrm{Tr}\widetilde{T}_k(H_N)+\frac{2}{\sqrt{k}}N\int_{-1}^1 T_k(\lambda)\mu_V(d\lambda)\right)_{k=1}^M\right)=\widetilde{\mu}_{N,\beta}^{(M)}(\varphi)(1+o(1)),
\end{equation*}

\noindent where $o(1)$ is deterministic. Moreover, if $(A_k)_{k=1}^M$ are the i.i.d. standard Gaussians used in the definition of $\mu_{\beta}^{(M)}$, then $F(A_1,...,A_M)=\mu_{\beta}^{(M)}(\varphi)$. It follows easily from the dominated convergence theorem that $F$ is a continuous function, so if we knew that

$$\left(-\frac{2}{\sqrt{k}}\mathrm{Tr}\widetilde{T}_k(H_N)+\frac{2}{\sqrt{k}}N\int_{-1}^1 T_k(\lambda)\mu_V(d\lambda)\right)_{k=1}^M\stackrel{d}{\to}(A_1,...,A_M)$$

\noindent as $N\to\infty$, we would be done. This is of course well known and follows from more general results such as \cite{johansson} for polynomial potentials or \cite{bg} for more general ones. Nevertheless, we point out that it also follows from our analysis. If one looks at the function $\Tree(\lambda)=\sum_{j=1}^M \alpha_j \frac{2}{\sqrt{j}}(\widetilde{T}_j(\lambda)-\int T_j(u)\mu_V(du))$, one can then check that it follows from Proposition \ref{prop:fh} (setting $\beta_j=0$ for all $j$) that

\begin{equation*}
\E e^{\sum_{j=1}^N \Tree(\lambda_j)}=e^{\frac{1}{2}\sum_{k=1}^{M}\alpha_j^{2}},
\end{equation*}

\noindent which implies the claim.
\end{proof}

Theorem \ref{th:main} is essentially a direct corollary of Proposition \ref{prop:main}.

\begin{proof}[Proof of Theorem \ref{th:main}]
It is a standard probabilistic argument that Proposition \ref{prop:main} implies that also $\mu_{N,\beta}(\varphi)$ converges in law to $\mu_\beta(\varphi)$ as $N\to\infty$ (for compactly supported continuous $\varphi:(-1,1)\to [0,\infty)$) -- see e.g. \cite[Theorem 4.28]{kallenberg2}. Upgrading to weak convergence is actually also very standard. One can simply approximate general continuous $\varphi:[-1,1]\to [0,\infty)$ by ones with compact support in $(-1,1)$ and argue by Markov's inequality. For further details, we refer to e.g. \cite[Section 4]{kallenberg1}.
\end{proof}

\appendix

\section{Proof of differential identities}\label{app:di}

In this appendix we prove Lemma \ref{le:di} and Lemma \ref{le:di2}.

\begin{proof}[Proof of Lemma \ref{le:di}]
First of all, note that all of the appearing objects are differentiable functions of $t$ as can be seen from the determinantal representation of the polynomials \eqref{eq:gramdet}.

Recall from \eqref{eq:dchi} that $\log D_{j}(f_t)=-2\sum_{k=0}^j \log\chi_k(f_t)$. Also from \eqref{eq:ortho}, we see that all polynomials of degree less than $j$ are orthogonal to $p_j$, so

\begin{equation*}
\int_\R \chi_j(f_t)x^j p_j(x;f_t)f_t(x)e^{-NV(x)}dx=1
\end{equation*}

\noindent and

\begin{align*}
\int \left[\partial_t p_j(x;f_t)\right]p_j(x;f_t)f_t(x)e^{-NV(x)}dx&=\int \left[\partial_t\chi_j(f_t)\right]x^j p_j(x;f_t)f_t(x)e^{-NV(x)}dx\\
\notag &=\frac{\partial_t \chi_j(f_t)}{\chi_j(f_t)}.
\end{align*}

Thus we see that

\begin{equation}\label{eq:dlogt}
\partial_t \log D_j(f_t)=-\int \partial_t\left[\sum_{l=0}^j p_l(x;f_t)^2\right]f_t(x)e^{-NV(x)}dx.
\end{equation}

The Christoffel-Darboux identity (see e.g. \cite[page 55]{deift}) states that

\begin{equation}\label{eq:cd}
\sum_{l=0}^j p_l(x;f_t)^2=\frac{\chi_j(f_t)}{\chi_{j+1}(f_t)}[p_{j+1}'(x;f_t)p_j(x;f_t)-p_j'(x;f_t)p_{j+1}(x;f_t)].
\end{equation}

Here $'$ denotes differentiation with respect to $x$. Plugging this into \eqref{eq:dlogt}, we see that

\begin{align*}
\partial_t \log D_j(f_t)&=-\int\partial_t\left[\frac{\chi_j(f_t)}{\chi_{j+1}(f_t)}[p_{j+1}'(x;f_t)p_j(x;f_t)-p_j'(x;f_t)p_{j+1}(x;f_t)]\right] f_t(x)e^{-NV(x)}dx\\
\notag &=-\partial_t\int\frac{\chi_j(f_t)}{\chi_{j+1}(f_t)}[p_{j+1}'(x;f_t)p_j(x;f_t)-p_j'(x;f_t)p_{j+1}(x;f_t)] f_t(x)e^{-NV(x)}dx\\
\notag & \qquad +\int\frac{\chi_j(f_t)}{\chi_{j+1}(f_t)}[p_{j+1}'(x;f_t)p_j(x;f_t)-p_j'(x;f_t)p_{j+1}(x;f_t)] \partial_t f_t(x)e^{-NV(x)}dx.
\end{align*}

Using \eqref{eq:ortho}, one finds that the first integral equals $j+1$ (note that the term corresponding to $p'_j p_{j+1}$ integrates to zero by orthogonality)
so its derivative equals zero. Recalling that for $Y(z,t)=Y_{j+1}(z,t)$, we have

\begin{equation*}
Y(z,t)=\begin{pmatrix}
\frac{1}{\chi_{j+1}(f_t)}p_{j+1}(z,f_t) & *\\
-2\pi i \chi_j(f_t)p_j(z,f_t) & *
\end{pmatrix},
\end{equation*}

\noindent where we ignore the second column of the matrix as it's not relevant right now. Thus we see the claim by replacing $p_j$ and $p_{j+1}$ by the entries of $Y$ and setting $j=N-1$.
\end{proof}

We now prove our second differential identity.

\begin{proof}[Proof of Lemma \ref{le:di2}]
The beginning of the proof is identical to the proof of Lemma \ref{le:di}. Indeed, we can repeat everything up to \eqref{eq:dlogt} to get

\begin{equation*}
\partial_s \log D_j(f,V_s)=-\int_\R \partial_s\left[\sum_{l=0}^{j}p_l(x;f,V_s)^2\right]f(x)e^{-NV_s(x)}dx.
\end{equation*}

Again making use of Christoffel-Darboux and orthogonality, we find

\begin{align*}
&\partial_s \log D_j(f;V_s)\\
&=\int\frac{\chi_j(f;V_s)}{\chi_{j+1}(f;V_s)}[p_{j+1}'(x;f,V_s)p_j(x;f,V_s)-p_j'(x;f,V_s)p_{j+1}(x;f,V_s)]f(x) \partial_s e^{-NV_s(x)}dx,
\end{align*}

\noindent which yields the claim when we set $j=N-1$.

\end{proof}

\section{Proofs for the first transformation}\label{app:trans1}

In this appendix we prove Lemma \ref{le:gbv}, Lemma \ref{le:hsign}, and Lemma \ref{le:trhp}. Variants of Lemma \ref{le:gbv} are certainly well known in Riemann-Hilbert literature (see e.g. \cite[Proposition 5.4]{dkmlvz}), but to have it in precisely the form we need it, we sketch a proof.

\begin{proof}[Proof of Lemma \ref{le:gbv}]
The first statement -- \eqref{eq:gbv1} -- is simply linearity and making use of the fact that for the GUE, one has $\ell_{GUE}=-1-2\log 2$ in our normalization. This amounts to simply calculating the logarithmic potential (or noncommutative entropy) of the semi-circle law. This is a standard calculation and we omit the proof, see e.g. Theorem 4.1 in \cite{gp} or alternatively one can integrate \eqref{eq:el1} against the arcsine law and use the logarithmic potential of the arcsine law \cite[Section 1.3: Example 3.5]{st}.

For \eqref{eq:gbv2} consider first the case where $|\lambda|-1>M$. Here we note that $g_{s,+}(\lambda)+g_{s,-}(\lambda)=2\log |\lambda|+\mathcal{O}(1)$ as $|\lambda|\to \infty$ (uniformly in $s$), but we know that $V(\lambda)/\log |\lambda|\to \infty$ as $|\lambda|\to \infty$, so we see that by choosing $M$ large enough (independent of $s$), $g_{s,+}(\lambda)+g_{s,-}(\lambda)-V_s(\lambda)-\ell_s\leq -\log |\lambda|$.

For the $|\lambda|-1<M$-case, note that the left side of \eqref{eq:gbv2} is a continuous function, and if we take $M'<M$, then our function is a continuous function which is (uniformly in $s$) negative on $[M',M]$. Thus it's enough to consider the situation where $M$ is small. In particular, we can assume it's so small, that $d$ is positive in $|\lambda|-1\in(0,M)$. Let us focus on the $\lambda>1$ case. The $\lambda<-1$ case is similar.

Let us suppress the dependence on $s$ and write $F(\lambda)=g_+(\lambda)+g_-(\lambda)-V(\lambda)-\ell$. As $F(1)=0$, we have by using the Euler-Lagrange equation \eqref{eq:el1} at the point $x=1$

\begin{align*}
F(\lambda)=F(\lambda)-F(1)&=2\int_{-1}^1 (\log (\lambda-x)-\log (1-x))\mu_V(dx)-V'(1)(\lambda-1)+\mathcal{O}((\lambda-1)^2)\\
&=2\int_{-1}^1\int_1^\lambda \frac{du}{u-x}\mu_V(dx)-2\int_{-1}^1\frac{\lambda-1}{1-x}\mu_V(dx)+\mathcal{O}((\lambda-1)^2)\\
&=2\int_{-1}^1 \int_1^\lambda \left[\frac{1}{u-x}-\frac{1}{1-x}\right]du\mu_V(dx)+\mathcal{O}((\lambda-1)^2)\\
&=-2\int_1^\lambda(u-1)\int_{-1}^1\frac{d(x)\sqrt{1-x^2}}{(u-x)(1-x)}dxdu+\mathcal{O}((\lambda-1)^2).
\end{align*}

In the $x$-integral, let us make the change of variables, $1-x=(u-1)y$. We find

\begin{align*}
\int_{-1}^1\frac{d(x)\sqrt{1-x^2}}{(u-x)(1-x)}dx&=(u-1)\int_0^{\frac{2}{u-1}}\frac{d(1-(u-1)y)\sqrt{(u-1)y}\sqrt{2-(u-1)y}}{(u-1)^2y(1+y)}dy\\
&=\sqrt{2}d(1)(u-1)^{-1/2}\int_0^{\frac{2}{u-1}}\frac{dy}{\sqrt{y}(1+y)}+\mathcal{O}\left(\int_0^{\frac{2}{u-1}}\frac{\sqrt{(u-1)y}}{(1+y)}dy\right)\\
&=\mathcal{O}((u-1)^{-1/2}).
\end{align*}

We conclude that $F(\lambda)=-\int_1^\lambda\mathcal{O}(\sqrt{u-1})du+\mathcal{O}((\lambda-1)^2)$ which implies the claim in \eqref{eq:gbv2}.

For \eqref{eq:gbv3}, we note that for $\lambda\in \R$ and $x\in(-1,1)$

\begin{equation*}
\lim_{\epsilon\to 0^+}[\log(\lambda+i\epsilon-x)-\log(\lambda-i\epsilon-x)]=\begin{cases}
2\pi i, & \lambda<x\\
0, & \lambda>x
\end{cases}.
\end{equation*}

Thus for $\lambda\in \R$

\begin{equation*}
g_{s,+}(\lambda)-g_{s,-}(\lambda)=\begin{cases}
2\pi i, & \lambda<-1\\
2\pi i\int_{\lambda}^1 \left[(1-s)\frac{2}{\pi}+sd(x)\right]\sqrt{1-x^2}dx, & |\lambda|<1\\
0, & \lambda>1
\end{cases}
\end{equation*}

\noindent which is \eqref{eq:gbv3}.
\end{proof}

We now move on to prove Lemma \ref{le:hsign}.

\begin{proof}[Proof of Lemma \ref{le:hsign}]
Let $\lambda\in (-1,1)$ and $\epsilon>0$ be small. We have

\begin{align*}
h_s(\lambda+i\epsilon)&=-2\pi i \int_{1}^\lambda\left[(1-s)\frac{2}{\pi}+s d(x)\right]\sqrt{1-x^2}dx\\
\notag & \quad -2\pi i\int_0^{\epsilon}\left[(1-s)\frac{2}{\pi}+s d(\lambda+iu)\right]\sqrt{1-(\lambda+iu)^2}idu.
\end{align*}

The first term is purely imaginary. The second term is an analytic function of $\epsilon$ (in a small enough $\lambda$-dependent neighborhood of the origin), it vanishes at $\epsilon=0$, its derivative at $\epsilon=0$ is positive, and second derivative in a neighborhood of zero is bounded. From this one can conclude that for small enough $\epsilon>0$, the real part of $h_s(\lambda+i\epsilon)>0$. A similar argument works for the claim about the real part of $h_s(\lambda-i\epsilon)$. Such an argument is easily extended into a uniform one in this case.
\end{proof}

Finally we prove Lemma \ref{le:trhp}.

\begin{proof}[Proof of Lemma \ref{le:trhp}]
Uniqueness can be argued as for $Y$. The analyticity condition comes from analyticity of $Y$ and $g_s$, so let us look at the jump conditions. Consider first $\lambda\in(-1,1)$. Then from \eqref{eq:tdef}, \eqref{eq:Yjump}, \eqref{eq:gbv3}, \eqref{eq:gbv1}, and some elementary matrix calculations one finds

\begin{align*}
T_+(\lambda)&=e^{-N\ell_s\sigma_3/2}Y_-(z)\begin{pmatrix}
1 & f_t(\lambda) e^{-NV_s(\lambda)}\\
0 & 1
\end{pmatrix}
e^{-N\left(g_{s,-}(\lambda)+2\pi i \int_\lambda^1\left[(1-s)\frac{2}{\pi}+sd(x)\right]\sqrt{1-x^2}dx-\ell_s/2\right)\sigma_3}\\
&=T_-(\lambda)\begin{pmatrix}
1 & e^{2Ng_{s,-}(\lambda)-N\ell_s}f_t(\lambda) e^{-NV_s(\lambda)}\\
0 & 1
\end{pmatrix}e^{-Nh_s(\lambda)\sigma_3}\\
&=T_-(\lambda)\begin{pmatrix}
e^{-Nh_s(\lambda)} & f_t(\lambda)\\
0 & e^{Nh_s(\lambda)}.
\end{pmatrix}
\end{align*}

For $|\lambda|>1$, we note that by \eqref{eq:gbv3}, $g_{s,+}(\lambda)-g_{s,-}(\lambda)\in\lbrace 0,2\pi i\rbrace$, and a similar argument results in

\begin{equation*}
T_+(\lambda)=T_-(\lambda)\begin{pmatrix}
1 & e^{N(g_{+,s}(\lambda)+g_{s,-}(\lambda)-\ell_s-V_s(\lambda))}f_t(\lambda)\\
0 & 1
\end{pmatrix}
\end{equation*}

\noindent which is precisely \eqref{eq:tjump2}.

For the behavior at infinity, we note that as $z\to\infty$ (uniformly for $z$ not on the negative real axis) $g_s(z)=\log z+\mathcal{O}(|z|^{-1})$. Thus we see from \eqref{eq:Ynorm} and \eqref{eq:tdef} that indeed \eqref{eq:tnorm} is satisfied (with behavior on the negative real axis coming from continuity up to the boundary).
\end{proof}

\section{The RHP for the global parametrix}\label{app:global}

In this appendix we will sketch a proof of Lemma \ref{le:global}. We will make use of the fact that the result is proven for $t=0$, i.e. the case when $\mathcal{T}=0$, in \cite[Section 4.2]{krasovsky} (which relies on a similar result in \cite[Section 5]{kmlvav}, which again makes use of results in e.g. \cite{deift}).

\begin{proof}[Sketch of a proof of Lemma \ref{le:global}]
The analyticity condition was already argued in Remark \ref{rem:globalholo}. The normalization at infinity is easy to see from the fact that the $a$-matrix (in right hand side of \eqref{eq:global}) is $2I+\mathcal{O}(|z|^{-1})$ and $\mathcal{D}_t(z)=\mathcal{D}_t(\infty)+\mathcal{O}(|z|^{-1})$ as $z\to\infty$. Thus the jump condition is the main one to check.

This would be a fairly short calculation to check directly, but we make use of it being known for $t=0$ and the representation \eqref{eq:globalalt}. We start by noting that by the Sokhotski-Plemelj formula and \eqref{eq:globalalt}, for $\lambda\in(-1,1)\setminus \lbrace x_j\rbrace_{j=1}^k$

\begin{equation*}
P_{\pm}^{(\infty)}(\lambda,t)=e^{\frac{\sigma_3}{2\pi}\int_{-1}^1 \frac{\mathcal{T}_t(x)}{\sqrt{1-x^2}}dx}P_\pm^{(\infty)}(\lambda,0)e^{-\sigma_3\frac{r_\pm(\lambda)}{2\pi}\left[\pm \pi i \frac{\mathcal{T}_t(\lambda)}{\sqrt{1-\lambda^2}}+P.V.\int_{-1}^1 \frac{\mathcal{T}_t(x)}{\sqrt{1-x^2}}\frac{dx}{\lambda-x}\right]},
\end{equation*}

\noindent where $P.V.$ denotes the Cauchy principal value integral. Thus from the jump condition of $P^{(\infty)}(z,0)$ (note that $\det P^{(\infty)}(z,t)=1$ so everything makes sense)

\begin{align*}
\left[P_-^{(\infty)}(\lambda,t)\right]^{-1}P_+^{(\infty)}(\lambda,t)&=e^{\sigma_3\frac{r_-(\lambda)}{2\pi}\left[-\pi i \frac{\mathcal{T}_t(\lambda)}{\sqrt{1-\lambda^2}}+P.V.\int_{-1}^1 \frac{\mathcal{T}_t(x)}{\sqrt{1-x^2}}\frac{dx}{\lambda-x}\right]}\begin{pmatrix}
0 & f_0(\lambda)\\
-f_0(\lambda)^{-1} & 0
\end{pmatrix}\\
\notag&\qquad \times e^{-\sigma_3\frac{r_+(\lambda)}{2\pi}\left[\pi i \frac{\mathcal{T}_t(\lambda)}{\sqrt{1-\lambda^2}}+P.V.\int_{-1}^1 \frac{\mathcal{T}_t(x)}{\sqrt{1-x^2}}\frac{dx}{\lambda-x}\right]}
\end{align*}

Noting that (from the definition of $r$; see \eqref{eq:rdef}) $r_+(\lambda)=i\sqrt{1-\lambda^2}$ and $r_-(\lambda)=-i\sqrt{1-\lambda^2}$ so with a simple calculation

\begin{equation*}
\left[P_-^{(\infty)}(\lambda,t)\right]^{-1}P_+^{(\infty)}(\lambda,t)=\begin{pmatrix}
0 & e^{\mathcal{T}_t(\lambda)}f_0(\lambda)\\
-e^{-\mathcal{T}_t(\lambda)}f_0(\lambda)^{-1} & 0
\end{pmatrix},
\end{equation*}

\noindent which is precisely the claim as $f_0e^{\mathcal{T}_t}=f_t$.

\end{proof}

\section{The RHP for the local parametrix near a singularity}\label{app:locals}

Here we give further details about the local parametrix near a singularity. First of all, we give a full description of the solution to the model RHP - the function $\Psi$.

\begin{definition}
Recall that we use Roman numerals for the octants of the plane: $\mathrm{I}=\lbrace r e^{i\theta}: r>0,\theta\in(0,\pi/4)\rbrace$ and so on. We also write $I_\nu$ and $K_\nu$ for the modified Bessel functions of the first and second kind, as well as $H_\nu^{(1)}$ and $H_\nu^{(2)}$ for the Hankel functions of the first and second kind. We then define  $($again roots are principal branch roots$)$

\begin{equation}
\Psi(\zeta)=\frac{1}{2}\sqrt{\pi \zeta}\begin{pmatrix}
H_{\frac{\beta_j+1}{2}}^{(2)}(\zeta) & -i H_{\frac{\beta_j+1}{2}}^{(1)}(\zeta)\\
H_{\frac{\beta_j-1}{2}}^{(2)}(\zeta) & -iH_{\frac{\beta_j-1}{2}}^{(1)}(\zeta)
\end{pmatrix} e^{-\left(\frac{\beta_j}{2}+\frac{1}{4}\right)\pi i \sigma_3}  \quad \zeta\in \mathrm{I},
\end{equation}

\begin{equation}
\Psi(\zeta)=\sqrt{\zeta}\begin{pmatrix}
\sqrt{\pi}I_{\frac{\beta_j+1}{2}}(-i\zeta) & -\frac{1}{\sqrt{\pi}} K_{\frac{\beta_j+1}{2}}(-i\zeta)\\
-i\sqrt{\pi}I_{\frac{\beta_j-1}{2}}(-i\zeta) & -\frac{i}{\sqrt{\pi}} K_{\frac{\beta_j-1}{2}}(-i\zeta)
\end{pmatrix} e^{-\frac{\beta_j}{4}\pi i \sigma_3}  \quad \zeta\in \mathrm{II},
\end{equation}

\begin{equation}
\Psi(\zeta)=\sqrt{\zeta}\begin{pmatrix}
\sqrt{\pi}I_{\frac{\beta_j+1}{2}}(-i\zeta) & -\frac{1}{\sqrt{\pi}} K_{\frac{\beta_j+1}{2}}(-i\zeta)\\
-i\sqrt{\pi}I_{\frac{\beta_j-1}{2}}(-i\zeta) & -\frac{i}{\sqrt{\pi}} K_{\frac{\beta_j-1}{2}}(-i\zeta)
\end{pmatrix} e^{\frac{\beta_j}{4}\pi i \sigma_3}  \quad \zeta\in \mathrm{III},
\end{equation}

\begin{equation}
\Psi(\zeta)=\frac{1}{2}\sqrt{-\pi \zeta}\begin{pmatrix}
iH_{\frac{\beta_j+1}{2}}^{(1)}(-\zeta) & - H_{\frac{\beta_j+1}{2}}^{(2)}(-\zeta)\\
-iH_{\frac{\beta_j-1}{1}}^{(1)}(-\zeta) & H_{\frac{\beta_j-1}{2}}^{(2)}(-\zeta)
\end{pmatrix} e^{\left(\frac{\beta_j}{2}+\frac{1}{4}\right)\pi i \sigma_3}  \quad \zeta\in \mathrm{IV},
\end{equation}

\begin{equation}
\Psi(\zeta)=\frac{1}{2}\sqrt{-\pi \zeta}\begin{pmatrix}
-H_{\frac{\beta_j+1}{2}}^{(2)}(-\zeta) & -i H_{\frac{\beta_j+1}{2}}^{(1)}(-\zeta)\\
H_{\frac{\beta_j-1}{2}}^{(2)}(-\zeta) & iH_{\frac{\beta_j-1}{2}}^{(1)}(-\zeta)
\end{pmatrix} e^{-\left(\frac{\beta_j}{2}+\frac{1}{4}\right)\pi i \sigma_3}  \quad \zeta\in \mathrm{V},
\end{equation}

\begin{equation}
\Psi(\zeta)=\sqrt{\zeta}\begin{pmatrix}
-i\sqrt{\pi}I_{\frac{\beta_j+1}{2}}(i\zeta) & -\frac{i}{\sqrt{\pi}} K_{\frac{\beta_j+1}{2}}(i\zeta)\\
\sqrt{\pi}I_{\frac{\beta_j-1}{2}}(i\zeta) & -\frac{1}{\sqrt{\pi}} K_{\frac{\beta_j-1}{2}}(i\zeta)
\end{pmatrix} e^{-\frac{\beta_j}{4}\pi i \sigma_3}  \quad \zeta\in \mathrm{VI},
\end{equation}

\begin{equation}
\Psi(\zeta)=\sqrt{\zeta}\begin{pmatrix}
-i\sqrt{\pi}I_{\frac{\beta_j+1}{2}}(i\zeta) & -\frac{i}{\sqrt{\pi}} K_{\frac{\beta_j+1}{2}}(i\zeta)\\
\sqrt{\pi}I_{\frac{\beta_j-1}{2}}(i\zeta) & -\frac{1}{\sqrt{\pi}} K_{\frac{\beta_j-1}{2}}(i\zeta)
\end{pmatrix} e^{\frac{\beta_j}{4}\pi i \sigma_3}  \quad \zeta\in \mathrm{VII},
\end{equation}

\begin{equation}
\Psi(\zeta)=\frac{1}{2}\sqrt{\pi \zeta}\begin{pmatrix}
-iH_{\frac{\beta_j+1}{2}}^{(1)}(\zeta) & - H_{\frac{\beta_j+1}{2}}^{(2)}(\zeta)\\
-iH_{\frac{\beta_j-1}{1}}^{(1)}(\zeta) & -H_{\frac{\beta_j-1}{2}}^{(2)}(\zeta)
\end{pmatrix} e^{\left(\frac{\beta_j}{2}+\frac{1}{4}\right)\pi i \sigma_3}  \quad \zeta\in \mathrm{VIII}.
\end{equation}
\end{definition}

In \cite[Theorem 4.2]{vanlessen} it is shown that this function indeed satisfies the problem we used in Definition \ref{def:psi}. An important fact about the function $\Psi$ is its behavior near the origin. The following was also part of \cite[Theorem 4.2]{vanlessen}: as $\zeta\to 0$

\begin{equation}\label{eq:psiasy}
\Psi(\zeta)=\begin{cases}
\begin{pmatrix}
\mathcal{O}(|\zeta|^{\beta_j/2}) &\mathcal{O}(|\zeta|^{-\beta_j/2})\\
\mathcal{O}(|\zeta|^{\beta_j/2}) & \mathcal{O}(|\zeta|^{-\beta_j/2})
\end{pmatrix}, & \zeta\in \mathrm{II,III,VI,VII}\\
\begin{pmatrix}
\mathcal{O}(|\zeta|^{-\beta_j/2}) &\mathcal{O}(|\zeta|^{-\beta_j/2})\\
\mathcal{O}(|\zeta|^{-\beta_j/2}) & \mathcal{O}(|\zeta|^{-\beta_j/2})
\end{pmatrix}, & \zeta\in \mathrm{I,IV,V,VIII}
\end{cases}.
\end{equation}

We also mention that the function $\Psi$ could be expressed in terms of the confluent hypergeometric function of the second kind as in \cite{dik1,dik2}. Let us now sketch the proof of Lemma \ref{le:localrhp}.

\begin{proof}[Sketch of a proof of Lemma \ref{le:localrhp}]
Consider first the analyticity condition. As we mentioned in Remark \ref{rem:eholo}, one can check that $E$ is analytic in $U_{x_j}'$, so the jumps of $P^{(x_j)}$ come from those of $\Psi(\zeta_s(z))$, $W_j(z)^{-\sigma_3}$ and $e^{-N\phi_s(z)\sigma_3}$.

As $\zeta_s$ preserves the real axis, and $\Sigma$ was chosen so that under $\zeta_s$, $\Sigma\cap U_{x_j}'$ is mapped to the real axis and lines intersecting origin at angles $\pm \pi/4$. Thus from Definition \ref{def:psi}, $\Psi(\zeta_s(z))$ has jumps on $\Sigma$ and $\lbrace z:\mathrm{Re}(\zeta_s(z))=0\rbrace$.

From \eqref{eq:wdef} -- the definition of $W_j$ -- we see that $W_j$ has jumps only across $\R$ and $\lbrace z:\mathrm{Re}(\zeta_s(z))=0\rbrace$. Also from \eqref{eq:phidef} and \eqref{eq:hdef} we see that $\phi$ only has a jump across $\R$.

Thus to see that $P^{(x_j)}(z,t,s)$ is analytic in $U_{x_j}'\setminus \Sigma$, we need to check that the jump of $W_j(z)^{-\sigma_3}$ cancels that of $\Psi(\zeta_s(z))$ along $\lbrace z:\mathrm{Re}(\zeta_s(z))=0\rbrace$. Let us look at for example the jump across $\lbrace z: \mathrm{Re}(\zeta_s(z))=0, \mathrm{Im}(\zeta_s(z))>0\rbrace=\zeta_s^{-1}(\Gamma_3)$. From \eqref{eq:wdef} we find that for $\lambda\in \zeta_s^{-1}(\Gamma_3)$ (where the orientation is as for $\Gamma_3$)

\begin{equation*}
W_{j,+}(\lambda)W_{j,-}(\lambda)^{-1}=\frac{(\lambda-x_j)^{\beta_j/2}}{(x_j-\lambda)^{\beta_j/2}}=e^{i\pi \frac{\beta_j}{2}}.
\end{equation*}

Combining this with \eqref{eq:psijump}

\begin{equation*}
\Psi_+(\zeta_s(\lambda))W_{j,+}(\lambda)^{-\sigma_3}=\Psi_{-}(\zeta_s(\lambda))e^{i\pi\frac{\beta_j}{2}\sigma_3} e^{-i\pi\frac{\beta_j}{2}\sigma_3}W_{j,-}(\lambda)=\Psi_{-}(\zeta_s(\lambda))W_{j,-}(\lambda),
\end{equation*}

\noindent so we see that $P^{(x_j)}(z)$ is continuous across $\zeta_s^{-1}(\Gamma_3)$. The argument is similar for the jump across $\zeta_s^{-1}(\Gamma_7)$. We conclude that $P^{(x_j)}$ is analytic in $U_{x_j}'\setminus \Sigma$.

Consider now the jump structure. The existence of continuous boundary values is inherited from the corresponding properties of $\Psi$, $W_j$ and $\phi_s$. As $W_j$ and $\phi_s$ have no jumps across $\Sigma_{j-1}^\pm$ or $\Sigma_j^\pm$, the jumps here come from the jumps of $\Psi$. Let us consider for example $\lambda\in \zeta_s^{-1}(\Gamma_2)$. Here using the jump condition of $\Psi$, an elementary matrix calculation shows that

\begin{align*}
P^{(x_j)}_+(\lambda)&=P^{(x_j)}_-(\lambda) W_j(\lambda)^{\sigma_3}e^{N\phi_s(\lambda)}\begin{pmatrix}
1 & 0\\
e^{-i\pi \beta_j} & 1
\end{pmatrix}W_j(\lambda)^{-\sigma_3}e^{-N\phi_s(\lambda)}\\
\notag &=P_-^{(x_j)}(\lambda)\begin{pmatrix}
1 & 0\\
f_t(\lambda)^{-1} e^{-Nh_s(\lambda)} & 1
\end{pmatrix}.
\end{align*}

Calculating the jump matrix across $\Sigma_{j-1}^\pm$ and $\Sigma_j^-$ is similar. For the jump across $\R$, we have for example for $\lambda\in U_{x_j}'\cap(x_j,\infty)$, from \eqref{eq:wdef}, \eqref{eq:phidef}, the analyticity of $h_s$ across $U_{x_j}'\cap \R$,  along with Definition \ref{def:psi}:

\begin{align*}
P^{(x_j)}_+(\lambda)&=P^{(x_j)}_-(\lambda)\begin{pmatrix}
0 & e^{Nh_s(\lambda)-i\pi\frac{\beta_j}{2}}W_{j,-}(\lambda)^2e^{2N\phi_{s,-}(\lambda)}\\
-e^{-Nh_s(\lambda)+i\pi\beta_j}W_{j,-}(\lambda)^{-2}e^{-2N\phi_{s,-}(\lambda)} & 0
\end{pmatrix}\\
\notag &=P^{(x_j)}_-(\lambda)\begin{pmatrix}
0 & f_t(\lambda) \\
-f_t(\lambda)^{-1} & 0
\end{pmatrix}.
\end{align*}

The calculation for the jump across $U_{x_j}'\cap(-\infty,x_j)$ is similar.

\vspace{0.3cm}

To see \eqref{eq:localrhpsing}, note first that as $z\to x_j$, $\zeta_s(x)=\mathcal{O}(|z-x_j|)$ (the implicit constant depending on $x_j$, $N$, and $s$, but this doesn't matter now) and  $W_j(z)=\mathcal{O}(|z-x_j|^{\beta_j/2})$. So we have from \eqref{eq:psiasy} that for $z\in\zeta_s^{-1}(\mathrm{I})$ and $z\to x_j$

\begin{equation*}
\Psi(\zeta_s(z))W_j(z)^{-\sigma_3}=\begin{pmatrix}
\mathcal{O}(|z-x_j|^{-\beta_j}) & \mathcal{O}(1)\\
\mathcal{O}(|z-x_j|^{-\beta_j}) & \mathcal{O}(1)
\end{pmatrix}.
\end{equation*}

As $E$ is analytic in $U_{x_j}'$, it is in particular bounded at $x_j$, so as multiplying from the left doesn't mix the columns, we have the same behavior for $E(z)\Psi(\zeta_s(z))W_j(z)^{-\sigma_3}$. Now also $\phi_s$ is bounded at $x_j$ and again multiplying by a diagonal matrix doesn't mix the columns so we have the claimed asymptotics for $P^{(x_j)}(z)$ as $z\to x_j$ from $\zeta_s^{-1}(\mathrm{I})$. The other regions are similar.

\vspace{0.3cm}

Let us now focus on the matching condition \eqref{eq:localmatch}. We note that as $d$ is positive on $[-1,1]$, we see that for $z\in \partial U_{x_j}$ (and for $\delta$ small enough), $|\zeta_s(z)|\asymp N$ where the implied constants are uniform in $x_j\in(-1+3\delta,1-3\delta)$, $s\in[0,1]$, and $z\in \partial U_{x_j}$. Thus to study $\Psi(\zeta_s(z))$, we can make use of the large argument expansion of Bessel functions. We won't go into great detail here, but simply refer the reader to \cite[Section 4.3]{vanlessen} and references therein.

For simplicity, we focus on the domain $\lbrace z: \mathrm{arg}\ \zeta_s(z)\in(0,\pi/2)\rbrace$. In the other domains, one has different asymptotics for $\Psi$, but the argument is similar. The relevant asymptotics here are

\begin{equation}\label{eq:besselasy}
\Psi(\zeta)=\frac{1}{\sqrt{2}}\begin{pmatrix}
1 & -i\\
-i & 1
\end{pmatrix}\left[I+\mathcal{O}(|\zeta|^{-1})\right]e^{\frac{\pi i}{4}\sigma_3}e^{-i\zeta\sigma_3}e^{-\pi i\frac{\beta_j}{4}\sigma_3},
\end{equation}

\noindent where the implied constant in $\mathcal{O}(|\zeta|^{-1})$ is uniform in the first quadrant. Here and below, the $\mathcal{O}$-notation will refer to a $2\times 2$ matrix whose entries satisfy the relevant bound. Noting from \eqref{eq:zetadef}, \eqref{eq:hdef}, and \eqref{eq:phidef}, that for $z\in U_{x_j}'\cap\lbrace\mathrm{Im}(z)>0\rbrace$

\begin{equation*}
\zeta_s(z)=-Ni(\phi_{s,+}(x_j)-\phi_s(z)).
\end{equation*}

It then follows from this and \eqref{eq:besselasy} that for $z\in \zeta_s^{-1}(\mathrm{I}\cup\mathrm{II})\cap \partial U_{x_j}$

\begin{align*}
\Psi(\zeta_s(z))W_{j}(z)^{-\sigma_3}e^{-N\phi_s(z)\sigma_3}&=\frac{1}{\sqrt{2}}\begin{pmatrix}
1 & -i\\
-i & 1
\end{pmatrix}\left[I+\mathcal{O}(N^{-1})\right]e^{\frac{\pi i}{4}\sigma_3}e^{-N(\phi_{s,+}(x_j)-\phi_s(z))\sigma_3}e^{-\pi i\frac{\beta_j}{4}\sigma_3}\\
\notag & \qquad \times W_j(z)^{-\sigma_3}e^{-N\phi_s(z)\sigma_3}\\
\notag &=\frac{1}{\sqrt{2}}\begin{pmatrix}
1 & -i\\
-i & 1
\end{pmatrix}\left[I+\mathcal{O}(N^{-1})\right] e^{i\frac{\pi}{4}(1-\beta_j)\sigma_3}e^{-N\phi_{s,+}(x_j)\sigma_3}W_j(z)^{-\sigma_3},
\end{align*}

\noindent where the $\mathcal{O}(N^{-1})$ term is uniform in everything relevant. Using \eqref{eq:edef1} and \eqref{eq:localdef}, we see that for $z\in \zeta_s^{-1}(\mathrm{I}\cup\mathrm{II})\cap \partial U_{x_j}$

\begin{equation*}
P^{(x_j)}(z)\left[P^{(\infty)}(z)\right]^{-1}=A(z)(I+\mathcal{O}(N^{-1}))A(z)^{-1}=I+A(z)\mathcal{O}(N^{-1})A(z)^{-1},
\end{equation*}

\noindent where the $\mathcal{O}(N^{-1})$ term is uniform in everything relevant and

\begin{equation*}
A(z)=P^{(\infty)}(z)W_j(z)^{\sigma_3} e^{N\phi_{s,+}(x_j)\sigma_3}e^{-i\frac{\pi}{4}(1-\beta_j)\sigma_3}=E(z)\left[\frac{1}{\sqrt{2}}\begin{pmatrix}
1 & i\\
i & 1
\end{pmatrix}\right]^{-1}
\end{equation*}

\noindent

The claim (in this sector of the boundary) will then follow if we show that $A$ is uniformly bounded in everything relevant. As $\phi_{s,+}(x_j)$ is purely imaginary (see \eqref{eq:hdef}), we see that the relevant question is the boundedness of $P^{(\infty)}(z)W_j(z)^{\sigma_3}$ and its inverse. Looking at \eqref{eq:global}, we see that this is equivalent to $\mathcal{D}_t(z)^{-1}W_j(z)$ being uniformly bounded and uniformly bounded away from zero. Let us write this quantity out. From \eqref{eq:ddef} and \eqref{eq:wdef} we have

\begin{equation*}
\left|\mathcal{D}_t(z)^{-1}W_j(z)\right|=\left|(z+r(z))^{\mathcal{A}} e^{-\frac{r(z)}{2\pi}\int_{-1}^1 \frac{\mathcal{T}_t(\lambda)}{\sqrt{1-\lambda^2}}\frac{d\lambda}{z-\lambda}} e^{\frac{1}{2}\mathcal{T}_t(z)}\right|.
\end{equation*}

$z+r(z)$ is obviously bounded for $z$ in a compact set, the integral term is uniformly bounded in everything relevant by Lemma \ref{le:cauchybounds}, and the last term is bounded as $\mathcal{T}_t$ is uniformly bounded in everything relevant. Similarly we see uniform boundedness away from zero. This concludes the proof for $z\in \zeta_s^{-1}(\mathrm{I}\cup\mathrm{II})\cap \partial U_{x_j}$. The proof in the remaining parts of the boundary are similar.
\end{proof}

We now move on to considering the proof of Lemma \ref{le:2ndorderlocals}.

\begin{proof}[Proof of Lemma \ref{le:2ndorderlocals}]
Here we simply need to take into account the next term in the asymptotic expansion of $\Psi$. The argument is otherwise as in the proof of Lemma \ref{le:localrhp}. For simplicity, we will focus on the case where $\zeta$ is in the first quadrant. Other quadrants are handled in a similar manner. We refer to the discussion around \cite[equation (5.9)]{vanlessen} for the following asymptotics:

\begin{equation}\label{eq:psiasy2ndorder}
\Psi(\zeta)=\frac{1}{\sqrt{2}}\begin{pmatrix}
1 & -i\\
-i & 1
\end{pmatrix}\left[I-i\frac{\beta_j}{4\zeta}\begin{pmatrix}
\frac{\beta_j}{2} & i\\
i& -\frac{\beta_j}{2}
\end{pmatrix}+\mathcal{O}\left(|\zeta|^{-2}\right)\right]e^{i\left(\frac{\pi}{4}-\frac{\beta_j\pi}{4}-\zeta\right)\sigma_3},
\end{equation}

\noindent where the error $\mathcal{O}(|\zeta|^{-2})$ is uniform for $\zeta$ in the first quadrant. Then arguing as in the previous proof, we see that

\begin{align*}
P^{(x_j)}(z)\left[P^{(\infty)}(z)\right]^{-1}&=I-i\frac{\beta_j}{4\zeta_s(z)}A(z)\begin{pmatrix}
\frac{\beta_j}{2} & i\\
i& -\frac{\beta_j}{2}
\end{pmatrix}A(z)^{-1}+\mathcal{O}\left(|\zeta_s(z)|^{-2}\right),
\end{align*}

\noindent where we used the uniform boundedness of $A$ and $A^{-1}$. Noting that

\begin{equation*}
\frac{1}{\sqrt{2}}\begin{pmatrix}
1 & -i\\
-i & 1
\end{pmatrix}\begin{pmatrix}
\frac{\beta_j}{2} & i\\
i & -\frac{\beta_j}{2}
\end{pmatrix}\frac{1}{\sqrt{2}}\begin{pmatrix}
1 & i\\
i & 1
\end{pmatrix}=\begin{pmatrix}
0 & \left(1+\frac{\beta_j}{2}\right)i\\
\left(1-\frac{\beta_j}{2}\right)i & 0
\end{pmatrix},
\end{equation*}

\noindent making use of $\zeta_s(z)\asymp N$ uniformly in everything relevant for $z\in \partial U_{x_j}$ and the fact that the asymptotic expansion of $\Psi$ is uniform, we see the claim. Again, the argument in the other regions is similar.

\end{proof}

\section{The RHP for the local parametrix near the edge of the spectrum}\label{app:locale}

In this section we will give some further details about the parametrices near the edge of the spectrum. First we will justify the definition of the function $\xi_s$ from \eqref{eq:xidef}.

\begin{proof}[Justification of the definition of $\xi_s$]
The argument is essentially as in \cite[Section 7]{dkmlvz}. Let us first recall some properties of $\phi_s$. From \eqref{eq:phidef} and \eqref{eq:hdef}, we note that $\phi_s$ has a jump across $U_1'\cap(-1,1)$ but is continuous across $U_1'\cap(1,\infty)$, so it is analytic in $U_1'\setminus[-1,1]$.
Moreover, in $U_1'\setminus[-1,1]$ we can write

\begin{equation}\label{eq:gtil}
\frac{3\pi}{2} d_s(z) (z+1)^{1/2} (z-1)^{1/2} = \widetilde{G}^{(1)}_s(z) (z-1)^{1/2},
\end{equation}

\noindent where $\widetilde{G}^{(1)}_s$ is analytic in $U_1'$. Expanding $\widetilde{G}^{(1)}_s$ as a series, integrating, and taking into account the branch structure of $\phi_s$, we can write

\begin{equation}\label{eq:Gs}
-\frac{3}{2}\phi_s(z)=G^{(1)}_s(z)(z-1)^{3/2},
\end{equation}

\noindent where the power is according to the principal branch and $G^{(1)}_s$ is analytic in $U_{1}'$. If we expand $G^{(1)}_s(z)=\sum_{k=0}^\infty G^{(1)}_{s,k}(z-1)^k$ and $\widetilde{G}^{(1)}_s(w)=\sum_{k=0}^\infty \widetilde{G}^{(1)}_{s,k}(w-1)^k$, then

\begin{equation*}
G^{(1)}_{s,k}=\frac{2}{3+2k} \widetilde{G}^{(1)}_{s,k}.
\end{equation*}

Now as $\widetilde{G}^{(1)}_{s,0}=\frac{3\pi}{\sqrt{2}}d_s(1)$ is uniformly bounded away from zero, we see from the above display that the same holds for $G^{(1)}_{s,0}$. By Cauchy's integral formula (for derivatives),

\begin{equation*}
\left|\widetilde{G}^{(1)}_{s,k}\right|\leq (3\delta/2)^{-k}\sup_{|z-1|=\delta}\left|\frac{3}{2}\sqrt{z+1}\left[sd(z)+(1-s)\frac{2}{\pi}\right]\right|\leq C_\delta (3\delta/2)^{-k},
\end{equation*}

\noindent for some constant $C_\delta$ independent of $s$, so we again get a similar bound for $G^{(1)}_{s,k}$. From this type of estimate, one can easily argue that by possibly decreasing $\delta$ by some $s$-independent factor, $G^{(1)}_s$ is zero free in $U_1'$. Thus with a suitable convention for the branch of the power, the function

\begin{equation*}
\xi_s(z)=N^{2/3}(z-1) G^{(1)}_s(z)^{2/3}
\end{equation*}

\noindent is analytic in $U_1'$.

For injectivity, note that the derivative of the function $z\mapsto (z-1)G^{(1)}_s(z)^{2/3}$ at $z=1$ is uniformly (in $s$) bounded away from zero and its second derivative is uniformly bounded in $s$ and in a small enough ($s$ independent) neighborhood of $1$. Thus by decreasing $\delta$ if needed (in an $s$ independent manner), we have univalence of $\xi_s$.

\end{proof}

We now sketch the proof of Lemma \ref{le:localerhp}.

\begin{proof}[Sketch of a proof of Lemma \ref{le:localerhp}]
Let us first of all consider the analyticity of $F$. $P^{(\infty)}$ is analytic in $U_1'\setminus [-1,1]$, $f^{1/2}$ is analytic in $U_1'$, and as $\zeta_s(1)=0$, $\zeta_s^{1/4}$ has a branch cut in $U_1$. We note from \eqref{eq:Gs} that as one can check (from \eqref{eq:hdef}) that $-\phi_s(\lambda)>0$ for $\lambda>1$, $G_s(\lambda)>0$ for $\lambda>1$. Thus $G_s$ is real on $\R\cap U_1'$. As we argued above that it's zero free, it must be positive on $\R\cap U_1'$, so we see that $\xi_s(\lambda)<0$ for $\lambda<1$. As we are dealing with the principal branch, the cut of $\xi_s^{1/4}$ is along $U_1'\cap(-1,1)$. It's thus enough to check that $F$ is continuous across $(-1,1)\cap U_1'$ and does not have an isolated singularity at $z=1$.

For the continuity across $(-1,1)$, let $\lambda\in(-1,1)\cap U_1'$. We have from \eqref{eq:globaljump} and the jump for $\xi_s^{1/4}$: for $\lambda\in(-1,1)\cap U_1'$

\begin{equation*}
[\xi_s]_+^{1/4}(\lambda)=i[\xi_s]_-^{1/4}(\lambda),
\end{equation*}

\noindent so that

\begin{align*}
F_-(\lambda)^{-1}F_+(\lambda)&=\left([\xi_s]^{1/4}_-(\lambda)\right)^{-\sigma_3} \frac{1}{2}\begin{pmatrix}
1 & 1\\
-1 & 1
\end{pmatrix} e^{-i\frac{\pi}{4}\sigma_3}f_t(\lambda)^{-\sigma_3/2}\left[P^{(\infty)}_-(\lambda)\right]^{-1} P^{(\infty)}_+(\lambda)\\
\notag & \qquad \times f_t(\lambda)^{\sigma_3/2}e^{i\frac{\pi}{4}\sigma_3}\begin{pmatrix}
1 & -1\\
1 & 1
\end{pmatrix}\left([\xi_s]^{1/4}_+(\lambda)\right)^{\sigma_3} \\
\notag &=\left([\xi_s]^{1/4}_-(\lambda)\right)^{-\sigma_3} \frac{1}{2}\begin{pmatrix}
1 & 1\\
-1 & 1
\end{pmatrix} \begin{pmatrix}
0 & -i\\
-i & 0
\end{pmatrix}\begin{pmatrix}
1 & -1\\
1 & 1
\end{pmatrix}\left([\xi_s]^{1/4}_+(\lambda)\right)^{\sigma_3} \\
&=I.
\end{align*}

\noindent Thus $F$ is continuous across $(-1,1)\cap U_1'$.

For the absence of an isolated singularity, we note that the entries of $\xi_s(z)^{\sigma_3/4}$ behave at worst like $|z-1|^{-1/4}$ as $z\to 1$. From \eqref{eq:global} and Lemma \ref{le:cauchybounds}, we see that the entries of $P^{(\infty)}(z)$ behave at worst like $|z-1|^{-1/4}$ as well. As $f^{1/2}(z)$ is bounded at $z=1$, the entries of $F(z)$ behave at worst like $|z-1|^{-1/2}$. This is not strong enough to be a pole, so there can be no isolated singularity at $z=1$ and $F$ is analytic.

\vspace{0.3cm}

Towards checking the analyticity of $P^{(1)}$ on $U_1'\setminus \Sigma$, we refer to \cite[Section 7]{dkmlvz} on the following matter (in their notation $Q=\Psi^\sigma$): $Q(\xi_s(z))$ is analytic on $U_1'\setminus \Sigma$ and it satisfies the following jump conditions:

\begin{equation}\label{eq:qjump1}
Q_+(\xi_s(\lambda))=Q_-(\xi_s(\lambda))\begin{pmatrix}
1 & 0\\
1 & 1
\end{pmatrix}, \quad \lambda\in \Sigma_{k+1}^\pm\cap U_1',
\end{equation}

\begin{equation}\label{eq:qjump2}
Q_+(\xi_s(\lambda))=Q_-(\xi_s(\lambda))\begin{pmatrix}
0 & 1\\
-1 & 0
\end{pmatrix}, \quad \lambda\in (-1,1)^\pm\cap U_1',
\end{equation}

\noindent and

\begin{equation}\label{eq:qjump3}
Q_+(\xi_s(\lambda))=Q_-(\xi_s(\lambda))\begin{pmatrix}
1 & 1\\
0 & 1
\end{pmatrix}, \quad \lambda\in (1,\infty)^\pm\cap U_1'.
\end{equation}

As $f_t^{\pm 1/2}$ is analytic in $U_1'$ as if $F$, and $\phi_s$ has a jump along $(-1,1)\cap U_1'$, we see that $P^{(1)}$ indeed is analytic in $U_1'$.

The jump conditions come from those of $Q$. Let us check for example the one across $(-1,1)\cap U_1'$ -- \eqref{eq:localej1}. For $\lambda\in(-1,1)\cap U_1'$, we have

\begin{align*}
\left[P_-^{(1)}(\lambda)\right]^{-1}P^{(1)}_+(\lambda)&=f_t(\lambda)^{\sigma_3/2}e^{N\phi_{s,-}(\lambda)}Q_-(\xi_s(\lambda)) Q_+(\xi_s(\lambda)e^{-N\phi_{s,+}(\lambda)}f_t(\lambda)^{-\sigma_3/2}\\
\notag &=f_t(\lambda)^{\sigma_3/2} e^{-\frac{1}{2}Nh_s(\lambda)\sigma_3}\begin{pmatrix}
0 & 1\\
-1 & 0
\end{pmatrix}e^{-\frac{1}{2}N h_s(\lambda)\sigma_3}f_t(\lambda)^{-\sigma_3/2}\\
\notag &=\begin{pmatrix}
0 & f_t(\lambda)\\
-f_t(\lambda)^{-1} & 0
\end{pmatrix}.
\end{align*}

\noindent The other jump conditions are similar.

Let us then check the matching condition. Let $z\in \partial U_1$. For small enough $\delta$ (independent of $s$), it is clear from \eqref{eq:hdef} and \eqref{eq:phidef} that $|\phi_s(z)|$ is bounded away from zero uniformly in $s$ and uniformly in $z\in \partial U_1$. Thus $|\xi_s(z)|\asymp N^{2/3}$ where the implied constants are uniform in $z$ and $s$. We can thus make use of the large $|\xi|$ asymptotics of $\mathrm{Ai}(\xi)$ and $\mathrm{Ai}'(\xi)$ to obtain asymptotics for $Q(\xi_s(z))$. For this, we will again refer to \cite{dkmlvz} -- in particular \cite[(7.30)]{dkmlvz}: for $z\in \partial U_1$

\begin{equation*}
Q(\xi_s(z))e^{\frac{2}{3}\xi_s(z)^{3/2}\sigma_3}=\frac{e^{\pi i /12}}{2\sqrt{\pi}}\xi_s(z)^{-\sigma_3/4}\left[\begin{pmatrix}
1 & 1\\
-1 & 1
\end{pmatrix}e^{-i\frac{\pi}{4}\sigma_3}+\mathcal{O}(N^{-1})\right],
\end{equation*}

\noindent where the error is uniform in $z$ and $s$. Recalling that the construction of $\xi_s$ was precisely so that $\frac{2}{3}\xi_s(z)^{3/2}=-N\phi_s(z)$, we see that

\begin{equation*}
Q(\xi_s(z))e^{-N\phi_s(z)\sigma_3}f_t(z)^{-\sigma_3/2}=\frac{e^{\pi i /12}}{2\sqrt{\pi}}\xi_s(z)^{-\sigma_3/4}\left[\begin{pmatrix}
1 & 1\\
-1 & 1
\end{pmatrix}e^{-i\frac{\pi}{4}\sigma_3}+\mathcal{O}(N^{-1})\right] f_t(z)^{-\sigma_3/2},
\end{equation*}

\noindent with the $\mathcal{O}(N^{-1})$-term being uniform in everything relevant. Thus

\begin{align*}
P^{(1)}(z)\left[P^{(\infty)}(z)\right]^{-1}&=I+P^{(\infty)}(z) f_t(z)^{\sigma_3/2}\mathcal{O}(N^{-1})f_t(z)^{-\sigma_3/2}\left[P^{(\infty)}(z)\right]^{-1}.
\end{align*}

As $f_t(z)^{\pm 1}$ as well as the entries of $[P^{(\infty)}]^{\pm 1}$ are uniformly (in everything relevant) bounded on $\partial U_1$, the claim follows.
\end{proof}

We will also give a proof of Lemma \ref{le:locale2ndorderasy}.

\begin{proof}[Proof of Lemma \ref{le:locale2ndorderasy}]
This is again proven as the matching condition, but using finer asymptotics of the Airy function. In particular, one has (see \cite[(7.30)]{dkmlvz})

\begin{align*}
Q(\xi_s(z))&e^{-N\phi_s(z)\sigma_3}\\
&=\frac{e^{\pi i /12}}{2\sqrt{\pi}}\xi_s(z)^{-\sigma_3/4}\left[\begin{pmatrix}
1 & 1\\
-1 & 1
\end{pmatrix}+\begin{pmatrix}
-\frac{5}{48} & \frac{5}{48}\\
-\frac{7}{48} & -\frac{7}{48}
\end{pmatrix}\xi_s(z)^{-3/2}+\mathcal{O}(|\xi_s(z)|^{-3})\right]e^{-i\frac{\pi}{4}\sigma_3},
\end{align*}

\noindent where the constant implied by the $\mathcal{O}$ notation is uniform in everything relevant. Thus arguing as in the previous proof, we see that for $z\in \partial U_1$

\begin{align*}
&P^{(1)}(z)\left[P^{(\infty)}(z)\right]^{-1}\\
&=I+P^{(\infty)}(z)f(z)^{\sigma_3/2}e^{i\pi\sigma_3/4}\frac{1}{8}\begin{pmatrix}
\frac{1}{6} & 1\\
-1 & -\frac{1}{6}
\end{pmatrix}e^{-i\pi\sigma_3/4}f(z)^{-\sigma_3/2}\left[P^{(\infty)}(z)\right]^{-1}\xi_s(z)^{-3/2}+\mathcal{O}(N^{-2})
\end{align*}

\noindent uniformly in everything relevant.

\end{proof}

\section{\texorpdfstring{Proofs concerning the final transformation and solving the $R$-RHP}{Proofs concerning the final transformation and solving the small norm RHP}}\label{app:rrhp}

In this section we sketch proofs concerning the final transformation and the solution of the $R$-RHP. We start with checking that $R$ indeed solves the RHP of Lemma \ref{le:rrhp}.

\begin{proof}[Proof of Lemma \ref{le:rrhp}]
Uniqueness follows from $S$ being the unique solution to its problem. The last condition is immediate to check as for large $|z|$, $R(z)=S(z)[P^{(\infty)}(z)]^{-1}$ and both of these terms are asymptotically $I+\mathcal{O}(|z|^{-1})$. The jump conditions simply make use of the definition of $R$ and the jump conditions of $S$ -- these are direct to check and we skip this.

For the analyticity condition we begin with the domain $U_{\pm 1}$. Here the construction of $P^{(\pm 1)}$ was such that it would have the same jumps as $S$ so $R$ has no branch cuts inside of $U_{\pm 1}$. We are left with the possibility that $R$ would have an isolated singularity at $z=\pm 1$. Recall that $S(z)$ is bounded as $z\to \pm 1$, while Lemma \ref{le:cauchybounds} implies that the entries of $[P^{(\infty)}(z)]^{-1}$ can blow up at most like $|z\mp 1|^{-1/4}$ as $z\to \pm 1$. Thus the possible isolated singularity of $R$ is not strong enough to be a pole (or essential), so it is removable, and $R$ is analytic in $U_{\pm 1}$.

Consider now a neighborhood $U_{x_j}$. Again, by the construction of the parametrix, there are no jumps here, and the only possible singularity is an isolated singularity at $x_j$. Recall now that as $z\to x_j$ from outside of the lenses, $S(z)=\mathcal{O}(1)$, and as $z\to x_j$ from inside of the lenses,

\begin{equation*}
S(z)=\begin{pmatrix}
\mathcal{O}(|z-x_j|^{-\beta_j}) & \mathcal{O}(1)\\
\mathcal{O}(|z-x_j|^{-\beta_j}) & \mathcal{O}(1)
\end{pmatrix}.
\end{equation*}

$P^{(x_j)}(z)$ has similar behavior near $x_j$. To estimate it's inverse, we note  that $\det P^{(x_j)}(z)=1$ for all $z\in U_{x_j}$ - which follows directly from the definitions once one knows that $\det \Psi=1$ (which we argued following Definition \ref{def:psi}, or one could check directly using the explicit representation of $\Psi$ from Appendix \ref{app:locals}).

We thus see that as $z\to x_j$ from outside of the lenses, $[P^{(x_j)}(z)]^{-1}$ remains bounded, and as $z\to x_j$ from inside the lenses, we have

\begin{equation*}
\left[P^{(x_j)}(z)\right]^{-1}=\begin{pmatrix}
\mathcal{O}(1) & \mathcal{O}(1)\\
\mathcal{O}(|z-x_j|^{-\beta_j}) & \mathcal{O}(|z-x_j|^{-\beta_j})
\end{pmatrix}
\end{equation*}

\noindent so we conclude that from the inside of the lens, the entries of the matrix $S(z)[P^{(x_j)}(z)]^{-1}$ have singularities of order $\mathcal{O}(|z-x_j|^{-\beta_j})$ at worst. Now we see that as $S(z)[P^{(x_j)}(z)]^{-1}$ remains bounded as $z\to x_j$ from outside of the lenses, it can't have a pole at $x_j$. But as the degree of the singularity is bounded (we can find an integer $k$ such that $(z-x_j)^kS(z)[P^{(x_j)}(z)]^{-1}$ tends to zero as $z\to x_j$), the singularity can't be essential either. Thus the only possibility is that the singularity is removable, and $R(z)$ is analytic in $U_{x_j}$. Thus we see that $R$ indeed solves the Riemann-Hilbert problem.
\end{proof}

We next prove the relevant estimate for the jump matrix.

\begin{proof}[Proof of Lemma \ref{le:rjumpmat}]
Let us first consider the jump matrix on $\R\setminus [-1-\delta,1+\delta]$.  Here we have

\begin{equation*}
\Delta(\lambda)=P^{(\infty)}(\lambda)\begin{pmatrix}
0 & f_t(\lambda)e^{N(g_{s,+}(\lambda)+g_{s,-}(\lambda)-V_s(\lambda)-\ell_s)}\\
0 & 0
\end{pmatrix}\left[P^{(\infty)}(\lambda)\right]^{-1}.
\end{equation*}

First of all, we note that the entries of $P^{(\infty)}(\lambda)$ and $[P^{(\infty)}(\lambda)]^{-1}$ are bounded (uniformly in everything relevant) in this area, and $f_t(\lambda)$ grows like $|\lambda|^{\sum_{j=1}^k \beta_j}$ as $|\lambda|\to\infty$. From \eqref{eq:gbv2}, we see that there exist constants $C,M>0$ depending only on $V$ such that for $|\lambda|>1+M$, $e^{N(g_{s,+}(\lambda)+g_{s,-}(\lambda)-V_s(\lambda)-\ell_s)}\leq |\lambda|^{-N}$ and for $|\lambda|-1\in(0,M)$, $e^{N(g_{s,+}(\lambda)+g_{s,-}(\lambda)-V_s(\lambda)-\ell_s)}\leq e^{-NC(|\lambda|-1)^{3/2}}$. From these estimates, it's easy to see that any $L^p$ norm on $\R\setminus[-1-\delta,1+\delta]$ is exponentially small in $N$.

Consider next the part of the contour lying on the boundaries of the lenses. More precisely, we have for $\lambda\in \cup_{j=1}^{k+1}\Sigma_j^\pm\setminus\overline{U_{-1}\cup\cup_{j=1}^k U_{x_j}\cup U_1}$,

\begin{equation*}
\Delta(\lambda)=P^{(\infty)}(\lambda)\begin{pmatrix}
0 & 0\\
f_t(\lambda)^{-1}e^{\mp Nh_s(\lambda)} & 0
\end{pmatrix}\left[P^{(\infty)}(\lambda)\right]^{-1}.
\end{equation*}

We now refer to Lemma \ref{le:hsign}, which states that for example for $\lambda\in\Sigma_j^+\setminus\overline{U_{-1}\cup\cup_{l=1}^k U_{x_l}\cup U_1}$, there exists an $\epsilon>0$ independent of $s$ and $\lambda$ such that  $\mathrm{Re}(h_s(\lambda))>\epsilon$ (we assume that the distance between this part of the contour and the real axis is bounded away from zero uniformly in everything relevant). Moreover, $f_t(\lambda)^{-1}$ is uniformly bounded here so we again get exponential smallness for any $L^p$ norm uniformly in everything relevant for this part of the contour (as the contour has finite length). The $\Sigma_j^{-}$-case is identical.

For $\partial U_{x_j}$ and $\partial U_{\pm 1}$ the bounds come from the matching conditions in Lemma \ref{le:localrhp} and Lemma \ref{le:localerhp}. Combining the estimates from the different parts of the contour is elementary and we find the claim.

\end{proof}

The next proof we consider is the representation of $R$ in terms of a certain Neumann-series. The proof follows \cite[Theorem 7.8]{dkmlvz}, and while it is a standard fact, we record it here for completeness.

\begin{proof}[Proof of Proposition \ref{prop:neumann}]
By the Sokhotski-Plemelj theorem, we see that the function $\widehat{R}=I+ C(R_+ - R_-)$ satisfies $\widehat{R}_+-\widehat{R}_-=R_+-R_-$ across $\Gamma_\delta\setminus \lbrace \mathrm{intersection \ points}\rbrace$ (note that from our proof of Lemma \ref{le:rjumpmat}, we see that $R_+-R_-=R_-\Delta$ has nice enough decay at infinity for $\widehat{R}$ to be well defined). Thus the function $\widehat{R}-R$ has no jump across $\Gamma_\delta\setminus \lbrace\mathrm{intersection \ points}\rbrace$. By construction, both functions are bounded at the intersection points of the different parts of the contour, and behave like $I+\mathcal{O}(|z|^{-1})$ as $z\to\infty$, so by Liouville's theorem

\begin{equation*}
R=I+C(R_+-R_-)=I+C(R_-\Delta).
\end{equation*}

In particular, taking the limit from the $-$ side, we obtain
$$R_- - I = C_-(R_- \Delta) = C_\Delta(R_-) \quad \Leftrightarrow  \quad (I-C_\Delta)(R_- - I) = C_\Delta(I).$$

\noindent It is well known that $C_-$ is a bounded operator from $L^2(\Gamma_\delta)$ to $L^2(\Gamma_\delta)$ -- see e.g. the discussion and references in \cite[Appendix A]{dkmlvz}. Given the estimate in Lemma \ref{le:rjumpmat} the operator norm of $C_\Delta$ is of order $\mathcal{O}(1/N)$, $I - C_\Delta$ is invertible (and the inverse can be expanded as a Neumann series) for $N$ sufficiently large and the result follows.

\end{proof}

Finally we prove the main result concerning $R$. Our proof is a minor modification of that in \cite{krasovsky}.

\begin{proof}[Proof of Theorem \ref{th:rasy}]
Note that since $(I - C_\Delta)(R_- - I) = C_\Delta(I)$ and since the $L^2$-boundedness of $C_-$ implies that $||C_\Delta(I)||_{L^2(\Gamma_\delta)}=\mathcal{O}(N^{-1})$ (uniformly in everything relevant), we have

\begin{align*}
||R_- - I||_{L^2(\Gamma_\delta)} \le ||(I - C_\Delta)^{-1}||_{L^2(\Gamma_\delta)\to L^2(\Gamma_\delta)} ||C_{\Delta}(I)||_{L^2(\Gamma_\delta)} \le \frac{c_1}{N}
\end{align*}

\noindent for some $c_1 > 0$ (independent of the relevant quantities).

Now fix some small $\epsilon > 0$, and suppose $z$ is at least $\epsilon$ away from the jump contour $\Gamma_\delta$. Recall that in the proof of \eqref{eq:R_Neuman}, we saw that $(I-C_\Delta)^{-1}C_\Delta(I)=R_--I$, so we have (for $c_2,c_3,c_4$ depending on $\epsilon$ but not on $t,s,...$)

\begin{align*}
|R- I| &\le |C(\Delta)| + |C((R_- - I)\Delta)| \\
& \le \frac{c_2}{N} + c_3 ||R_- - I||_{L^2(\Gamma_\delta)} ||\Delta||_{L^2(\Gamma_\delta)} \le \frac{c_4}{N},
\end{align*}

\noindent where we used Cauchy-Schwarz in the second step and the facts that $R_-$ is bounded on $\Gamma_\delta$ and behaves like $I+\mathcal{O}(|\lambda|^{-1})$, as $\lambda\to\infty$.

\vspace{0.3cm}

For $z \in \mathbb{C} \backslash \Gamma_\delta$ that is within a distance of $\epsilon$ from $\Gamma_\delta$ but not close to any intersection points, we use the usual trick of contour deformation. First note that we can analytically continue the jump matrix $J_R$ to, without loss of generality, a $(2\epsilon)$-neighbourhood of $\Gamma_\delta$, with the estimates in Lemma \ref{le:rjumpmat} remaining true (up to a change of constants).

We may assume that $z$ lies on the $+$ side of $\Gamma_\delta$. Let $\widetilde{\Gamma}_{\delta}$ be the contour in Figure \ref{fig:R_deform}, obtained from $\Gamma_\delta$ with the dotted part replaced by a half circle of radius $\epsilon$,  and $\widetilde{R}$ be defined as shown, where $J$ is the analytic continuation of $J_R$. Then $\widetilde{R}(z)$ satisfies the same Riemann-Hilbert problem as $R(z)$ except on the new contour $\widetilde{\Gamma}_\delta$. Repeating our argument for the case where $z$ is at distance at least $\epsilon$ from the contour, we see that

\begin{equation*}
|R(z) - I| = |\widetilde{R}(z) - I| \le \frac{c_5}{N},
\end{equation*}

\noindent for a $c_5$ which is uniform in the relevant quantities. Now note that all estimates established so far are also uniform in $\delta \in K \subset (0, \delta_0]$ for some compact set $K$ and $\delta_0 > 0$, see \cite[Section 7.2]{dkmlvz}. If $z$ is close to any intersection points we may then deform our contour by varying $\delta$.

For the derivative, let us consider the case where the distance between $z$ and the jump contour is greater than $\epsilon$. Then by Cauchy's integral formula we have

\begin{equation*}
R'(z)=\frac{1}{2\pi i}\int_{|w-z|=\epsilon}\frac{R(w)}{(w-z)^2}dw=\frac{1}{2\pi i}\int_{|w-z|=\epsilon}\frac{R(w)-I}{(w-z)^2}dw = \mathcal{O}(N^{-1})
\end{equation*}

\noindent where the last equality follows from the uniform estimates for $R(w)-I$. For $z$ close to the contour we argue by contour deformation again.

\begin{figure}
\begin{center}
\begin{tikzpicture}

\draw[thick, ->] (0,0) -- (0.5, 0) node[above]{$+$} node[below] {$-$} -- (4,0) (0,0) -- (2,0);
\draw[thick, ->] (6,0) -- (10,0) node[right] {$\widetilde{\Gamma}_\delta$} (6,0) -- (8,0);
\draw[dotted, thick] (4,0) -- (6,0);
\draw[fill = black] (5, 0.1) circle [radius = 0.05] node[above]{$z$};
\draw[thick] (4,0) arc [radius = 1, start angle = 180, end angle = 360];
\draw[thick, ->] (4,0) arc [radius = 1, start angle = 180, end angle = 270];
\draw[thick, <->] (5,-0.2) -- (5, -0.5) node[left] {$\epsilon$} -- (5, -0.8);

\node at (2, 1) {$\widetilde{R} = R$};
\node at (2, -1) {$\widetilde{R} = R$};
\draw[thick, ->] (7, -1) node[below] {$\widetilde{R} = RJ$} -- (5.5, -0.3);

\end{tikzpicture}
\end{center}
\caption{\label{fig:R_deform} Deforming the R-RHP.}
\end{figure}
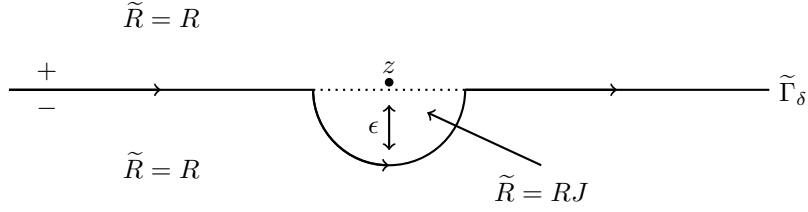

We now want to extract the second order asymptotics when $\Tree = 0$. Since

\begin{equation*}
R=I+C(\Delta)+C((R_--I)\Delta),
\end{equation*}

\noindent repeating our argument with minor modifications we see that

\begin{equation*}
R-I-C(\Delta)=\mathcal{O}(N^{-2})
\quad \text{ and } \quad R' - C(\Delta)' = \mathcal{O}(N^{-2})
\end{equation*}

\noindent uniformly off of $\Gamma_\delta$ and uniformly in everything relevant. Now by definition, we have

\begin{equation*}
[C(\Delta)](z)=\int_{\Gamma_\delta}\frac{\Delta(w)}{w-z}\frac{dw}{2\pi i}.
\end{equation*}

With similar arguments as in the proof of Lemma \ref{le:rjumpmat}, one can easily see (e.g. using Cauchy-Schwarz and a $L^2$-norm bound on the jump matrix on the unbounded part of the contour and a $L^\infty$-norm bound on the part of the contour on the boundary of the lenses) that the contribution from the part of the contour on $\R$ and on the boundary of the lenses has uniformly (in everything relevant) exponentially small contribution to $C(\Delta)$. Thus we have for $z$ not on the jump contour

\begin{align*}
[C(\Delta)](z)&=\sum_{j=0}^{k+1}\oint_{\partial U_{x_j}}\frac{\Delta(w)}{w-z}\frac{dw}{2\pi i}+\mathcal{O}(N^{-2})
=: \sum_{j=0}^{k+1}R_1^{(x_j)}(z) + \mathcal{O}(N^{-2}),
\end{align*}

\noindent where the orientation of the contours is in the clockwise direction and the $\mathcal{O}(N^{-2})$ is uniform in everything relevant. From Lemma \ref{le:2ndorderlocals}, Lemma \ref{le:locale2ndorderasy}, and Remark \ref{rem:localematch}, we can then write (again for $z$ off of the jump contour)

\begin{align*}
&R_1^{(x_j)}(z)
=\frac{1}{2\pi i}\oint_{\partial U_{x_j}}\frac{dw}{w-z}\frac{\beta_j}{4\zeta_s^{(x_j)}(w)}E^{(x_j)}(w)\begin{pmatrix}
0 & 1+\frac{\beta_j}{2}\\
1-\frac{\beta_j}{2} & 0\end{pmatrix}\left[E^{(x_j)}(w)\right]^{-1}, \quad 1\leq j\leq k\\
&R_1^{(\pm 1)}(z)= \frac{1}{2\pi i}\oint_{\partial U_{\pm 1}}\frac{dw}{w-z}F^{(\pm1)}(w)\begin{pmatrix}
0 & \pm\frac{5}{48}\left[\xi_s^{(\pm 1)}(w)\right]^{-2}\\
-\frac{7}{48}\left[\xi_s^{(\pm 1)}(w)\right]^{-1} & 0
\end{pmatrix}\left[F^{(\pm 1)}(w)\right]^{-1}\\
\end{align*}

\noindent where the superscripts have been added to underline that the functions depend on the singularity we are considering.

Consider now $z\notin U_{x_j}$ with $j\in\lbrace 1,...,k\rbrace$. Then as $E,E^{-1}$ are analytic in $U_{x_j}$ and $1/\zeta^{(x_j)}_s(w)$ has a simple pole at $w=x_j$ (and no other singularities in $U_{x_j}$), we see that

\begin{align*}
R_1^{(x_j)}(z) &=\frac{1}{z-x_j}\frac{\beta_j}{4\pi N\left(\frac{2}{\pi}(1-s)+sd(x_j)\right)\sqrt{1-x_j^2}}E^{(x_j)}(x_j)\begin{pmatrix}
0 & 1+\frac{\beta_j}{2}\\
1-\frac{\beta_j}{2} & 0\end{pmatrix}\left[E^{(x_j)}(x_j)\right]^{-1},
\end{align*}

\noindent where, by writing $b_{x_j} = a_+(x_j)^2 + a_+(x_j)^{-2}$ and $\bar{b}_{x_j} = a_+(x_j)^2 - a_+(x_j)^{-2}$, { one finds (after an elementary calculation)}
\begin{align*}
& E^{(x_j)}(x_j)\begin{pmatrix}
0 & 1+\frac{\beta_j}{2}\\
1-\frac{\beta_j}{2} & 0\end{pmatrix}\left[E^{(x_j)}(x_j)\right]^{-1}\\
& \qquad = \frac{1}{8}
\begin{pmatrix}
-i[2(c_{x_j}^2 + c_{x_j}^{-2}) b_{x_j} \bar{b}_{x_j} + \beta_j (b_{x_j}^2 + \bar{b}_{x_j}^2)] &
2 \mathcal{D}(\infty)^2 [(c_{x_j}^2 b_{x_j}^2 + c_{x_j}^{-2} \bar{b}_{x_j}^2 ) + \beta_j b_{x_j} \bar{b}_{x_j}{]}\\
2 \mathcal{D}(\infty)^{-2} [(c_{x_j}^{-2} b_{x_j}^2 + c_{x_j}^{2} \bar{b}_{x_j}^2 ) + \beta_j b_{x_j} \bar{b}_{x_j}{]}&
i[2(c_{x_j}^2 + c_{x_j}^{-2}) b_{x_j} \bar{b}_{x_j} + \beta_j (b_{x_j}^2 + \bar{b}_{x_j}^2)]
\end{pmatrix}.
\end{align*}

\noindent {Here we made use of the fact that $E^{(x_j)}$ is analytic at $x_j$ so we can evaluate $E^{(x_j)}(x_j)$ using the formula \eqref{eq:edef1}.}

For $R_1^{(\pm 1)}(z)$ with $z \notin U_{\pm 1}$ the residue calculations are more involved (but still straightforward) because of the presence of a second order pole. We just summarize here that

\begin{align*}
& R_{1}^{(-1)}(z)
= - \frac{2^{1/2}}{2N}\frac{1}{(-1-z)^2} \frac{5}{48G_s^{(-1)}(-1)}
\begin{pmatrix}
-i & \mathcal{D}(\infty)^2 \\ \mathcal{D}(\infty)^{-2} & i
\end{pmatrix}\\
&  +\frac{2^{1/2}}{8N}\frac{1}{z+1} \begin{pmatrix}
i\left[\frac{9 - 96 \mathcal{A}^2}{48G_s^{(-1)}(-1)}- \frac{5\left[ G_s^{(-1)}\right]'(-1)}{12G_s^{(-1)}(-1)^2}\right] &
\mathcal{D}(\infty)^2 \left[ \frac{19 +96\mathcal{A}(1 + \mathcal{A})}{48G_s^{(-1)}(-1)} + \frac{5\left[G_s^{(-1)}\right]'(-1)}{12G_s^{(-1)}(-1)^2}\right]\\
\frac{i}{\mathcal{D}(\infty)^{2}} \left[ \frac{19 -96\mathcal{A}(1-\mathcal{A})}{48G_s^{(-1)}(-1)} + \frac{5\left[ G_s^{(-1)}\right]'(-1)}{12G_s^{(-1)}(-1)^2}\right]&
-i\left[\frac{9 - 96 \mathcal{A}^2}{48G_s^{(1)}(1)}- \frac{5\left[G_s^{(-1)}\right]'(-1)}{12G_s^{(-1)}(-1)^2}\right]
\end{pmatrix},\\
& R_1^{(1)}(z)
= - \frac{2^{1/2}}{2N}\frac{1}{(1-z)^2} \frac{5}{48G_s^{(1)}(1)}
\begin{pmatrix}
1 & -i\mathcal{D}(\infty)^2 \\ -i \mathcal{D}(\infty)^{-2} & -1
\end{pmatrix}\\
&  -\frac{2^{1/2}}{8N}\frac{1}{1-z} \begin{pmatrix}
\frac{9 - 96 \mathcal{A}^2}{48G_s^{(1)}(1)}+ \frac{5\left[ G_s^{(1)}\right]'(1)}{12G_s^{(1)}(1)^2} &
i \mathcal{D}(\infty)^2 \left[ \frac{19 +96\mathcal{A} + 96 \mathcal{A}^2}{48G_s^{(1)}(1)} - \frac{5\left[ G_s^{(1)}\right]'(1)}{12G_s^{(1)}(1)^2}\right]\\
i \mathcal{D}(\infty)^{-2} \left[ \frac{19 -96\mathcal{A} + 96 \mathcal{A}^2}{48G_s^{(1)}(1)} - \frac{5\left[ G_s^{(1)}\right]'(1)}{12G_s^{(1)}(1)^2}\right]&
-\frac{9 - 96 \mathcal{A}^2}{48G_s^{(1)}(1)}- \frac{5\left[ G_s^{(1)}\right]'(1)}{12G_s^{(1)}(1)^2}
\end{pmatrix},
\end{align*}

\noindent where the functions $G_s^{(\pm 1)}(z)$ come from
\begin{align*}
\xi_s^{(-1)}(z) = e^{-i\pi} N^{2/3} G_s^{(-1)}(z)^{2/3} (z+1),
\qquad \xi_s^{(1)}(z) = N^{2/3} G_s^{(1)}(z)^{2/3} (z-1),
\end{align*}

\noindent (see Appendix \ref{app:locale}). $\Jcal^{(x_j)}(z)$ may now be obtained by direct calculation.
\end{proof}

\section{Uniformity of the asymptotics in Theorem \ref{th:cf}}\label{app:cfcheck}

In this appendix we will give a brief outline of how to check that the asymptotics in Theorem \ref{th:cf} are still uniform when we replace $V$ by $V_{x,y}$ when $x,y\in(-1+\epsilon,1-\epsilon)$ (in the notation of Section \ref{sec:mainproof}). We will not try to be self contained here and we will use notations both from \cite{cf} and ones we've adopted earlier in this article. We won't provide all of the relevant definitions from \cite{cf}. We will simply try to provide a map of how to go over the argument.

Let us write $u=(x-y)^2/4\geq 0$ (which in the notation of \cite{cf} is $t$) and $v=(x+y)/2\in(-1+\epsilon,1-\epsilon)$, where $\epsilon$ is determined by the support of our non-negative test function. We also write $V_v(z)=V(z+v)$. In the notation of Section \ref{sec:mainproof}, we are interested in the asymptotics of $D_{N-1}(f_u;V_v)$, which in the notation of \cite{cf} would be $\widehat{Z}_N(u,\beta,V_v)/N!$. Note that in the notation of \cite{cf}, $\beta$ is replaced by $\alpha$.

Let us write $Y$ for the solution of the RHP related to $D_{N-1}(f_u;V_v)$.  $Y$ depends on $u$ and $v$, but as usual, we suppress this dependence in our notation. Then as the "center of mass" and "relative motion" coordinates decouple, or $\partial_u V_v=0$ for all $u$ and $v$, the proof of \cite[Proposition 4.1]{cf} carries through word to word and one finds

\begin{equation}\label{eq:cfdi}
\partial_u \log D_{N-1}(f_u;V_v)=-\frac{\beta}{2\sqrt{u}}\left[(Y(\sqrt{u})^{-1}Y'(\sqrt{u}))_{22}-(Y(-\sqrt{u})^{-1}Y'(-\sqrt{u}))_{22}\right].
\end{equation}

\noindent The goal will be to integrate this from zero to some positive $u$. Even though $\pm \sqrt{u}$ lie on the jump contour of $Y$, this quantity in fact does not have a jump so the notation is justified. Moreover one can calculate the relevant quantities at a point $z$ and then let $z\to \pm \sqrt{u}$ -- in particular the point $z$ can be taken to be outside of the relevant lenses and for simplicity in the lower half plane (see \cite[Figure 8]{cf}).  In \cite[Section 6]{cf}, using results of \cite{ck}, it is argued that near the points $\pm \sqrt{u}$, but outside of the lenses, one can write

\begin{equation}\label{eq:cfyrep}
Y(z)=e^{-N\frac{\ell_{v}}{2}\sigma_3}\left(R_v(z)E_v(z)\Psi^{(2)}(\lambda_v(z);s_{N,u})W_v(z)\right)e^{Ng_v(z)\sigma_3}e^{\frac{N\ell_{v}}{2}\sigma_3},
\end{equation}

\noindent where $\ell_v$ and $g_v$ refer to the $\ell$- and $g$-quantities constructed from the potential $V_v$. If we restrict to points $z$ outside of the lenses and in the lower half plane, then one has

\begin{equation}\label{eq:wcf}
W_v(z)=\left[(z^2-u)^{-\beta/2}e^{\frac{-\pi i \beta}{2}}e^{N\phi_v(z)}\right]^{\sigma_3},
\end{equation}

\noindent where (see the discussion around \cite[equation (4.13)]{cf} for details about the branch and integration contour -- note that in the notation of \cite{cf}, $d$ is $h$ and the support of the equilibrium measure is $[a,b]$ instead of our $[-1-v,1-v]$)

\begin{equation}\label{eq:phicf}
\phi_v(z)=\pi\int_{1-v}^z d_v(\xi)((\xi+v)+1)((\xi+v)-1))^{1/2}d\xi.
\end{equation}

$\lambda_v$ is a coordinate change which for $z$ in the lower half plane is defined by (see \cite[equation (6.2)]{cf})

\begin{equation}\label{eq:lambdacf}
\lambda_v(z)=-iN\left(-\phi_v(z)-\frac{\phi_{v,+}(\sqrt{u})+\phi_{v,+}(-\sqrt{u})}{2}\right).
\end{equation}

The main reason the uniformity of the asymptotics holds is that varying $v\in(-1+\epsilon,1-\epsilon)$ does not change the qualitative behavior of the asymptotics of $\lambda_v(z)$. If one were to allow $v=\pm 1$, then the situation would be different.

For the definition of $\Psi^{(2)}(\lambda,s)$, we refer to \cite[Section 3]{cf}, but point out here that while it depends on $\beta$, it does not depend on $x,y,$ or $V$. The function $E_v$ is analytic in a neighborhood of zero (containing the points $\pm \sqrt{u}$) and for the values of $z$ we are interested in, it can be written as (see \cite[Section 6.4]{cf})

\begin{equation}\label{eq:ecf}
E_v(z)=\mathcal{N}_v(z)W_v(z)^{-1}e^{-i\lambda_v(z)\sigma_3}=\mathcal{N}_v(z)\left[(z^2-u)^{\beta/2}e^{\pi i \beta/2}\right]^{\sigma_3} e^{\frac{N}{2}\left(\phi_{v,+}(\sqrt{u})+\phi_{v,+}(-\sqrt{u})\right)\sigma_3},
\end{equation}

\noindent where $\mathcal{N}_v(z)$ is the global parametrix which is of similar form as the one we consider in Section \ref{sec:global} apart from the support of the equilibrium measure now being $[-1-v,1-v]$ which changes the formulas slightly. See also around \cite[equations (5.5) and (6.1)]{cf} for details. In particular, as $z\to \pm \sqrt{u}$ for a fixed $N$, $\mathcal{N}_v(z)\sim(z\mp \sqrt{u})^{-\frac{\beta}{2}\sigma_3}$ uniformly in $v$. This combined with the fact that $\phi_{v,+}(\pm \sqrt{u})$ is purely imaginary implies that in a neighborhood of the origin, $E_v$, $E_v^{-1}$, and $E_v'$ are bounded uniformly in $v\in(-1+\epsilon,1-\epsilon)$.

Finally $R_v$ is a solution to a small norm RHP. As pointed out in \cite{cf}, the analysis of $R_v$ and its RHP is essentially carried out in \cite{ck}. While verifying in full detail the asymptotic behavior of $R_v$ is not something we will do, we will briefly sketch part of the argument, namely uniform asymptotics for the jump matrix across part of the boundary of a neighborhood of the origin. Analyzing the jump matrix of $R$ in the remaining part of the contour is similar and with a standard argument one finds that $R$ is uniformly close to the identity and its derivative is uniformly small.

From the definition of $R_v$ in \cite[Section 6.5]{cf} we see for $z$ on the boundary of some neighborhood of the origin containing the points $\pm \sqrt{u}$

\begin{equation}\label{eq:cfrjump}
R_{v,+}(z)=R_{v,-}(z) E_v(z)\Psi^{(2)}(\lambda_v(z);s_{N,u})W_v(z)\mathcal{N}_v(z)^{-1}
\end{equation}

Following the notation in \cite[Section 3]{cf}, we note that we can write

\begin{equation}\label{eq:cfpsi2}
\Psi^{(2)}(\lambda,s)=\Psi_{CK}\left(-\frac{4\lambda}{|s|}i;s\right)\chi(\lambda),
\end{equation}

\noindent where $\Psi_{CK}$ is the solution to the RHP in \cite[Section 3]{ck} and $\chi(\lambda)$ is defined in \cite[(3.12)]{cf}. We note that as $u$ is always small for us, $|\lambda_v(z)/|s||\sim u^{-1/2}$ is large if $z$ is at a fixed distance from $\pm \sqrt{u}$. We thus want to know the $\lambda\to\infty$ asymptotics of $\Psi_{CK}(\lambda,s)$ for all values of $s$. This was studied in \cite{ck}. For the relevant asymptotics for $\Psi_{CK}(\zeta;s)$, we refer to the discussion relevant to \cite[equations (3.6), (5.25), and (6.32)]{ck}.  For $\Psi^{(2)}(\lambda;s)$ these asymptotics translate into the following: for large $|\lambda|$

\begin{equation*}
\Psi^{(2)}(\lambda;s)=\begin{cases}
(I+\mathcal{O}(|s||\lambda|^{-1})) e^{i\lambda\sigma_3}, & s\to -i0^+\\
(I+\mathcal{O}(|\lambda|^{-1}))e^{i\lambda \sigma_3}, & s=\mathcal{O}(1)\\
(I+\mathcal{O}(|s \lambda|^{-1}))e^{i\lambda\sigma_3}, & s\to -i\infty
\end{cases}.
\end{equation*}

Using \eqref{eq:ecf} and fact that $E_v$ and $E_v^{-1}$ are uniformly bounded, we thus see that for all $u$ and uniformly in $v$, the jump matrix along this part of the jump contour is

\begin{equation*}
I+E_v(z)\mathcal{O}(\min(|s|,|s|^{-1})|\lambda_v(z)|^{-1})E_v(z)^{-1}=I+\mathcal{O}(\min(|s|,|s|^{-1})|\lambda_v(z)|^{-1}).
\end{equation*}

Going over such an argument in full detail would then imply that $R_v$ can be solved through the general small-norm approach and one has uniform asymptotics for $R_v$, e.g. $R_v(z)=I+\mathcal{O}(N^{-1})$ and $R_v'(z)=\mathcal{O}(N^{-1})$ uniformly in $z$ and $v\in(-1+\epsilon,1-\epsilon)$.

Let us now return to the differential identity \eqref{eq:cfdi}. With a basic matrix algebra argument, one finds from \eqref{eq:cfyrep} as in \cite[Section 5]{cf}

\begin{align*}
\left(Y^{-1}(z)Y'(z)\right)_{22}&=(B(z))_{22}-\frac{N}{2}V_v'(z)+\frac{\beta}{2}\left[\frac{1}{z-\sqrt{u}}+\frac{1}{z+\sqrt{u}}\right]\\
&\quad+\left[\Psi^{(2)}(\lambda_v(z);s_{N,u})\frac{d}{dz}\Psi^{(2)}(\lambda_v(z);s_{N,u})\right]_{22},
\end{align*}

\noindent where

\begin{equation*}
B(z)=\Psi^{(2)}(\lambda_v(z);s_{N,u})^{-1}(R_v(z)E_v(z))^{-1}(R_v(z)E_v(z))' \Psi^{(2)}(\lambda_v(z);s_{N,u}).
\end{equation*}

For the asymptotics of the $\frac{d}{dz}\Psi^{(2)}$-term, one can argue exactly like in \cite[Section 6.4]{cf} (see also \cite[equations (5.27) and (5.28); Lemma 5.3]{cf}) to find that as $z\to \pm \sqrt{u}$,

\begin{align*}
\left(\Psi^{(2)}(\lambda_v(z);s)^{-1}\frac{d}{dz}\Psi^{(2)}(\lambda_v(z);s)\right)_{22}&=\pm 2i\frac{\lambda_v'(\pm \sqrt{u})}{s_{N,u}}\left(\frac{\sigma_\beta(s_{N,u})-\frac{\beta^2}{2}}{\beta}+\frac{s_{N,u}}{2}\right)-\frac{\beta}{2}\frac{1}{z\mp \sqrt{u}}\\
&\quad +\mathcal{O}(1),
\end{align*}

\noindent where $\mathcal{O}(1)$ is uniform in $v$.

Thus what remains is the $B$-term. For this, by what we've argued about $R$ and $E$, we see that $(RE)^{-1}(RE)'=\mathcal{O}(1)$ uniformly in $v$ in a neighborhood of zero. Thus it is enough to show that as $z\to\pm \sqrt{u}$, $((\Psi^{(2)})^{-1}\mathcal{O}(1)\Psi^{(2)})_{22}=\mathcal{O}(1)$ uniformly in $v$. Here again the asymptotics of $\Psi^{(2)}$ come from \cite{ck}, and in fact the uniformity in $v$ follows from the argument for a fixed $v$ as in \cite[Section 5.6 and Section 6.6]{cf} and the uniform behavior of $\lambda_v$.

\end{document}